\documentclass{article}

\usepackage{arxiv}

\usepackage[utf8]{inputenc} 
\usepackage[T1]{fontenc}    
\usepackage{hyperref}       
\usepackage{url}            
\usepackage{booktabs}       
\usepackage{amsfonts}       
\usepackage{nicefrac}       
\usepackage{microtype}      
\usepackage{lipsum}
\usepackage{times}
\usepackage[ruled]{algorithm2e}
\usepackage{multicol}
\usepackage{amsmath,amssymb,amsthm}
\usepackage{comment}
\usepackage{dsfont}
\usepackage{xpatch}
\usepackage{multicol}
\usepackage{graphicx}
\usepackage{tikz}
\usetikzlibrary{patterns,arrows,decorations.pathreplacing}

\xpatchcmd{\proof}{\itshape}{\normalfont\proofnamefont}{}{}
\newcommand{\proofnamefont}{\bfseries}

\newtheorem{proposition}{Proposition}
\newtheorem{corollary}[proposition]{Corollary}
\newtheorem{lemma}[proposition]{Lemma}
\newtheorem{remark}{Remark}
\newtheorem{example}{Example}

\newtheorem{theorem}[proposition]{Theorem}
\newtheorem{definition}{Definition}

\def\1{~\mbox{I\hspace{-.6em}1}} 
\title{Extremes for stationary regularly varying random fields over arbitrary index sets}

\author{ Riccardo Passeggeri\\
{\tt riccardo.passeggeri14@imperial.ac.uk}\\
Department of Mathematics,  Imperial College London,  \\
80 Queen's Gate, London,\\
SW7 2AZ, UK\\[1cm]
{\bf Olivier Wintenberger}\\
{\tt olivier.wintenberger@sorbonne-universite.fr}\\
LPSM, Sorbonne Universit\'{e}, \\
4 place Jussieu,\\ 
Paris, 75005, France}

\begin{document}

\maketitle


\abstract{We consider the clustering of extremes for stationary regularly varying random fields over arbitrary growing index sets. We study sufficient assumptions on the index set such that the limit of the point random fields of the exceedances above a high threshold exists. Under the so-called anti-clustering condition, the extremal dependence is only local. Thus the index set can have a general form compared to previous literature \cite{BP,Stehr}. However, we cannot describe the clustering of extreme values in terms of the usual spectral tail measure \cite{SW} except for hyperrectangles or index sets in the lattice case. Using the recent extension of the spectral measure for star-shaped equipped space \cite{SZM}, the $\upsilon$-spectral tail measure provides a natural extension that describes the clustering effect in full generality. }\\

{\bf Keywords:} {Extremes, Regular variation, Extremal index, Max-stable random field, Space-time models}

{\bf MSC Classification:} {60G70, 60G60, 62G32}
\maketitle
\section{Introduction}
Asymptotic results for extreme values of random fields have attracted much attention recently, see Samorodnistky and Wu \cite{SW}, Basrak and Planinic \cite{BP}, and Jakubowski and Soja-Kukieła \cite{JS}, to name a few. Extending the basic results of Basrak and Segers \cite{BS} in the context of time series, the newly developed approaches focus on stationary regularly varying $\mathbb{R}^{d}$-valued random fields $\mathbf{X}=( \mathbf{X}_{\mathbf{t}})_{\mathbf{t}\in\mathbb{Z}^{k}}$: The random vectors $(\mathbf{X}_{\mathbf{t}_{1}},....,\mathbf{X}_{\mathbf{t}_{n}})$ are regularly varying in $\mathbb{R}^{nd}$ for each $\mathbf{t}_{1},...,\mathbf{t}_{n}\in\mathbb{Z}^{k}$. The existence of the spectral tail random field $\mathbf\Theta:=(\mathbf{\Theta}_{\mathbf{t}})_{\mathbf{t}\in\mathbb{Z}^{k}}$ characterizes the limit behavior of the extremes around the origin $\{\bf 0\}$ under the condition that $\bf X_{\bf0}$ is extreme and normalized by $|\bf X_{\bf0}|$. The random field $\mathbf \Theta$ characterizes the extrema of $\mathbf{X} $ and hence any asymptotic extreme value set. One phenomenon is the clustering of extrema, i.e., the tendency for extrema to occur locally. To formalize this phenomenon, the approach is to extend the basic result of Davis and Hsing \cite{DHs} via the convergence of the point process of exceedances. More precisely, we define $N_n$ as a simple point random field of exceedances on a hyperrectangle $C_n=[1,n]^k$, $n\ge 1$
	\begin{equation*}
	N_{n}:=\sum_{\mathbf{t}\in C_n}\varepsilon_{a_{n^k}^{-1}\mathbf{X}_{\mathbf{t}}}\,.
	\end{equation*}
The level of excesses is set to $a_n$, which satisfies $\lim_{n\to \infty}n\mathbb P(|\mathbf{X}_{\mathbf{0}}| > a_n)=1$, as if the observations were independent.
Samorodnistky and Wu \cite{SW} show that $N_n$ converges to a cluster point random field $N$ on $\mathbb R^d\setminus\{0\}$ and an explicit representation of the latter is given in Basrak and Planinic \cite{BP}. More precisely, there is a spectral cluster field $\mathbf{Q}:=(\mathbf{Q}_{\mathbf{t}})_{\mathbf{t}\in\mathbb{Z}^k}$ whose distribution is derived from that of $(\mathbf{\Theta}_{\mathbf{t}})$ and for which holds
\begin{equation*}
N^{\Lambda}=\sum_{i=1}^{\infty} \sum_{\mathbf{t}\in\mathbb{Z}^k}\varepsilon_{\Gamma_{i}^{-1/\alpha} \mathbf{Q}_{ i,\mathbf{t}}}\,,
\end{equation*}
where $ (\sum_{\mathbf{t}\in\mathbb{Z}^k}\varepsilon_{ \mathbf{Q}_{ i,\mathbf{t}})_{i\ge 1}})$ are independent and identically distributed (iid) copies of $\sum_{\mathbf{t}\in\mathbb{Z}^k}\varepsilon_{ \mathbf{Q}_{ \mathbf{t}}}$, independent of the points $(\Gamma_{i})$ of a standard Poisson process.
The random field $\mathbf{Q}$ is crucial since it accurately describes the asymptotic clustering phenomenon. The paper aims to introduce and analyze a new setting adapted to index sets other than the hyperrectangle $C_n$. 

Let us consider $\Lambda_n$ as an arbitrary index set of $\mathbb Z^k\setminus\{\mathbf 0\}$. Such an extension of the rectangular index set is not straightforward, as Stehr and R\'onn-Nielsen \cite{Stehr} showed in the asymptotically independent case ($\mathbf \Theta_{\mathbf t}=0$ for all ${\mathbf t}\ne \mathbf 0$). For index sets $(\Lambda_n)$, a geometric condition must hold. To motivate the study of index sets that are not rectangular, let us describe the most common index sets in the literature. The spatio-temporal sets $\Lambda_n$ are typically of the form $\mathcal C\times \{T,\ldots,mT\}$, where $\mathcal C$ is a fixed lattice of $\mathbb{Z}^{2}$ and $T\ge 1$ is the observation period through time expressed in space-time units, and $m$ is the number of observation periods. As usual, we consider a stationary, regularly varying random field $(\mathbf{X}_{\mathbf t})$. Remark that the assumption of stationarity in time and space on $\mathbf X$ is standard in environmental statistics, even when the spatial grid $\mathcal C$ is large but finite, see the review paper by Davison, Padoan and Ribatet \cite{DPR} and references therein. We first obtain the existence of a limiting point random field $N^ \Lambda$ as follows 
	\begin{equation*}
	N^{\Lambda}_{n}:=\sum_{\mathbf{t}\in \Lambda_{n}}\varepsilon_{{a_{n}^\Lambda}^{-1}\mathbf{X}_{\mathbf{t}}}\stackrel{d}{\to}N^{\Lambda}\,,\qquad n\to\infty\,,
	\end{equation*}
where $a_{n}^\Lambda$ comes from $\lim\limits_{n\to\infty}|\Lambda_n|\mathbb{P}(|\mathbf{X}_{\mathbf{0}}| > a_{n}^\Lambda)=1$.

In contrast to the hyperrectangle case, the distribution of the edge point random field $N^ \Lambda$ depends on the asymptotic lattice properties of the general lattice $\Lambda_n$. Not surprisingly, the asymptotic form of the index set $\Lambda_n$ constrains the clustering effect. We derive the limiting distribution of the point random field under a sufficient condition that ensures that $\Lambda_n$ consists asymptotically of translated versions of countably many fixed sets $\mathcal D_j$. It appears that the limiting cluster point random field distribution is a mixture of expressions of the spectral tail random field over the different $\mathcal D_j$. However, a representation of the clusters using the original spectral tail random field $(\mathbf{\Theta}_{\mathbf{t}})$ is limited to specific  $\Lambda_n$. We derive the representation of the limiting points similar as in Basrak and Planinic \cite{BP} only when all the $\mathcal D_j$'s are lattice. This condition is satisfied for $C_n$, and we recover the characterization of the clusters first provided in Basrak and Planinic \cite{BP}.

For irregular index set $\Lambda_n$, the representation of the (asymptotic) clusters does not naturally use the original spectral tail random field $(\mathbf{\Theta}_{\mathbf{t}})$. Instead, one has to introduce the concept of the $\Upsilon$-tail field that characterizes the limiting behaviour of the extremes around the region $\Upsilon$ given that $\bf X_{\bf t}$ is extreme over $\Upsilon$ and normalized by a modulus of $({\bf X}_{\bf t})_{{\bf t}\in \Upsilon}$. This framework has already been developed for iid sequences by Ferreira and de Haan \cite{ferreira} and for time series cases by Segers et al. \cite{SZM} but not for random fields.

Our main contribution is to introduce a very general setting for index sets, namely Condition $(\mathcal{D}^{\Lambda})$, which does not involve any topological properties. This condition allows for countably many different shapes, and it is only an asymptotic condition. In particular, it implies that some shapes repeat approximately an infinite number of times proportional to $|\Lambda_n|$. Surprisingly, the shape of the asymptotic local region $\Upsilon$ is arbitrary. In contrast with existing results such as Stehr and R\'onn-Nielsen \cite{Stehr} we show that convexity is not required when dealing with extremes. Under the anti-clustering condition, we deal with index sets that are local regions such as $\Upsilon$ reproduced over a lattice. 

The rest of the paper is organized as follows. In Section \ref{sec:prel} are collected preliminaries, notation and main assumptions. That the crucial Condition $(\mathcal{D}^{\Lambda})$ implies asymptotic local region reproduced over lattices is the main theoretical challenge of the paper, alleviated in Section \ref{sec:lattice}. The asymptotic clusters are studied in Section \ref{sec:spectral} for any index set $\Lambda_n$ satisfying Condition $(\mathcal{D}^{\Lambda})$. Their characterization is provided using the $\Upsilon-$spectral tail field in Section \ref{sec:Upsilon}. Two applications of this new approach are developed in Section \ref{sec:appl}, determining the extremal index and providing sufficient conditions for max-stable random fields. Section \ref{sec:proof} contains the proofs of the results of Section \ref{sec:lattice} and the rest of the proofs are collected in Section \ref{sec:proof2} and \ref{sec:proof3}.

\section{Preliminaries, notation and main assumptions}\label{sec:prel}
Let $(\mathbf{X}_{\mathbf{t}})_{\mathbf{t}\in\mathbb{Z}^{k}}$ be an $\mathbb{R}^{d}$-valued regularly varying stationary random field.

\subsection{Spectral tail fields}
Let us recall two fundamental results of Samorodnitsky and Wu  \cite{SW}: the existence of the tail field and the time change formula for the tail and spectral tail fields.
\begin{theorem}
	[Theorem 2.1 in \cite{SW}] An $\mathbb{R}^{d}$-valued stationary random field $(\mathbf{X}_{\mathbf{t}})_{\mathbf{t}\in\mathbb{Z}^{k}}$ is jointly regularly varying with index $\alpha$ if and only if there exists a random field $(\mathbf{Y}_{\mathbf{t}})_{\mathbf{t}\in\mathbb{Z}^{k}}$ such that
	\begin{equation*}
	\mathcal{L}\big(x^{-1}\mathbf{X}_{\mathbf{t}}:\mathbf{t}\in\mathbb{Z}^{k}\big||\mathbf{X}_{\mathbf{0}}|>x\big)\stackrel{fdd}{\to}\mathcal{L}(\mathbf{Y}_{\mathbf{t}}:\mathbf{t}\in\mathbb{Z}^{k})
	\end{equation*}
	as $x\to\infty$, and $\mathbb{P}(|\mathbf{Y}_{\mathbf{0}}|>y)=y^{-\alpha}$ for $y\geq1$. We call $(\mathbf{Y})_{\mathbf{t}\in\mathbb{Z}^{k}}$ the \textnormal{tail field} of $(\mathbf{X}_{\mathbf{t}})_{\mathbf{t}\in\mathbb{Z}^{k}}$.
\end{theorem}
\begin{theorem}
	[Theorem 3.2 in \cite{SW}] Let $(\mathbf{Y})_{\mathbf{t}\in\mathbb{Z}^{k}}$ be the tail field corresponding to an $\mathbb{R}^{d}$-valued stationary random field $(\mathbf{X}_{\mathbf{t}})_{\mathbf{t}\in\mathbb{Z}^{k}}$ that is jointly regularly varying with index $\alpha$ and define $\mathbf{\Theta}_{\mathbf{t}}=\mathbf{Y}_{\mathbf{t}}/|\mathbf{Y}_{\mathbf{0}}|$, $\mathbf{t}\in\mathbb{Z}^{k}$. Let $g:(\mathbb{R}^{d})^{\mathbb{Z}^{k}}\to\mathbb{R}$ be a bounded measurable function. Take any $\mathbf{s}\in\mathbb{Z}^{k}$. Then the following identities hold:
	\begin{equation}\label{timechangeY}
	\mathbb{E}[g(\mathbf{Y}_{\mathbf{t} - \mathbf{s}})\mathbf{1}(\mathbf{Y}_{-\mathbf{s}}\neq\mathbf{0})]=\int_{0}^{\infty}\mathbb{E}[g(r\mathbf{\Theta}_{\mathbf{t}})\mathbf{1}(r|\mathbf{\Theta}_{\mathbf{s}}|>1)]d(-r^{-\alpha}),
	\end{equation}
	\begin{equation}\label{timechangeTheta}
	\mathbb{E}[g(\mathbf{\Theta}_{\mathbf{t}  - \mathbf{s}})\mathbf{1}(\mathbf{\Theta}_{-\mathbf{s}}\neq\mathbf{0})]=\mathbb{E}\bigg[g\bigg(\frac{\mathbf{\Theta}_{\mathbf{t} }}{|\mathbf{\Theta}_{\mathbf{s}}|}\bigg)|\mathbf{\Theta}_{\mathbf{s}}|^{\alpha}  \bigg].
	\end{equation}
	We call $(\mathbf{\Theta}_{\mathbf{t}})_{\mathbf{t}\in\mathbb{Z}^{k}}$ the \textnormal{spectral field} of $(\mathbf{X}_{\mathbf{t}})_{\mathbf{t}\in\mathbb{Z}^{k}}$.
\end{theorem}

Denote by $\leq$ the component-wise order on $\mathbb{Z}^{k}$, thus for $\mathbf{i} = (i_{1}, . . . , i_{k})$, $\mathbf{j} = (j_{1}, . . . , j_{k})$ in $\mathbb{Z}^{k}$, $i\leq j$ if $i_{l}\leq j_{l}$ for all $l = 1, . . . , k$.

We consider a complete order $\prec$ on $\mathbb{Z}^{k}$ that is invariant: if $\mathbf{s} \prec \mathbf{t}$ for $\mathbf{s}, \mathbf{t} \in \mathbb{Z}^{k}$ implies that $\mathbf{s} + \mathbf{i} \prec \mathbf{t} + \mathbf{i}$ for any $\mathbf{i}\in\mathbb{Z}^{k}$. An example of an invariant order is the lexicographic (or dictionary) order: for $\mathbf{s}, \mathbf{t} \in \mathbb{Z}^{k}$, we say that $\mathbf{s} \prec \mathbf{t}$ if either (1) $s_{1}< t_{1}$, or (2) there exists $2 \leq j \leq k$ such that $s_{i} = t_{i}$ for all $i = 1, . . . , j - 1$, and $s_{j} < t_{j}$.

\subsection{Condition ($\mathcal{D}^{\Lambda} $) on the index set}

Consider the following simple point random field:
	\begin{equation*}
	N^{\Lambda}_{n}:=\sum_{\mathbf{t}\in \Lambda_{n}}\varepsilon_{{a_{n}^\Lambda}^{-1}\mathbf{X}_{\mathbf{t}}},
	\end{equation*}
	where the sequence $(a_{n}^\Lambda)$ satisfies $\lim\limits_{n\to\infty}|\Lambda_n|\mathbb{P}(|\mathbf{X}_{\mathbf{0}}|>a_{n}^\Lambda)=1$ and $\Lambda_{n}$ is any subset of $\mathbb{Z}^{k}$ such that $|\Lambda_{n}| \to \infty$ as $n\to\infty$. For any set $\Upsilon\subset\mathbb{Z}^{k}$, $c>0$, and $\mathbf{t}\in\mathbb{Z}^{k}$, let $(\Upsilon)^+:=\{\mathbf{u}\in\Upsilon:\mathbf{u}\succ \mathbf{0}\}$, $(\Upsilon)_{-\mathbf{t}}:=\{\mathbf{u}\in\mathbb{Z}^{k}:\mathbf{u}=\mathbf{s}-\mathbf{t},\mathbf{s}\in\Upsilon\}$ and   $\Upsilon^{(\mathbf{t},c)}:=((\Upsilon)_{-\mathbf{t}}\cap K_{c})^+$ where the hypercube $K_c$ is defined as 
	$
	K_c=[-c,c]^k\cap \mathbb{Z}^{k}
	$, $c\ge 0$.
	Through the paper we assume that $\Lambda_n$ satisfies the following condition.\\
\textbf{Condition} ($\mathcal{D}^{\Lambda} $): \textit{There exist (possibly countably many) different subsets of $\{\mathbf{t}\in\mathbb{Z}^{k}:\mathbf{t}\succ\mathbf{0}\}$, which we denote by $\mathcal{D}_{1},\mathcal{D}_{2},...$, s.t.~
$$
\lim\limits_{n\to\infty}\dfrac{|\{\mathbf{t}\in\Lambda_{n}:\Lambda_{n}^{(\mathbf{t},p)}=\mathcal{D}_{i}\cap K_{p}\}|}{|\Lambda_n|}=:\lambda_{i,p}\to \lambda_{i}\,, \qquad p\to\infty\,,
$$ 
with $\lambda_{i}>0$ and $\sum_{i=1}^{q}\lambda_{i}=1$, where $q\in\mathbb{N}\cup\{\infty\}$ is the number of these $\mathcal{D}$s.}\\
	
	Condition ($\mathcal{D}^{\Lambda}$) says the following. Consider a point $\mathbf{t}$ in $\Lambda_{n}$. Translate the set $\Lambda_{n}$ by $-\mathbf{t}$ so that $\mathbf{t}$ is now at $\mathbf{0}$. Take an hypercube around $\mathbf{0}$ of side $2p$, for $p$ large enough, and intersect it with the positive points according to $\succ$ and with the translated set. Thus, we have obtained $\Lambda_{n}^{(\mathbf{t},p)}$. Now, it might happen that the same set $\Lambda_{n}^{(\mathbf{t},p)}$ is exactly the same for other points in $\Lambda_{n}$, and also for other points in $\Lambda_{m}$ with $m>n$. Condition ($\mathcal{D}^{\Lambda}$) imposes that there are different sets (denoted $\mathcal{D}_{i}\cap K_{p}$, $i\in\{1,...,q\}$) such that the number of points $\mathbf{t}\in\Lambda_{n}$ for which $\Lambda_{n}^{(\mathbf{t},p)}$ is equal to one of these sets, divided by $|\Lambda_n|$, has a limit and these limits form a weighted sum. 
	
	This condition provides a minimum requirement to have (at least asymptotically) a structure for studying the long-time clustering behaviour of extremes.

\begin{example}\label{Example-1}
    Imagine observing precipitations over a specific geographical area $\cal C$. There are stations spread throughout the geographical region that measure the precipitation. Let $\Lambda_n$ lying in $\mathbb{Z}^d$ where $d$ is the sum of the time dimension and the space dimension (thus $d=3$ or $d=4$ depending on whether we consider the geographical region $\cal C$ to lie in $\mathbb{Z}^2$ or $\mathbb{Z}^3$, respectively). In the time direction, each point is the number of rain rained in a certain amount of time, while the space direction indicates the location where this is measured. Thus, $\mathbf{X}_\mathbf{t}$ where $\mathbf{t}\in\Lambda_n$ corresponds to the amount of rain measured in a certain period in a specific location. Assume that the measurements over $\cal C$ repeat in a constant frequency (\textit{e.g.}~at every week). This assumption corresponds to condition ($\mathcal{D}^\Lambda$), where we take the order $\succ$ to be increasing with successive observations. In particular, imagine measuring over $\cal C$ infinitely many times. Denote this set by ${\cal C}_\infty$. Then, each $\mathcal{D}$ is  ${\cal C}_\infty $ centered at $\mathbf{0}$ (that is translated version of ${\cal C}_\infty$ by minus one of its points) and consider the points successive to $\mathbf{0}$. Notice that only  $q=|\cal C|$ distinct $\cal D$. \cite{BK} already considered similar index sets.
\end{example}
\begin{example}
    The framework of \cite{Stehr,Aarhus2} is a particular specification of our framework. Indeed, consider Assumption 1 in \cite{Stehr} (which is Assumption 3 in \cite{Aarhus2}): The sequence $(C_n)_{n\in\mathbb{N}}$ consists of $p$–convex bodies (\textit{i.e.} connected sets which are also unions of $p$ convex sets), where $C_n=\cup_{i=1}^p C_{n,i}$ and $|C_n|\to\infty$ as $n\to\infty$, and $\frac{\sum_{i=1}^p V_j(C_{n,i})}{|C_n|^{j/d}}$ is bounded for each $j=1,...,d-1$, where $V_j(C_{n,i})$ indicates the intrinsic volumes of the convex body $C_{n,i}$. Consider the two dimensional case, so $d=2$ -- similar arguments apply to other dimensions. For any convex body $C$, we have that $V_0(C)=1$ and $V_1(C)$ is equal to the perimeter of $C$ divided by $\pi$. Then, Assumption 1 in \cite{Stehr} states that the sum of the perimeters of the $C_{n,i}$s must not grow faster than the square root of the volume of $C_n$. There are cases where this is not true, like when one of the $C_{n,i}$s is a rectangle with edges increasing with different speed. In general, this assumption ensures that the $C_{n,i}$s must grow in all directions, implying that the number of points in $C_n$ away from the boundary divided by the number of points in $C_n$ tends to $1$ as $n\to\infty$. Formally it implies that, for any $r\in\mathbb{N}$,  $\frac{|\{\mathbf{t}\in C_{n}:C_{n}^{(\mathbf{t},r)}=\mathcal{D}\cap K_{r}\}|}{|C_n|}\to 1$ as $n\to\infty$, where $\mathcal{D}$ is simply given by $\{\mathbf{t}\in\mathbb{Z}^{k}:\mathbf{t}\succ\mathbf{0}\}$. Therefore, Assumption 1 in \cite{Stehr} is strictly stronger than condition ($\mathcal{D}^{\Lambda}$).
    
    It is important to explicitly look at the differences of our framework with the one of \cite{Stehr,Aarhus2}. First, Condition ($\mathcal{D}^{\Lambda}$) is only an asymptotic condition, thus the set $\Lambda_n$ (or $C_n$) does not need to satisfy any constraint for finite $n$. The lack of a non-asymptotic structure for $\Lambda_n$ is a challenge, and in particular for the proof of Theorem \ref{t1-L-E-Pro}. We overcome this by imposing structures that will be satisfied asymptotically. Another feature of our setting also exacerbates this issue: the possibility of having countably many different asymptotic sets (denoted by $\mathcal{D}$s). Indeed, having countably many sets does not allow distinguishing the points in $\Lambda_n$ that will eventually form an asymptotic set from other points in $\Lambda_n$, because this distinction happens only asymptotically. We overcome this by using that only finitely many of these sets have weights (denoted by $\lambda$s) greater than $\varepsilon$, for any $\varepsilon>0$. The third difference is the structure of the asymptotic sets. While in \cite{Stehr,Aarhus2} the only allowed asymptotic set is $\{\mathbf{t}\in\mathbb{Z}^{k}:\mathbf{t}\succ\mathbf{0}\}$ as just shown, in our framework any possible subset of $\{\mathbf{t}\in\mathbb{Z}^{k}:\mathbf{t}\succ\mathbf{0}\}$ is allowed. For example, we might have that $\Lambda_n$ is a rectangle where only one side increases.
\end{example}
We conclude this section by pointing out that condition ($\mathcal{D}^{\Lambda}$) comes from the proof of the main asymptotic results of the paper, and it is the most refined (\textit{i.e.}~weakest) condition we could attain. This condition is satisfied in all the previous settings (see \cite{DHs,SW,Stehr,Aarhus2,BK}).

\subsection{Mixing and anti-clustering conditions}
Following the seminal work of Davis and Hsing \cite{DHs} on stationary time series, we assume two complementary conditions. The anti-clustering condition avoids too strong clustering effects. The mixing condition approximates the Laplace functional of the point random field $N^{\Lambda}$ over $\Lambda_n$  in terms of products of Laplace functionals of copies of the point random field over a smaller index set. Such conditions were extended to random fields by Samorodnistky and Wu \cite{SW} for the specific index set $C_n=[1,n]^k$. Some care is required when considering the general index set $\Lambda_n$.

Take a sequence of positive integers $(r_{n})$ such that $\lim\limits_{n\to\infty}r_{n}=|\Lambda_n|/|\Lambda_{r_{n}}|=\infty$ and  let $k_{n}= \lfloor |\Lambda_n|/|\Lambda_{r_{n}}|\rfloor$. Let $R_{l,\Lambda_{n}}:=\big(\bigcup_{\mathbf{t}\in\Lambda_{n}}(\Lambda_{n})_{-\mathbf{t}}\big)^+\setminus K_{l}$ and let $\hat{M}^{\Lambda,|\mathbf{X}|}_{l,n}:=\max_{\mathbf{i}\in R_{l,\Lambda_{n}}}|\mathbf{X}_{\mathbf{i}}|$ and 
consider the following anti-clustering
\\
\textbf{Condition} (AC$^{\Lambda}_{\succeq}$): \textit{The $\mathbb{R}^{d}$-valued stationary regularly varying random field $(\mathbf{X}_{\mathbf{t}}:\mathbf{t}\in\mathbb{Z}^{k})$ satisfies the \textnormal{(AC$^{\Lambda}_{\succeq}$)} condition if there exists an integer sequence $r_{n}\to\infty$ and $k_{n}=|\Lambda_n|/|\Lambda_{r_n}|\to\infty$ such that}
\begin{equation*}
\lim\limits_{l\to\infty}\limsup_{n\to\infty}\mathbb{P}\Big(\hat{M}^{\Lambda,|\mathbf{X}|}_{l,r_{n}}>a_{n}^\Lambda x\,\big||\mathbf{X}_{\mathbf{0}}|>a_{n}^\Lambda x\Big)=0.
\end{equation*}
Let $d_{n}:=\max_{x,y\in\Lambda_{r_{n}}}\max_{j=1,..,k}|x^{(j)}-y^{(j)}|$, namely $d_{n}$ be the maximum distance between the points of $\Lambda_{r_{n}}$. Observe that
\begin{multline*}
\lim\limits_{l\to\infty}\limsup_{n\to\infty}\mathbb{P}\Big(\hat{M}^{|\mathbf{X}|}_{l,d_{n};\succeq}>a_{n}^\Lambda x\,\big||\mathbf{X}_{\mathbf{0}}|>a_{n}^\Lambda x\Big)=0\\
\Rightarrow\lim\limits_{l\to\infty}\limsup_{n\to\infty}\mathbb{P}\Big(\hat{M}^{\Lambda,|\mathbf{X}|}_{l,r_{n}}>a_{n}^\Lambda x\,\big||\mathbf{X}_{\mathbf{0}}|>a_{n}^\Lambda x\Big)=0
\end{multline*}
where $\hat{M}^{|\mathbf{X}|}_{a,b;\succeq}=\max_{a\leq|\mathbf{i}|\leq b,\, \mathbf{i}\succeq\mathbf{0} }|\mathbf{X}_{\mathbf{i}}|$ for $a,b\in\mathbb{Z}$ and $|\mathbf{i}|:=\max(|i_{1}|,..., |i_{k}|)$. This sufficient condition is often easier to check in practice and is implied by the anti-clustering condition considered in Samorodnistky and Wu \cite{SW} that required a stronger condition on the maxima over indices  $a\leq|\mathbf{i}|\leq b$ in any directions.

For the mixing condition, we require extra classical notation, namely 
		\begin{equation*}
		\tilde{N}^{\Lambda}_{r_{n}}:=\sum_{\mathbf{t}\in \Lambda_{r_{n}}}\varepsilon_{{a_{n}^\Lambda}^{-1}\mathbf{X}_{\mathbf{t}}}\,,
		\end{equation*} 
	and for any	$E\subset\mathbb{R}^{d}$,  $\mathbb{C}^{+}_{K}(E)$ the class of continuous non-negative functions $g$ on $E$. Further, let the Laplace functional of a point random field $\xi$ with points $(\mathbf{Y}_{\mathbf{i}})$ in the space $E\subset\mathbb{R}^{d}$ be denoted by
\begin{equation*}
\Psi_{\xi}(g):=\mathbb{E}\bigg[\exp\bigg(-\int_{E}gd\xi \bigg) \bigg]=\mathbb{E}\bigg[\exp\bigg(-\sum_{\mathbf{i}}g(\mathbf{Y}_{\mathbf{i}}) \bigg) \bigg],\quad g\in \mathbb{C}^{+}_{K}(E).
\end{equation*}
We adopt the notation $\mathbb{C}^{+}_{K}:=\mathbb{C}^{+}_{K}(\mathbb{R}^{d}\setminus\{\mathbf{0}\})$. 
\\ \textbf{Condition} $\mathcal{A}^{\Lambda}(a_{n}^\Lambda)$: \textit{Choose the integer sequences $r_{n}\to\infty$ and $k_{n}= |\Lambda_n|/|\Lambda_{r_n}|\to\infty$ from condition (AC$^{\Lambda}_{\succeq}$). The $\mathbb{R}^{d}$-valued stationary regularly varying random field $(\mathbf{X}_{\mathbf{t}}:\mathbf{t}\in\mathbb{Z}^{k})$ satisfies the condition $\mathcal{A}^{\Lambda}(a_{n}^\Lambda)$ if }
\begin{equation*}
\Psi_{N^{\Lambda}_{n}}(g)-(\Psi_{\tilde{N}^{\Lambda}_{r_{n}}}(g))^{k_{n}}\to0,\quad n\to\infty,\quad g\in \mathbb{C}^{+}_{K}.
\end{equation*}

\section{Lattice properties}\label{sec:lattice}
Before giving the main results, we need to investigate further the lattice properties of the index sets $\mathcal{D}_{j}$ appearing in Condition ($\mathcal{D}^{\Lambda}$) on $\Lambda_n$. We will distinguish two settings, lattice properties on the upper orthant and the whole index set. Condition ($\mathcal{D}^{\Lambda}$) implicitly involves the upper orthant and it would have been possible to focus on the whole index set by adapting Condition ($\mathcal{D}^{\Lambda}$) accordingly. This approach would have been entirely equivalent to ours. But notice that the two settings are crucial for our main results, and one cannot make the economy of one of them.

\subsection{Lattice properties on the upper orthant}
Recall that $\mathcal{D}_{1},\mathcal{D}_{2},...$ are the subsets of the upper-orthant $\big(\mathbb{Z}^{k}\big)^+$ that appear in Condition ($\mathcal{D}^{\Lambda} $).
\begin{proposition}\label{lem-AC1-L-2}
Let $\Lambda_{n}$ satisfy $|\Lambda_n|\to \infty$ as $n\to \infty$ together with Condition ($\mathcal{D}^{\Lambda}$).
\\\textnormal{(I)} For every $\mathcal{D}_{j}$ and $\mathcal{D}_{i}$ with $j\neq i$ there exists a $p$ large enough s.t.~$\mathcal{D}_{j}\cap K_{p}\neq \mathcal{D}_{i}\cap K_{p}$. Further, for every $\mathcal{D}_{j}$ and every $p\in\mathbb{N}$ we have the identity 
$$
\lim\limits_{n\to\infty}\dfrac{|\{\mathbf{t}\in\Lambda_{n}:\Lambda_{n}^{(\mathbf{t},p)}=\mathcal{D}_{j}\cap K_{p}\}|}{|\Lambda_n|}=\lambda_{j,p}= \sum_{i\in I_{p}^{(j)}}\lambda_{i},
$$ 
where $I^{(j)}_{p}:=\{i\in\{1,...,q\}:\mathcal{D}_{i}\cap K_{p}=\mathcal{D}_{j}\cap K_{p}\}$.
\\\textnormal{(II)} The empty set is a possible $\mathcal{D}$. 
\\\textnormal{(III)} For every $\mathcal{D}_{j}$, there exist $b_{j}$ many different $\mathcal{D}$s, where $b_{j}\in\mathbb{N}$ s.t.~$b_{j}\leq\lfloor1/\lambda_{j}\rfloor-1$, which we denote by $\mathcal{D}_{l_{1}},...,\mathcal{D}_{l_{b_{j}}}$ such that, for every $ \mathbf{z}\in \mathcal D_j$, $\mathcal D_{l_i}=((\mathcal D_j)_{-\mathbf{z}})^+$ for some $i=1,...,b_{j}$ and  then $\lambda_{i}\geq\lambda_{j}$.
\end{proposition}

Point (III) of Proposition \ref{lem-AC1-L-2} suggests that $\mathcal D_j$ contains shifted versions of potentially different $\mathcal D_\ell$. In order to exhibit the lattice property of $\mathcal D_j$, we define $\mathcal{G}_j$ as the set of the shifts that yields the same $\mathcal D_j$, namely
\begin{equation}\label{eq:gj}
\mathcal{G}_j :=\{\mathbf{z}\in \mathcal{D}_j\cup \{\mathbf{0}\}:  (\mathcal{D}_{j})_{-\mathbf{z}}^+ = \mathcal{D}_j \}\quad\text{and}\quad \mathcal{L}_j:=\mathcal{G}_j\cup-\mathcal{G}_j\,,\qquad 1\le j\le q\,.
\end{equation}
For every $1\le j\le q$, one can partition the set $\mathcal D_j$ using the lattice sets $\mathcal{L}_{l_i}$, $i=1,...,b_{j}$:
\begin{proposition}\label{prop:lattice1}
Let $\Lambda_{n}$ satisfy $|\Lambda_n|\to \infty$ as $n\to \infty$ together with Condition ($\mathcal{D}^{\Lambda}$). Fix $1\le j\le q$, then the set  $\mathcal{L}_{j}$ is a  lattice on $\mathbb{Z}^{k}$. For $i=1,...,b_j$, denoting $\mathbf{z}_{l_i}$ any point in $\mathcal{D}_j$ such that $\mathcal{D}_{l_i}=((\mathcal{D}_j)_{- \mathbf{z}_{l_i}})^+$ we have the partition
\begin{equation}\label{partition}
\mathcal{D}_{j} =\mathcal{L}_j^+\cup\bigcup_{i=1}^{b_j}((\mathcal L_{l_i})_{\mathbf{z}_{l_i}})^+.
\end{equation}
Further, for every $\mathcal{D}_{j}$, we have that $\mathcal{L}_{l_{i}}\supseteq\mathcal{L}_{j}$, and $\mathcal{L}_{l_{i}}$ and $\mathcal{L}_{j}$ have the same rank  for $i=1,...,b_{j}$. In particular, $\mathcal{D}_{j}$ is bounded if and only if $\mathcal{L}_{j}=\{\mathbf{0}\}$ and in this case $\mathcal{D}_{j} =\bigcup_{i=1}^{b_j}\{\mathbf{z}_{l_i}\}$.
\end{proposition}

Building on partition \eqref{partition} we want to exhibit some translation invariant properties of $\mathcal{D}_j$. Fix any $j=1,\ldots,q$ and denote $l_0=j$ for convenience, then any $i\in\{0,1,...,b_j\}$   satisfies the Translation Invariance Property \textnormal{(TIP$_j$)} if it has the following property:\\

\noindent \textbf{Translation Invariance Property \textnormal{(TIP$_j$)}:} {\it The index $i\in\{0,1,...,b_j\}$ satisfies \textnormal{(TIP$_j$)} if there is a point $\mathbf{x}\in((\mathcal{L}_{l_i})_{\mathbf{z}_{l_i}})^+$ such that $\mathbf{x}\prec\mathbf{y}$ for some $\mathbf{y}\in\mathcal{G}_j$.}\\

Further, we let $W_j$ denote the subset of $\{0,...,b_j\}$ satisfying  \textnormal{(TIP$_j$)}  and let $
\hat{\mathcal{D}}_{j}:=\bigcup_{h\in W_j}\mathcal{D}_{l_{h}}$.
\begin{proposition}\label{prop:lattice2}
	 Let $\Lambda_{n}$ satisfy $|\Lambda_n|\to \infty$ as $n\to \infty$ together with Condition ($\mathcal{D}^{\Lambda}$). Fix any $j=1,\ldots,q$. If $i\in\{0,1,...,b_j\}$ satisfies  \textnormal{(TIP$_j$)}   then $\mathcal{L}_{l_{i}}=\mathcal{L}_{j}$ and $\lambda_{l_{i}}=\lambda_{j}$. In particular, when $\mathcal L_j$ is a full rank lattice the \textnormal{(TIP$_j$)} condition is satisfied for all $i=0,...,b_j$ and when $\mathcal{D}_j$ is bounded the \textnormal{(TIP$_j$)} condition is never satisfied.
	
	Further, for every $i\in W_j$ we have  $\hat{\mathcal{D}}_{l_{i}}=\hat{\mathcal{D}}_{j}$ and $\hat{\mathcal{D}}_j\cup\{\mathbf{0}\}\cup-\hat{\mathcal{D}}_j$ is translation invariant for every point in $\mathcal{L}_{j}$.
\end{proposition}

\begin{remark}
	The case of (TIP$_j$) not holding for some $l=l_1,...,l_{b_j}$ is equivalent to the case of $\mathcal{G}_j$ lying on the hyperplane determined by the order $\succ$. For example, this is the case when we are in $\mathbb{R}^2$, the order goes along the horizontal lines (informally $(0,0)\prec(0,1)\prec(0,2)\prec...\prec(0,\infty)\prec(1,-\infty)\prec...\prec(1,0)\prec...$), and $\Lambda_n$ draws two lines which are parallel to the horizontal axis, see Figure \ref{fig:TIPfail} for an illustration. It is possible to see that in this case one $\hat{\mathcal{D}}$ (say $\hat{\mathcal{D}}_1$) is is simply given by $x$-axis, while for the other ($\hat{\mathcal{D}}_2$) we have the set provided in Figure \ref{fig:TIPfail-2}. These sets are translation invariant with respect to the points in the respective $\mathcal{L}_i$ and in this example $\mathcal{L}_1$ and $\mathcal{L}_2$ are both equal to $x$-axis.
	\begin{figure}[h]
\centerline{
\begin{tikzpicture}[line cap=round,line join=round,>=triangle 45,x=0.4091370558375635cm,y=0.9589041095890412cm]
\draw[->,color=black] (-1.3,0) -- (20,0);
\draw[->,color=black] (0,-0.7) -- (0,2);
\draw[color=black] (19.7,-0.7) node [anchor=south west] {x};
\draw[color=black] (-1,1.3) node [anchor=south west] {y};
\draw[color=black] (-0.3,-0.3) node {0};
\draw[color=red] (1,0.05) --  (22,0.05);
\draw[color=red] (-2,1) --  (22,1);
\foreach \x in {-1,1,2,3,4,5,6,7,8,9,10,11,12,13,14,15,16,17,18,19}
\draw[shift={(\x,0)},color=black]  (0,-0.1) --  (0,0.1);
\foreach \x in {-1,1,2,3,4,5,6,7,8,9,10,11,12,13,14,15,16,17,18,19}
\draw[shift={(\x,1)},color=black]  (0,-0.1) --  (0,0.1);
\draw[color=red] (22,-0.4) node {${\cal D}_1$};
\draw[color=red] (22,0.6) node {${\cal D}_2$};
\pattern[pattern=north west lines, pattern color=red] (1,-0.1)--(1,0.1)--(22,0.1)--(22,-0.1)--cycle;
\filldraw[black,color=red] (0,0) circle (2pt);
\draw[densely dotted,color=blue] (0,0) -- (1,1);
\draw[densely dotted,color=blue] (0,0) -- (-1,1);
\draw[densely dotted,color=blue] (0,0) -- (2,1);
\draw[densely dotted,color=blue] (0,0) -- (3,1);
\draw[densely dotted,color=blue] (0,0) -- (0,1);
\draw[densely dotted,color=blue] (0,0) -- (4,1);
\draw[color=blue] (3,0.3) node {${\cal E}_1$};
\end{tikzpicture}}
\caption{\label{fig:TIPfail} We consider an order that increases along the horizontal axis and then upward. Red chopped half-line starting at $(1,0)$ is ${\cal D}_1$, ${\cal D}_1$ plus the red line above is ${\cal D}_2$ also in red. Both ${\cal D}$s have the same ${\cal G}={\cal D}_1\cup \{(0,0)\}$ thus the same ${\cal L}$ which coincides with the x-axis. We check that ${\cal D}_2$ is partitioned into ${\cal L}_+$ and ${\cal L_{\mathbf z}}$ where $\mathbf z$ is any point with $1$ as the second coordinate. However, none of the points of ${\cal L_{\mathbf z}}$ precedes any point in ${\cal G}$, and the TIP property fails. Notice ${\cal E}_1$ is any couple of points $\{(0,0),{\mathbf z}\}$ in blue.}
\end{figure}
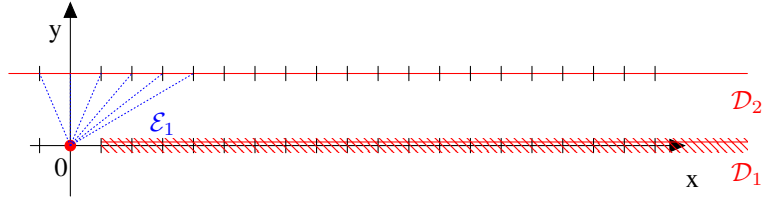
\end{remark}

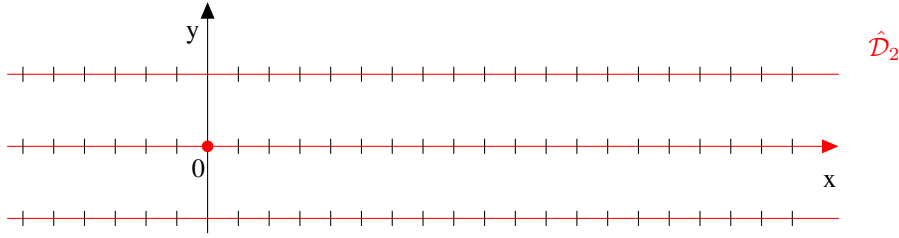
\begin{figure}[h]
\centerline{
\begin{tikzpicture}[line cap=round,line join=round,>=triangle 45,x=0.4091370558375635cm,y=0.9589041095890412cm]
\draw[->,color=red] (-6.5,0) -- (20.5,0);
\draw[->,color=black] (0,-1.2) -- (0,2);
\draw[color=black] (19.7,-0.7) node [anchor=south west] {x};
\draw[color=black] (-1,1.3) node [anchor=south west] {y};
\draw[color=black] (-0.3,-0.3) node {0};
\draw[color=red] (-6.5,1) --  (20.5,1);
\draw[color=red] (-6.5,-1) --  (20.5,-1);
\foreach \x in {-6,-5,-4,-3,-2,-1,1,2,3,4,5,6,7,8,9,10,11,12,13,14,15,16,17,18,19}
\draw[shift={(\x,0)},color=black]  (0,-0.1) --  (0,0.1);
\foreach \x in {-6,-5,-4,-3,-2,-1,1,2,3,4,5,6,7,8,9,10,11,12,13,14,15,16,17,18,19}
\draw[shift={(\x,1)},color=black]  (0,-0.1) --  (0,0.1);
\foreach \x in {-6,-5,-4,-3,-2,-1,1,2,3,4,5,6,7,8,9,10,11,12,13,14,15,16,17,18,19}
\draw[shift={(\x,-1)},color=black]  (0,-0.1) --  (0,0.1);
\draw[color=red] (22,1.4) node {${\hat{\cal D}}_2$};
\filldraw[black,color=red] (0,0) circle (2pt);
\end{tikzpicture}}
\caption{\label{fig:TIPfail-2} Representation of $\hat{\mathcal{D}}_2$. It is possible to see that $\hat{\mathcal{D}}_2$ is translation invariant for the points in $\mathcal{L}_2$ which in this case is given by the $x$-axis.}
\end{figure}

\subsection{Lattice properties on the whole index set}
In this subsection we consider subsets $\Xi_j$ of the whole index set $\mathbb{Z}^k$ that are the equivalent of the subsets $\mathcal{D}_{j}$ of the upper-orthant. As Condition ($\mathcal{D}^{\Lambda} $) defined only the $\mathcal{D}_{j}$, the existence of the $\Xi_j$, shown in the next result, is deduced from it.
\begin{proposition}\label{lem-Xi1} Let $\Lambda_{n}$ satisfy $|\Lambda_n|\to \infty$ as $n\to \infty$ together with Condition ($\mathcal{D}^{\Lambda}$). For any $p\in\mathbb{N}$ and any $\Xi$ subset of $\mathbb{Z}^k$ with $\mathbf{0}\in\Xi$ we have that the limits
\begin{multline*}
\lim\limits_{n\to\infty}\dfrac{|\{\mathbf{t}\in\Lambda_{n}:(\Lambda_{n})_{-\mathbf{t}}\cap K_{p}=\Xi\cap K_{p}\}|}{|\Lambda_n|},\quad\text{and}\\
\quad \lim\limits_{p\to\infty}\lim\limits_{n\to\infty}\dfrac{|\{\mathbf{t}\in\Lambda_{n}:(\Lambda_{n})_{-\mathbf{t}}\cap K_{p}=\Xi\cap K_{p}\}|}{|\Lambda_n|}
\end{multline*}
exist. Moreover, any such $\Xi$ such that
\begin{equation}\label{>0}
\lim\limits_{p\to\infty}\lim\limits_{n\to\infty}\dfrac{|\{\mathbf{t}\in\Lambda_{n}:(\Lambda_{n})_{-\mathbf{t}}\cap K_{p}=\Xi\cap K_{p}\}|}{|\Lambda_n|}>0
\end{equation}
satisfies $\Xi^+=\mathcal{D}_j$ for some $j=1,...,q$. 
\end{proposition}

Define by $\Xi_1,...,\Xi_{q'}$ the sets satisfying \eqref{>0} with  $q'\in\mathbb{N}\cup\{\infty\}$. For each $m=1,...,q'$ and $p\in\mathbb{N}$ define
\begin{align*}
\gamma_m&:=\lim\limits_{p\to\infty}\lim\limits_{n\to\infty}|\{\mathbf{t}\in\Lambda_{n}:(\Lambda_{n})_{-\mathbf{t}}\cap K_{p}=\Xi_m\cap K_{p}\}|/|\Lambda_n|,\\
\gamma_{p,m}&:=\lim\limits_{n\to\infty}|\{\mathbf{t}\in\Lambda_{n}:(\Lambda_{n})_{-\mathbf{t}}\cap K_{p}=\Xi_m\cap K_{p}\}|/|\Lambda_n|,\\
F^{(m)}_p&:=\{j\in\{1,...,q'\}:\Xi_j\cap K_p=\Xi_m\cap K_p\}.
\end{align*}
Let $l_0:=j$ and $\mathbf{z}_{j}:=\mathbf{0}$. From Proposition \ref{lem-AC1-L-2} recall that $\mathcal{D}_{j} =\mathcal{G}_j^+\cup\bigcup_{i=1}^{b_j}((\mathcal L_{l_i})_{\mathbf{z}_{l_i}})^+$, which we can rewrite as $\mathcal{D}_{j} =\bigcup_{i=0}^{b_j}((\mathcal L_{l_i})_{\mathbf{z}_{l_i}})^+$.
\begin{proposition}\label{lem-Xi1.5}
Let $\Lambda_{n}$ satisfy $|\Lambda_n|\to \infty$ as $n\to \infty$ together with Condition ($\mathcal{D}^{\Lambda}$). Every $\Xi_m$, $m=1,...,q'$, is a translation of $\bigcup_{i=0}^{b_j}(\mathcal L_{l_i})_{\mathbf{z}_{l_i}}$ for some $j=1,\ldots, q$. Moreover, $\sum_{m=1}^{q'}\gamma_m=1$, and $\gamma_{p,m}=\sum_{m'\in F^{(m)}_p}\gamma_{m'}$, for every $m=1,...,q'$.
\end{proposition}
It is important to notice that the translations of $\bigcup_{i=0}^{b_j}(\mathcal L_{l_i})_{\mathbf{z}_{l_i}}$ can coincide. A careful analysis is done in order to describe the distinct translations. Recall that $l_0=j$ and that  $\mathcal L_{j}=\mathcal L_{l_0}\subset\mathcal L_{l_i}$ for every $i=1,...,b_j$, are same rank lattices by Proposition \ref{prop:lattice1}. Thus there exists a finite number of translations of $\mathcal L_{j}$ covering any $\mathcal L_{l_i}$, $i=1,...,b_j$. Denote $n_j\ge 1$ this number. Moreover, let $\mathbf{x}_1^{(j)},...,\mathbf{x}_{n_j}^{(j)}$ be the points in $\bigcup_{i=0}^{b_j}(\mathcal L_{l_i})_{\mathbf{z}_{l_i}}$ such that $\mathbf{x}_k^{(j)}\succeq\mathbf{0}$ and that $\bigcup_{h=1}^{n_j}(\mathcal L_{j})_{\mathbf{x}_{h}^{(j)} }=\bigcup_{i=0}^{b_j}(\mathcal L_{l_i})_{\mathbf{z}_{l_i}}$. Finally, let $\mathcal{E}_j=\{\mathbf{x}_1^{(j)},...,\mathbf{x}_{n_j}^{(j)}\}$ where, for the sake of clarity, we include by convention $\{\mathbf{0}\}$ in $\mathcal{E}_j$ so that $\mathbf{0}$ is always the lowest (according to $\succ$) point in $\mathcal{E}_j$. Notice that a certain arbitrary choice is still possible when choosing $\mathcal{E}_j$, see Figure \ref{fig:TIPfail} for an example.

Any $\Xi_m$ contains $\{\mathbf{0}\}$ by definition. Thus, the different $\Xi_m$s correspond to the different translated versions of $\bigcup_{i=0}^{b_j}(\mathcal L_{l_i})_{\mathbf{z}_{l_i}}$ containing  $\{\mathbf{0}\}$. Having in mind the identity $\bigcup_{\mathbf{s} \in \mathcal{E}_j}(\mathcal L_{j})_{\mathbf{s}}=\bigcup_{i=0}^{b_j}(\mathcal L_{l_i})_{\mathbf{z}_{l_i}}$, the number of different translations  is thus $n_j$ and the shifts are the elements of $\mathcal{E}_j$ (and thus $n_j=|\mathcal{E}_j|$).
Denote $I^*$ the set of the indices $j=1,\ldots,q$ satisfying
\begin{equation}\label{*}
\lim\limits_{p\to\infty}\lim\limits_{n\to\infty}\Big|\Big\{\mathbf{t}\in\Lambda_{n}:(\Lambda_{n})_{-\mathbf{t}}\cap K_{p}=\bigcup_{\mathbf{s} \in \mathcal{E}_j}(\mathcal L_{j})_{\mathbf{s}}\cap K_{p}\Big\}\Big|/|\Lambda_n|>0.
\end{equation}
For every $j\in I^*$, let $\Xi^{*}_j:=\bigcup_{\mathbf{s} \in \mathcal{E}_j}(\mathcal L_{j})_{\mathbf{s}}$ and $\gamma^{*}_j$ be the positive limit in  \eqref{>0} associated to $\Xi^{*}_j$. For  every $j\in I^*$ we have $(\Xi^{*}_j)^+=\mathcal D_j$, that there exists an $m=1,\ldots,q'$ such that $\Xi^{*}_j=\Xi_m$ and that any $\Xi_m$, $m=1\ldots,q'$, are  translated versions of $\Xi^{*}_j$, $j\in I^*$. It essentially means that the $\Xi^{*}_j$ for $j\in I^*$ are the only relevant structures in the asymptotic of $(\Lambda_n)$ as the other ones are translated versions of them:
\begin{proposition}\label{lem-Xi2}
Let $\Lambda_{n}$ satisfy $|\Lambda_n|\to \infty$ as $n\to \infty$ together with Condition ($\mathcal{D}^{\Lambda}$).	We have the identity $\sum_{j\in I^*}\gamma^*_jn_j=\sum_{j\in I^*}\gamma^*_j|\mathcal{E}_j|=1$.
\end{proposition}
By definition we have
$$\Xi^{*}_j=\bigcup_{i=0}^{b_j}(\mathcal L_{l_i})_{\mathbf{z}_{l_i}}=\bigcup_{\mathbf{s}\in\mathcal{E}_j}(\mathcal{L}_j)_{\mathbf{s}}=\bigcup_{\mathbf{s}\in\mathcal{L}_j}(\mathcal{E}_j)_{\mathbf{s}}$$ for any $j\in I^*$.

In the following statement, we link the asymptotic behaviour of $(\Lambda_n)$ with specific non-asymptotic properties of some of its subsets. In particular, we extract from $(\Lambda_n)$ specific disjoint subsets, which have helpful non-asymptotic properties(for the proof of Theorem \ref{t1-L-E-Pro}) and show that these subsets asymptotically describe the whole  $(\Lambda_n)$ satisfying Condition $(\mathcal D^\Lambda)$.  

We introduce the following notation. Let $l\in\mathbb{N}$. Consider the maximum of the $m\in\mathbb{N}\cup\{0\}$ such that $\mathcal{D}_i\cap K_l\neq\mathcal{D}_j\cap K_l$ for every $i,j\in I^*$ with $i,j<m$. Denote this maximum by $\tilde{m}_{l,1}$. Consider the maximum of the $m\in\mathbb{N}\cup\{0\}$ such that $\mathcal{D}_w\cap K_{\lfloor l/2\rfloor}\neq\mathcal{D}_{l_h}\cap K_{\lfloor l/2\rfloor}$ for every $w,v\in I^*$ with $w,v<m$ and where $\mathcal{D}_w$ is bounded,  $\mathcal{D}_v$ is unbounded and $l_h$ is any index $l_{1},...,l_{b_v}$. Denote this maximum by $\tilde{m}_{l,2}$. Consider the maximum of the $m\in\mathbb{N}\cup\{0\}$ such that $\mathcal{E}_s\subset K_{\lfloor l/4\rfloor}$ for every $s\in I^*$ with $s<m$. Denote this maximum by $\tilde{m}_{l,3}$. Then, we define $m_l$ as follows $m_l:=\min(\tilde{m}_{l,1},\tilde{m}_{l,2},\tilde{m}_{l,3})$. Notice that such $m_l$ exists because $l$ is finite and $b_j$ is finite for every unbounded $\mathcal{D}_j$.

Further, let $S_{i,l}:=\{\mathbf{t}\in\Lambda_{n}:(\Lambda_{n})_{-\mathbf{t}}\cap K_{l}=\Xi^{*}_i\cap K_{l}\}$ for every $i\in I^*$ and $l,n\in\mathbb{N}$. Notice that $S_{i,l}$ depends on $n$, but we omit the dependency to lighten the notation.
	\begin{proposition}\label{prop:latlambd}
	Let $\Lambda_{n}$ satisfy $|\Lambda_n|\to \infty$ as $n\to \infty$ together with Condition ($\mathcal{D}^{\Lambda}$). Then for every $n\in\mathbb{N}$, $j,i\in I^*$ with $j,i<m_{4l}$ and $i\neq j$, $\mathbf{t}\in S_{j,4l}$ and $\mathbf{s}\in S_{i,4l}$, we have that $(\mathcal{D}_j\cap K_{2l})_{\mathbf{t}}\cap (\mathcal{D}_i\cap K_{2l})_{\mathbf{s}}=\emptyset$. Further, there exists a set $S'_{j,4l}$, with $S'_{j,4l}\subset S_{j,4l}$, such that for every $\mathbf{t}\in S'_{i,l}$
	\begin{equation}\label{tilde-point}
	(\mathcal{D}_{j}\cap K_{\lfloor l/2\rfloor })\setminus\bigcup_{\mathbf{s}\in (S_{j,l})_{-\mathbf{t}},\mathbf{s}\prec\mathbf{0}}(\mathcal{E}_{j})_{\mathbf{s}}=(\mathcal{D}_{j}\cap K_{\lfloor l/2\rfloor})\setminus \bigcup_{\mathbf{s}\in -\mathcal{G}_i\setminus\{\mathbf{0}\}}(\mathcal{E}_{j})_{\mathbf{s}}
	\end{equation}
	and that $\lim\limits_{n\to\infty}|S'_{j,4l}|/|\Lambda_{n}|=\gamma^*_j$ for every $j\in I^*$ with $j<m_{4l}$. Finally, we obtain that $\lim\limits_{l\to\infty}\lim\limits_{n\to\infty}\frac{\sum_{i\in I^*,i<m_{4l}}|S'_{i,4l}||\mathcal{E}_i|}{|\Lambda_{n}|}=1$.
	\end{proposition}
	
	We remark that even if the Condition $(\mathcal D^\Lambda)$ is asymptotic, the sets $S'_{j,4l}$ and $S_{j,4l}$ have both asymptotic and non-asymptotic properties.
	
\begin{example}[Continuing Example \ref{Example-1}]
Recall that in this example, the observations formed a pattern $\cal C$ that repeated itself with a certain frequency in order to constitute ${\cal C}_\infty$. Using the notation of this section, we see that such frequency is represented by $\mathcal{G}_0$. The number of $\mathcal{D}$'s is $b_0+1=|\cal D|$. The Translation Invariance Property  (TIP$_j$)   is satisfied for every  $i=0,...,b_j$ and $j=1,...,|\cal D|$. The infinite union of translated patterns is what we denote by $\Xi$ in this section. Notice that any different centering of ${\cal C}_\infty$ at $\mathbf{0}$ corresponds to a different $\Xi_m$, namely one of the $\Xi$'s which we denoted $\Xi^*_h$ for some $h=1,...,|{\cal C}|$. Further, notice that $I^*=\{h\}$, that ${\cal C}=\mathcal{E}_{h}$ and that $\gamma^*_h=1/|{\cal C}|$. A non-trivial result, even in this simple example, is the last statement in Proposition \ref{prop:lattice2}. Suppose we take the union of all the $\mathcal{D}$'s and $\{\mathbf{0}\}$ and their negative counterpart, then this union is translation invariant along with $\mathcal{L}_0$, namely along that certain frequencies with which the observations repeat their pattern. Moreover, Proposition \ref{prop:latlambd} states that it is important that the observations ($\Lambda_n$) repeat the pattern $\cal C$ for a sufficiently long time. For instance, a finite number of observations (for instance, cases where the station is working intermittently   -- this is typical of non-automatic weather stations) does not matter.
\end{example}

Finally, we refer to the {\bf lattice case} when every $\Xi$'s are lattices, meaning that  $\mathcal{E}_j=\{\mathbf{0}\}$ and $\Xi_j^*=\mathcal{L}_j$ for every $j=1,\ldots, q$.

\section{Main results expressed using the spectral tail field}\label{sec:spectral}
\subsection{Laplace functional of the limiting point random field}
The first result states the convergence of the Laplace functionals to some Laplace functional without an explicit description of the point random field. The proof of the result is based on a telescoping sum argument developed initially in the time series setting by Jakubowski and co-authors \cite{J,BJMW} together with lattice property (I) from Proposition \ref{lem-AC1-L-2}.
\begin{theorem}\label{t1-L}
	Let $k,d\in\mathbb{N}$. Consider an $\mathbb{R}^{d}$-valued stationary regularly varying random field $(\mathbf{X}_{\mathbf{t}}:\mathbf{t}\in\mathbb{Z}^{k})$ with index $\alpha > 0$. We assume conditions ($\mathcal{D}^{\Lambda}$), \textnormal{(AC$^{\Lambda}_{\succeq}$)}  and  $\mathcal{A}^{\Lambda}(a_{n}^\Lambda)$. Then $N_{n}^{\Lambda}\stackrel{d}{\to}N^{\Lambda}$ on the state space $\mathbb{R}^{d}\setminus\{\mathbf{0}\}$ and the limit random measure has Laplace functional for $g\in \mathbb{C}^{+}_{K}$, given by
	\begin{equation}\label{Psi-L}
	\Psi_{N^{\Lambda}}(g)=\exp\bigg(-\int_{0}^{\infty}\sum_{i=1}^{\infty}\lambda_{i}\mathbb{E}\bigg[e^{-\sum_{\mathbf{t}\in\mathcal{D}_{i}}g(y\mathbf{\Theta}_{\mathbf{t}})}\Big(1-e^{-g(y\mathbf{\Theta}_{\mathbf{0}})} \Big) \bigg]d(-y^{-\alpha})\bigg).
	\end{equation}	
\end{theorem}
\begin{remark}
	Notice that by Tonelli's theorem and by the monotone convergence theorem, the Laplace transform is the one of a mixture distribution
	\begin{equation*}
	\Psi_{N^{\Lambda}}(g)=\prod_{i=1}^{\infty}\exp\bigg(-\int_{0}^{\infty}\mathbb{E}\bigg[e^{-\sum_{\mathbf{t}\in\mathcal{D}_{i}}g(y\mathbf{\Theta}_{\mathbf{t}})}\Big(1-e^{-g(y\mathbf{\Theta}_{\mathbf{0}})} \Big) \bigg]d(-y^{-\alpha})\bigg)^{\lambda_{i}}.
	\end{equation*}
\end{remark}

\begin{remark}
    In the asymptotically independent case, we have that $|\mathbf\Theta_\mathbf{t}|=\mathbf{0}$ for every $\mathbf{t}\neq\mathbf{0}$ and so the limit random measure has Laplace functional 
    \begin{equation*}
        \Psi_{N^{\Lambda}}(g)=\exp\bigg(-\int_{0}^{\infty}\mathbb{E}\Big[1-e^{-g(y\mathbf{\Theta}_{\mathbf{0}})} \Big]d(-y^{-\alpha})\bigg).
    \end{equation*}
    Thus, it coincides with the case of one $\mathcal{D}$ and in particular $\mathcal{D}=\emptyset$.
\end{remark}
\subsection{The spectral cluster random field in the lattice case}\label{Subsec:the-spectral-cluster-random-field-in-the-lattice}

Define for any set $A\subset\mathbb{Z}^{k}$, any sequence $\mathbf{x}=(\mathbf x_{\mathbf{t}})_{\mathbf{t}\in\mathbb{Z}^{k}}$ and any $\alpha>0$,
\begin{equation*}
\|\mathbf{x}\|_{A,\alpha}:=\bigg(\sum_{\mathbf{t}\in A}|\mathbf x_{\mathbf{t}}|^{\alpha}\bigg)^{1/\alpha}.
\end{equation*}
For the spectral tail random field $(\mathbf{\Theta}_{\mathbf{t}})_{\mathbf{t}\in\mathbb{Z}^{k}}$ of a regularly varying stationary random field we use
\begin{equation*}
\|\mathbf{\Theta}\|_{ A,\alpha}=\bigg(\sum_{\mathbf{t}\in A}|\mathbf{\Theta}_{\mathbf{t}}|^{\alpha}\bigg)^{1/\alpha}
\end{equation*}
as the normalisation constant. When~$\|\mathbf{x}\|_{ A,\alpha}<\infty$ a.s.,  We define the spectral cluster random field by
\begin{equation*}
\mathbf{Q}_{ A}:=\frac{\mathbf{\Theta}}{\|\mathbf{\Theta}\|_{ A,\alpha}}.
\end{equation*}
Using the lattice properties investigated in Proposition \ref{lem-AC1-L-2}, we show the existence of the spectral tail random field over some lattice index sets.

\begin{proposition}\label{lem-bound-L}
	Consider an $\mathbb{R}^{d}$-valued stationary regularly varying random fields $(\mathbf{X}_{\mathbf{t}})_{\mathbf{t}\in\mathbb{Z}^{k}}$ with index $\alpha>0$.  Assume conditions ($\mathcal{D}^{\Lambda}$) and \textnormal{(AC$^{\Lambda}_{\succeq}$)}. Then, $\mathbf{\Theta}_{\mathbf{t}}\to0$ a.s.~as $|\mathbf{t}|\to \infty$ for $\mathbf{t}\in\bigcup_{j=1}^{\infty}\mathcal{D}_{j}$ and so we have $\|\mathbf{\Theta}\|_{\hat{\mathcal{D}}_{j}\cup-\hat{\mathcal{D}}_{j},\alpha}<\infty$ a.s.~for every $j\in\mathbb{N}$.
\end{proposition}

\subsection{Cluster point random field expressed using the spectral cluster field in the lattice case}
Now, we present an explicit formulation of the asymptotic Laplace functional as a mixture of cluster random fields when the $\mathcal{D}$s are lattices (on the positive points).

\begin{theorem}\label{t2-L}
	Consider an $\mathbb{R}^{d}$-valued stationary regularly varying random fields $(\mathbf{X}_{\mathbf{t}})_{\mathbf{t}\in\mathbb{Z}^{k}}$ with index $\alpha>0$. We assume  conditions ($\mathcal{D}^{\Lambda}$),   \textnormal{(AC$^{\Lambda}_{\succeq}$)} and $\mathcal{A}^{\Lambda}(a_{n}^\Lambda)$. Assume also that we are in the lattice case. Then, $N^{\Lambda}_{n}\stackrel{d}{\to}N^{\Lambda}$ on $\mathbb{R}_{\mathbf{0}}^{d}$ and the limit admits the cluster point random field representation
	\begin{equation*}
	N^{\Lambda}=\sum_{j=1}^{\infty}\sum_{i\in\mathbb{N}} \sum_{\mathbf{t}\in\Xi_j^*}\varepsilon_{\Gamma_{j,i}^{-1/\alpha}\lambda_{j}^{1/\alpha} \mathbf{Q}_{\Xi_j^*,i,\mathbf{t}}}
	\end{equation*}
	where $\Big(\sum_{\mathbf{t}\in\Xi_j^*}\varepsilon_{ \mathbf{Q}_{\Xi_j^*,i,\mathbf{t}}}\Big)_{i\in\mathbb{N}}$, is an iid sequence of point random fields with state space $\mathbb{R}^{d}$, and where $(\Gamma_{j,i})_{i\in\mathbb{N}}$ are the points of a unit rate homogeneous Poisson process on $(0, \infty)$ independent of $ (\mathbf{Q}_{\Xi_j^*,i,\mathbf{t}})_{\mathbf{t}\in\mathcal{D}_{j}}$, for every $j\ge 1$. Moreover, $\Big(\sum_{i\in\mathbb{N}} \sum_{\mathbf{t}\in\Xi_j^*}\varepsilon_{\Gamma_{j,i}^{-1/\alpha}\lambda_{j}^{1/\alpha} \mathbf{Q}_{\Xi_j^*,i,\mathbf{t}}}\Big)_{j\ge 1}$ is a sequence of independent point random fields with state space $\mathbb{R}^{d}$.
\end{theorem}
We extend   the  characterization  of the clusters first provided in Basrak and Planinic \cite{BP} on the whole index set $\Xi_1^*=\mathbb{Z}^k$, $q=1$ to potential mixtures of lattices with $q>1$. For instance, when the observations grow frequently along the axis. In this case, we have $q=k$, $\Xi^*_j=\{{{\bf 0}}\}^{j-1}\times  \mathbb{Z}\times \{{{\bf 0}}\}^{k-j}$, and $\gamma_j^*=\lambda_j$, for every $j=1,...,k$. The value of the weights depends on fast the observations grow along one axis relative to the others, e.g. if on axis $j$ there are twice the observations on axis $i$ (as $n\to\infty$) then $\gamma_j^*=2\gamma_i^*$.
\section{Point random field convergence using $\Upsilon-$spectral tail field}\label{sec:Upsilon}
For general index set $\Lambda_n$ satisfying Condition ($\mathcal{D}^{\Lambda}$) that are non necessarily lattice, we  introduce new spectral tail fields.
\subsection{The $\Upsilon-$spectral tail field}
Let $\rho$ be a modulus of continuity on $(\mathbb{R}^d)^{\mathbb{Z}^{k}}$ and for any finite $\Upsilon\subset\mathbb{Z}^{k}$ let $\rho_{\Upsilon}$ the truncation of $\rho$ to $\mathbb{R}^{d\Upsilon}$.
In the following, we extend some of the results of Basrak and Segers \cite{BS} to the case of the random fields. In the time series case, the following result is contained in Theorem 5.1 of Segers et al. \cite{SZM}.
\begin{proposition}\label{pro-New-1}
	Let $(\mathbf{X}_{\mathbf{t}})_{\mathbf{t}\in\mathbb{Z}^{k}}$ be a regularly varying of index $\alpha$ random field in $\mathbb{R}^{d}$, with $\alpha\in(0,\infty)$. Let $\Upsilon$ be a finite subset of $\mathbb{Z}^{k}$. Then there exists a random field $(\mathbf{Y}_{\Upsilon,\mathbf{t}})_{\mathbf{t}\in\mathbb{Z}^{k}}$ in $\mathbb{R}^{d}$ with $\mathbb{P}(\rho_{\Upsilon}(\mathbf{Y}_{\Upsilon }) > y) = y^{-\alpha}$ for $y \geq 1$ such that as $x\to\infty$,
	\begin{equation*}
\mathcal{L}(x^{-1} \mathbf{X}_{\mathbf{s}} , . . . , x^{-1} \mathbf{X}_{\mathbf{t}} | \rho_{\Upsilon}(\mathbf{X}) > x)\stackrel{f.d.d.}{\to}
\mathcal{L}(\mathbf{Y}_{\Upsilon,\mathbf{s}} , . . . , \mathbf{Y}_{\Upsilon,\mathbf{t}}).
	\end{equation*}
	Moreover, there exists a random field $(\mathbf{\Theta}_{\Upsilon,\mathbf{t}})_{\mathbf{t}\in\mathbb{Z}^{k}}$ in $\mathbb{R}^{d}$ such that as $x\to\infty$
		\begin{equation*}
		\mathcal{L}\bigg( \frac{x^{-1}\mathbf{X}_{\mathbf{s}}}{\rho_{\Upsilon}(\mathbf{X})} , . . . ,  \frac{x^{-1}\mathbf{X}_{\mathbf{t}}}{\rho_{\Upsilon}(\mathbf{X})}\, \Big|\, \rho_{\Upsilon}(\mathbf{X}) > x\bigg)\stackrel{f.d.d.}{\to}
		\mathcal{L}(\mathbf{\Theta}_{\Upsilon,\mathbf{s}} , . . . , \mathbf{\Theta}_{\Upsilon,\mathbf{t}}).
		\end{equation*}
\end{proposition}

It is possible to see that $\mathbf{\Theta}_{\Upsilon}$ in distribution is given by $\mathbf{Y}_{\Upsilon}/\rho_{\Upsilon}(\mathbf{Y})$. For stationary regularly varying random fields it is possible to extend the time change formula of Theorem 3.2 in Samorodnitsky and Wu \cite{SW} to $\Upsilon$-spectral tail field:
\begin{proposition}\label{pro-New-2}
Let $(\mathbf{Y}_{\Upsilon,\mathbf{t}})_{\mathbf{t}\in\mathbb{Z}^{k}}$ be the tail random field in Proposition \ref{pro-New-1} and consider $\mathbf{\Theta}_{\Upsilon,\mathbf{t}}=\mathbf{Y}_{\Upsilon,\mathbf{t}}/\rho_{\Upsilon}(\mathbf{Y})$, $\mathbf{t}\in\mathbb{Z}^{k}$. Let $g:(\mathbb{R}^{d})^{\mathbb{Z}^{k}}\to\mathbb{R}$ be a bounded measurable function. Then,
\\ \textnormal{(i)} $\rho_{\Upsilon}(\mathbf{Y})$ is independent of $(\mathbf{\Theta}_{\Upsilon,\mathbf{t}})_{\mathbf{t}\in\mathbb{Z}^{k}}$.
\\ \textnormal{(ii)} for any $\mathbf{s}\in\mathbb{Z}^{k}$,
\begin{equation}\label{timechangeY-New}
\mathbb{E}[g(\mathbf{Y}_{\Upsilon,\mathbf{t} - \mathbf{s}})\mathbf{1}(\rho_{(\Upsilon)_{-\mathbf{s}}}(\mathbf{Y})\neq\mathbf{0})]=\int_{0}^{\infty}\mathbb{E}[g(r\mathbf{\Theta}_{\Upsilon,\mathbf{t}})\mathbf{1}(r\rho_{(\Upsilon)_{\mathbf{s}}}(\mathbf{\Theta})>1)]d(-r^{-\alpha}),
\end{equation}
\\ \textnormal{(iii)} for any $\mathbf{s}\in\mathbb{Z}^{k}$,
\begin{equation}\label{timechangeTheta-New}
\mathbb{E}[g(\mathbf{\Theta}_{\Upsilon,\mathbf{t} - \mathbf{s}})\mathbf{1}(\rho_{(\Upsilon)_{-\mathbf{s}}}(\mathbf{\Theta})\neq\mathbf{0})]=\mathbb{E}\bigg[g\bigg(\frac{\mathbf{\Theta}_{\Upsilon,\mathbf{t}}}{\rho_{(\Upsilon)_{\mathbf{s}}}(\mathbf{\Theta})}\bigg)\rho_{(\Upsilon)_{\mathbf{s}}}(\mathbf{\Theta})^{\alpha}  \bigg].
\end{equation}
\end{proposition}

\begin{remark}
	It is possible to see that by definition $\rho_{\Upsilon}(\mathbf{\Theta_\Upsilon})= 1$ a.s..
\end{remark}

\subsection{Asymptotic Laplace functional expressed using the $\Upsilon-$spectral tail field}
We start with a simple result on the relation between the uniform norm and the other modulus of continuity.

\begin{lemma}\label{lem-CD}
Let $\Upsilon$ be a finite subset of $\mathbb{Z}^{k}$. There exists two positive constant $C$ and $D$ with $C\leq D$ such that, for every $\epsilon>0$, $\max_{\mathbf{t}\in\Upsilon}|\mathbf{x}_{\mathbf{t}}|<\epsilon$ implies $\rho_{\Upsilon}(\mathbf{x})<\frac{\epsilon}{C}$, and $\rho_{\Upsilon}(\mathbf{x})<\epsilon$ implies $\max_{\mathbf{t}\in\Upsilon}|\mathbf{x}_{\mathbf{t}}|<D\epsilon$.
\end{lemma}

\begin{corollary}\label{co-D-C-norm}
	Consider the notation of Lemma \ref{lem-CD}. Then, $\rho_{\Upsilon}(\mathbf{x})=1$ implies that $\max_{\mathbf{t}\in\Upsilon}|\mathbf{x}_{t}|\leq D$.
\end{corollary}
\begin{proof}
	From Lemma \ref{lem-CD} we have that $\rho_{\Upsilon}(\mathbf{x})<1+\delta$ implies that $\max_{\mathbf{t}\in\Upsilon}|\mathbf{x}_{t}|< (1+\delta)D$ for every $\delta>0$.
\end{proof}

We let $C_{m}$ and $D_{m}$ denote the constants of Lemma \ref{lem-CD} for $\mathcal{E}_{m}$, for every $m\in I^*$. Notice that $C_{m}$ and $D_{m}$ depends on the chosen $\rho$, but we do not write the dependency explicitly in the notation because it does not create confusion and it lightens the notation.

\begin{remark}
	Our setting, and in particular the following Theorem \ref{t1-L-E-Pro}, is general enough to allow for countable infinitely different moduli of continuity to be used at the same time, one different $\rho_j$ for each $\mathcal{E}_{j}$ as in Lemma \ref{lem:difmod}. With some abuse of notation we denote $\rho_{j,\mathcal{E}_{j}}$ by $\rho_{ \mathcal{E}_{j}}$.
\end{remark}  

Now, consider the following assumption on the modulus of continuity.
\\ \textbf{Condition ($A^{\Lambda}_{\rho}$):} We have $\sum_{j\in I^*}\gamma^*_{j}c_{j}D_{j}^{\alpha}<\infty$, where $c_{j}=\lim\limits_{n\to\infty}\frac{\mathbb{P}(\rho_{\mathcal{E}_{j}}(\mathbf{X})>a_{n})}{\mathbb{P}(|\mathbf{X}_{\mathbf{0}}|>a_{n})}$.
\\

This condition is satisfied in many cases. For example, if the modulus of continuity is unique and coincides with the uniform norm then $D_{j}=1$ and $c_{j}\leq |\mathcal{E}_{j}|$, and so $\sum_{j\in I^*}\gamma^*_{j}c_{j}D_{j}^{\alpha}\leq 1$. Moreover, for $\rho_{\alpha}(\cdot):=\|\cdot\|_\alpha$ we have that $D_{j}=1$ and $c_{j}=|\mathcal{E}_{j}|$ and so $\sum_{j\in I^*}\gamma^*_{j}c_{j}D_{j}^{\alpha}= 1$. We remark that such condition is needed to implement a dominated convergence theorem in the proof of Theorem \ref{t1-L-E-Pro} and so, as it happens in most of the cases where a dominated convergence theorem is used, it might be possible to obtain the result for a specific $\rho$ even if condition $A^{\Lambda}_{\rho}$ is not satisfied.

We are now ready to state an anti-clustering condition tailored for conditioning on the modulii of $\mathbf X$  being large over a local subset and not necessarily $\mathbf X_{\mathbf 0}$. For every $j\in I^*$, let $R^{(j)}_{l,\Lambda_{n}}:=\big(\bigcup_{\mathbf{t}\in\{\mathbf{s}\in\Lambda_{n}:(\Lambda_{n})_{-\mathbf{s}}\supset\mathcal{E}_j\}}((\Lambda_{n})_{-\mathbf{t}}\big)^+\setminus K_{l}$ and let $\hat{M}^{\Lambda,|\mathbf{X}|,(j)}_{l,n}:=\max_{\mathbf{i}\in R^{(j)}_{l,\Lambda_{n}}}|\mathbf{X}_{\mathbf{i}}|$ \\
\textbf{Condition} (AC$^{\Lambda}_{\succeq,I^*}$): \textit{The $\mathbb{R}^{d}$-valued stationary regularly varying random field $(\mathbf{X}_{\mathbf{t}})_{\mathbf{t}\in\mathbb{Z}^{k}}$ satisfies the  condition \textnormal{(AC$^{\Lambda}_{\succeq,I^*}$)} if there exists an integer sequences $r_{n}\to\infty$ such that $k_{n}= |\Lambda_n|/|\Lambda_{r_n}|\to\infty$ and for every $j\in I^*$}
\begin{equation*}
\lim\limits_{l\to\infty}\limsup_{n\to\infty}\mathbb{P}\Big(\hat{M}^{\Lambda,|\mathbf{X}|,(j)}_{2l,r_{n}}>a_{n}^\Lambda x\,\big|\max\limits_{\mathbf{t}\in\mathcal{E}_j} |\mathbf{X}_{\mathbf{t}}|>a_{n}^\Lambda x\Big)=0.
\end{equation*}

\begin{remark} We remark that condition (AC$^{\Lambda}_{\succeq,I^*}$) is weaker than assuming that for every $j\in I^*$
	\begin{equation*}
	\lim\limits_{l\to\infty}\limsup_{n\to\infty}\mathbb{P}\Big(\hat{M}^{\Lambda,|\mathbf{X}|}_{2l,r_{n}}>a_{n}^\Lambda x\,\big|\max\limits_{\mathbf{t}\in\mathcal{E}_j} |\mathbf{X}_{\mathbf{t}}|>a_{n}^\Lambda x\Big)=0.
	\end{equation*}
\end{remark}
\begin{remark}
    If $(\mathbf{X}_{\mathbf{t}}:\mathbf{t}\in\mathbb{Z}^{k})$ is $m$-dependent then the anti-clustering conditions considered in this paper, namely (AC$^{\Lambda}_{\succeq}$) and (AC$^{\Lambda}_{\succeq,I^*}$), are satisfied.
\end{remark}
Moreover, it is possible to see that in some cases condition (AC$^{\Lambda}_{\succeq,I^*}$) is strictly weaker than condition (AC$^{\Lambda}_{\succeq}$). As we see in the following example.

\begin{example}
    [Continuing Example \ref{Example-1}] Recall that in setting of Example \ref{Example-1} $I^*=\{h\}$ and that we denote $\mathcal{E}_h$ by $\cal C$. Thus, the condition (AC$^{\Lambda}_{\succeq,I^*}$) in this setting is:
    \begin{equation*}
	\lim\limits_{l\to\infty}\limsup_{n\to\infty}\mathbb{P}\Big(\hat{M}^{\Lambda,|\mathbf{X}|,(h)}_{2l,r_{n}}>a_{n}^\Lambda x\,\big|\max\limits_{\mathbf{t}\in \cal C} |\mathbf{X}_{\mathbf{t}}|>a_{n}^\Lambda x\Big)=0.
	\end{equation*}
	In case we know that the pattern of observations will be the same, namely $\Lambda_n=\bigcup_{\mathbf{t}\in\mathcal{L}_0}({\cal C})_\mathbf{t}\cap K_{n}$, which in practice means that the weather stations perform regularly, then $R_{l,\Lambda_n}^{(h)}=\Lambda_n\setminus K_l$ and so the condition (AC$^{\Lambda}_{\succeq,I^*}$) becomes
	 \begin{equation*}
	\lim\limits_{l\to\infty}\limsup_{n\to\infty}\mathbb{P}\Big(\max\limits_{\mathbf{i}\in\Lambda_{r_n}\setminus K_{2l}}|\mathbf{X}_\mathbf{i}|>a_{n}^\Lambda x\,\big|\max\limits_{\mathbf{t}\in \cal C} |\mathbf{X}_{\mathbf{t}}|>a_{n}^\Lambda x\Big)=0.
	\end{equation*}
On the other hand we have that $R_{l,\Lambda_{n}}$ is given by $\bigcup_{\mathbf{t}\in \Lambda_n}((\Lambda_n)_{-\mathbf{t}})^{+}\setminus K_l\supset R_{l,\Lambda_n}^{(h)}$. Then, it is possible to see that (AC$^{\Lambda}_{\succeq,I^*}$) is strictly weaker than condition (AC$^{\Lambda}_{\succeq}$).
\end{example}
Let $\tilde{\mathcal{D}}_{j}=\bigcup_{\mathbf{s}\in\mathcal{G}_{j}\setminus\{\mathbf{0}\}}(\mathcal{E}_{j})_{\mathbf{s}}$. We are now ready to present of the main results of this paper.
\begin{theorem}\label{t1-L-E-Pro}
Consider an $\mathbb{R}^{d}$-valued stationary regularly varying random field $(\mathbf{X}_{\mathbf{t}})_{\mathbf{t}\in\mathbb{Z}^{k}}$ with index $\alpha > 0$. We assume conditions ($\mathcal{D}^{\Lambda}$), \textnormal{(AC$^{\Lambda}_{\succeq,I^*}$)}, $\mathcal{A}^{\Lambda}(a_{n})$ and  $(A^{\Lambda}_{\rho})$.
	Then $N_{n}^{\Lambda}\stackrel{d}{\to}N^{\Lambda}$ on the state space $\mathbb{R}^{d}\setminus\{\mathbf{0}\}$ and the limit random measure has Laplace functional for $g\in \mathbb{C}^{+}_{K}$, given by
	\begin{multline}\label{Psi-L-E-Pro}
	\Psi_{N^{\Lambda}}(g)=\\\exp\bigg(-\int_{0}^{\infty}\sum_{j\in I^*}\gamma^*_{j}c_{j}\mathbb{E}\Big[\Big(1-e^{-\sum_{\mathbf{t}\in\mathcal{E}_{j}}g( y\mathbf{\Theta}_{\mathcal{E}_{j},\mathbf{t}})} \Big)e^{-\sum_{\mathbf{t}\in\tilde{\mathcal{D}}_{j}}g( y\mathbf{\Theta}_{\mathcal{E}_{j},\mathbf{t}})}\Big]d(-y^{-\alpha})\bigg)
	\end{multline}
	where $c_{j}=\lim\limits_{n\to\infty}\frac{\mathbb{P}(\rho_{\mathcal{E}_{j}}(\mathbf{X})>a_{n})}{\mathbb{P}(|\mathbf{X}_{\mathbf{0}}|>a_{n})}$.
\end{theorem}

\subsection{The $\Upsilon-$spectral cluster field}

Recall that $\mathcal{G}_{j}$ is the lattice intersected with the non negative points associated to $\mathcal{L}_{j}$ and that the extension $\mathcal{G}_{j}$ to the whole $\mathbb{Z}^{k}$ is just given by $\mathcal{L}_{j}=\mathcal{G}_{j}\cup-\mathcal{G}_{j}$. For every $\mathcal{E}_{j}$, denote by $\mathcal{H}_{j}=\bigcup_{\mathbf{s}\in\mathcal{L}_{j}}(\mathcal{E}_{j})_{\mathbf{s}}$. Notice that $\mathcal{H}_{j}$ coincides with $\Xi^*_j$ for $j\in I^*$.

\begin{proposition}\label{lem-bound-L-New}
	Consider an $\mathbb{R}^{d}$-valued stationary regularly varying random fields $(\mathbf{X}_{\mathbf{t}})_{\mathbf{t}\in\mathbb{Z}^{k}}$ with index $\alpha>0$. We assume conditions \textnormal{(AC$^{\Lambda}_{\succeq,I^*}$)} and  $(A^{\Lambda}_{\rho})$. Then, $|\mathbf{\Theta}_{\mathcal{E}_{j},\mathbf{t}}|\to0$ a.s.~for any $|\mathbf{t}|\to \infty$ and $\mathbf{t}\in\mathcal{D}_{j}$, and so
	$\sum_{\mathbf{t}\in\mathcal{L}_{j}}\rho_{(\mathcal{E}_{j})_{\mathbf{t}}}(\mathbf{\Theta})^{\alpha}<\infty$ a.s.~and $\sum_{\mathbf{t}\in\mathcal{H}_{j}}|\mathbf{\Theta}_{\mathcal{E}_{j},\mathbf{t}}|^{\alpha}<\infty$ a.s.~for every $j=1,...,q$.
\end{proposition}

Let $\Upsilon$ be a finite subset of $\mathbb{Z}^{k}$ and $A$ a subset $\mathbb{Z}^{k}$ and $\rho$ a modulus of continuity. Define 
\begin{equation*}
\|\mathbf{\Theta}_{\Upsilon}\|_{\rho, A,\alpha}=\bigg(\sum_{\mathbf{t}\in A}\rho_{(\Upsilon)_{\mathbf{t}}}(\mathbf{\Theta})^{\alpha}\bigg)^{1/\alpha}
\end{equation*}
as the normalisation constant. We define the spectral cluster random field by
\begin{equation*}
\mathbf{Q}_{\Upsilon, A}:=\frac{\mathbf{\Theta}_{\Upsilon}}{\|\mathbf{\Theta}_{\Upsilon}\|_{\rho, A,\alpha}},
\end{equation*}
where the dependence on $\rho$ is implicit. Notice that when the modulus of continuity is $\rho_{\alpha}$, $\Upsilon$ is $\mathcal{E}_{j}$, and $ A$ is $\mathcal{L}_{j}$ then
\begin{multline*}
\|\mathbf{\Theta}_{\mathcal{E}_{j}}\|_{\rho_\alpha,\mathcal{L}_{j},\alpha}=\bigg(\sum_{\mathbf{t}\in\mathcal{H}_{j}}|\mathbf{\Theta}_{\mathcal{E}_{j}}(\mathbf{t})|^{\alpha}\bigg)^{1/\alpha}=\|\mathbf{\Theta}_{\mathcal{E}_{j}}\|_{\mathcal{H}_{j},\alpha}\quad\textnormal{and}\\\quad \|\mathbf{Q}_{\mathcal{E}_{j},\mathcal{L}_{j}}\|_{\rho_\alpha,\mathcal{L}_{j},\alpha}=\|\mathbf{Q}_{\mathcal{E}_{j}}\|_{\mathcal{H}_{j},\alpha}=1.
\end{multline*}
Observe that for bounded $\mathcal{D}_{j}$ we have that $\mathcal{E}_{j}=\{\mathbf{0}\}\cup\mathcal{D}_{j}=\mathcal{H}_{j}$ and that $\mathcal{G}_{j}=\mathcal{L}_{j}=\{\mathbf{0}\}$. We remark that when $\Upsilon=\{\bf 0\}$ we have that $\mathbf{\Theta}_{\{\bf 0\}}$, $\|\mathbf{\Theta}_{\{\bf 0\}}\|_{\rho_\alpha, A,\alpha}$, and $\mathbf{Q}_{\{\bf 0\}, A}$ are simply given by $\mathbf{\Theta}$, $\|\mathbf{\Theta}\|_{ A,\alpha}$, and $\mathbf{Q}_{ A}$ (see Section \ref{Subsec:the-spectral-cluster-random-field-in-the-lattice}).
\subsection{Cluster point random field expressed using the $\Upsilon-$spectral cluster field}
\begin{theorem}\label{t2-L-New}
	Consider an $\mathbb{R}^{d}$-valued stationary regularly varying random fields $(\mathbf{X}_{\mathbf{t}})_{\mathbf{t}\in\mathbb{Z}^{k}}$ with index $\alpha>0$.  We assume conditions ($\mathcal{D}^{\Lambda}$), (AC$^{\Lambda}_{\succeq,I^*}$), $\mathcal{A}^{\Lambda}(a_{n})$ and  $(A^{\Lambda}_{\rho})$. Then, $N^{\Lambda}_{n}\stackrel{d}{\to}N^{\Lambda}$ on $\mathbb{R}_{\mathbf{0}}^{d}$ and the limit has Laplace functional for $g\in \mathbb{C}^{+}_{K}$, with the following expression:
	\begin{equation*}
	\Psi_{N^{\Lambda}}(g)=\exp\bigg(-\sum_{j\in I^*}\gamma^*_{j}c_{j}\int_{0}^{\infty}\mathbb{E}\Big[1
	-e^{-\sum_{\mathbf{t}\in\Xi_{j}^*}g( y\mathbf{Q}_{\mathcal{E}_{j},\mathcal{L}_{j},\mathbf{t}})}\Big]d(-y^{-\alpha})\bigg).
	\end{equation*}
\end{theorem}
\begin{proof}
	It follows from Theorem \ref{t1-L-E-Pro} using the same arguments as in the proof of Theorem \ref{t2-L}, the time change formula (\ref{timechangeTheta-New}) and the fact that from Proposition \ref{lem-bound-L-New} we have that  $\sum_{\mathbf{t}\in\mathcal{L}_j}\rho_{(\mathcal{E}_{j})_{\mathbf{t}}}(\mathbf{\Theta})^{\alpha}<\infty$ a.s., for every $j\in I^*$.
\end{proof}
In the following corollary we consider Theorems \ref{t1-L-E-Pro} and \ref{t2-L-New} when the modulus of continuity is $\rho_{\alpha}$.
\begin{corollary}\label{co-new-Q}
Let the modulus of continuity be $\rho_{\alpha}$. Consider an $\mathbb{R}^{d}$-valued stationary regularly varying random fields $(\mathbf{X}_{\mathbf{t}})_{\mathbf{t}\in\mathbb{Z}^{k}}$ with index $\alpha>0$.  We assume conditions ($\mathcal{D}^{\Lambda}$), \textnormal{(AC$^{\Lambda}_{\succeq,I^*}$)} and $\mathcal{A}^{\Lambda}(a_{n})$. Then, $N^{\Lambda}_{n}\stackrel{d}{\to}N^{\Lambda}$ on $\mathbb{R}_{\mathbf{0}}^{d}$ admitting, for $g\in \mathbb{C}^{+}_{K}$, the Laplace functional:
\begin{equation*}
\Psi_{N^{\Lambda}}(g)=\exp\bigg(-\sum_{j\in I^*}\gamma^*_{j}|\mathcal{E}_{j}|\int_{0}^{\infty}\mathbb{E}\Big[1
-e^{-\sum_{\mathbf{t}\in\Xi^*_{j}}g( y\mathbf{Q}_{\mathcal{E}_{j},\mathcal{L}_{j},\mathbf{t}})}\Big]d(-y^{-\alpha})\bigg).
\end{equation*}
where $\sum_{j\in I^*}\gamma^*_{j}|\mathcal{E}_{j}|=1$.
\end{corollary}
\begin{proof}
	The result follows from Theorems \ref{t1-L-E-Pro} and \ref{t2-L-New} and from the fact that when the modulus of continuity is $\rho_{\alpha}$ we have that:
	\begin{equation*}
c_{j}=\lim\limits_{n\to\infty}\frac{\mathbb{P}((\sum_{\mathbf{t}\in\mathcal{E}_{j}}|\mathbf{X}_{\mathbf{t}}|^{\alpha})^{1/\alpha}>a_{n})}{\mathbb{P}(|\mathbf{X}_{\mathbf{0}}|>a_{n})}=|\mathcal{E}_{j}|.
	\end{equation*}
\end{proof}
Moreover, we have the following result on the representation of $N^{\Lambda}$.
\begin{proposition}\label{pro-representation N}
	Consider $N^{\Lambda}$ given in  Theorems \ref{t1-L-E-Pro} and \ref{t2-L-New}, then
	\begin{equation}\label{Rep-1}
	N^{\Lambda}=\sum_{j\in I^*}\sum_{i\in\mathbb{N}}\sum_{\mathbf{t}\in\Xi^*_{j}}\varepsilon_{\Gamma_{j,i}^{-1/\alpha}(\gamma^*_{j}c_{j})^{1/\alpha}\mathbf{Q}_{\mathcal{E}_{j},\mathcal{L}_{j},i,\mathbf{t}}}
	\end{equation}
	where $\Big(\sum_{\mathbf{t}\in\Xi^*_{j}}\varepsilon_{\mathbf{Q}_{\mathcal{E}_{j},\mathcal{L}_{j},i,\mathbf{t}}}\Big)_{i\in\mathbb{N}}$, is an iid sequence of point random fields with state space $\mathbb{R}^{d}$, and where $(\Gamma_{j,i})_{i\in\mathbb{N}}$ are the points of a unit rate homogeneous Poisson process on $(0, \infty)$ independent of $(\mathbf{Q}_{\mathcal{E}_{j},\mathcal{L}_{j},i,\mathbf{t}})_{\mathbf{t}\in\Xi^*_{j}}$, for every $j\in I^*$. Moreover, $\Big(\sum_{i\in\mathbb{N}}\sum_{\mathbf{t}\in\Xi^*_{j}}\varepsilon_{\Gamma_{j,i}^{-1/\alpha}(\gamma^*_{j}c_{j})^{1/\alpha}\mathbf{Q}_{\mathcal{E}_{j},\mathcal{L}_{j},i,\mathbf{t}}}\Big)_{j\in I^*}$ is a sequence of independent point random fields with state space $\mathbb{R}^{d}$.
	
	Finally, in the setting of Corollary \ref{co-new-Q} we have
$ N^{\Lambda}=\sum_{i\in\mathbb{N}}\sum_{l\in\mathbb{N}}\varepsilon_{\Gamma_{i}^{-1/\alpha}\hat{\mathbf{Q}}_{i,l}}$ where $\Big(\sum_{\mathbf{l}\in\mathbb{N}}\varepsilon_{\hat{\mathbf{Q}}_{i,l}}\Big)_{i\in\mathbb{N}}$, is an iid sequence of point random fields with state space $\mathbb{R}^{d}$ with mixing distribution $\mathcal{L}(\sum_{\mathbf{l}\in\mathbb{N}}\varepsilon_{\hat{\mathbf{Q}}_{i,l}})= \sum_{j\in I^*}\gamma^*_{j}|\mathcal{E}_{j}|\mathcal{L}(\sum_{\mathbf{t}\in\Xi^*_{j}}\varepsilon_{\mathbf{Q}_{\mathcal{E}_{j},\mathcal{L}_{j},\mathbf{t}}}) $ for every $i\in\mathbb{N}$, and where $(\Gamma_{i})_{i\in\mathbb{N}}$ are the points of a unit rate homogeneous Poisson process on $(0, \infty)$ independent of $(\hat{\mathbf{Q}}_{l} )_{l\in\mathbb{N}}$.
\end{proposition}
\begin{proof}
	The result follows from Theorems \ref{t1-L-E-Pro} and \ref{t2-L-New} identifying the limiting Laplace functionals such as the ones of cluster Poisson random fields.
\end{proof}

\begin{remark}
Notice that the chosen order does not affect the spectral tail random field. This fact is a clear advantage of the $\Upsilon-$spectral cluster field approach compared to the approach of Section \ref{sec:spectral} because it allows the asymptotic representation \eqref{Rep-1} in full generality, not just in the lattice case.
\end{remark}

\begin{example}
    [Continuing Example \ref{Example-1}] Armed with the results of this and the previous section we can present the extreme asymptotic behaviour of $N_n^\Lambda$, where $\Lambda_n$ is described in Example \ref{Example-1}.  Then,
    \begin{equation*}
	\Psi_{N^{\Lambda}}(g)=\exp\bigg(-\int_{0}^{\infty}\mathbb{E}\Big[1
-e^{-\sum_{\mathbf{t}\in\Xi_h^*}g( y\mathbf{Q}_{{\cal C},\mathcal{L}_{h},\mathbf{t}})}\Big]d(-y^{-\alpha})\bigg),
	\end{equation*}
	where we consider $\rho_\alpha$ as the modulus of continuity. We remark that such a clean result is not achievable when the anti-clustering condition (AC$^{\Lambda}_{\succeq,I^*}$) does not hold. Since in this case (AC$^{\Lambda}_{\succeq,I^*}$) is strictly weaker than (AC$^{\Lambda}_{\succeq}$), this representation is equivalent to
	\begin{equation*}
	\Psi_{N^{\Lambda}}(g)=\exp\bigg(-\int_{0}^{\infty}\sum_{i=1}^{|{\cal C}|}\frac{1}{|{\cal C}|}\mathbb{E}\bigg[e^{-\sum_{\mathbf{t}\in\mathcal{D}_{i}}g(y\mathbf{\Theta}_{\mathbf{t}})}\Big(1-e^{-g(y\mathbf{\Theta}_{\mathbf{0}})} \Big) \bigg]d(-y^{-\alpha})\bigg).
	\end{equation*}
	Moreover, we can see that it is not important which reference point $h\in \cal C$ we consider, since our results enjoy certain translation properties. In particular, this representation is equivalent to
	\begin{equation*}
	\Psi_{N^{\Lambda}}(g)=\exp\bigg(-\int_{0}^{\infty}\mathbb{E}\Big[1
-e^{-\sum_{\mathbf{t}\in\Xi_{s}^*}g( y\mathbf{Q}_{{\cal C},\mathcal{L}_{s},\mathbf{t}})}\Big]d(-y^{-\alpha})\bigg),
	\end{equation*}
	where $s$ is any element of $\cal C$ and $\Xi_{s}^*$ is any different centering of $\cal C_\infty$ at $\mathbf{0}$ (in our example $\mathcal{L}_{h}=\mathcal{L}_{s}$ for every $h,s=1,...,|\cal C|$).
\end{example}

\section{Applications}\label{sec:appl}
\subsection{The extremal index}\label{Sec:Extremal-index}
In this section we investigate properties of the extremal index for random fields; see the work of Hashorva  \cite{Ha} for max-stable random fields. First, let us define it.
\begin{definition}[$\Lambda$-extremal index]
	Consider an $\mathbb{R}^d$-valued stationary random field $(\mathbf{X}_\mathbf{t})_{\mathbf{t}\in\mathbb{Z}^{k}}$. Assume that for each positive $\tau$ there exists a sequence $(u_{n}(\tau))$ such that $\lim\limits_{n\to\infty}|\Lambda_n|\mathbb{P}(|\mathbf{X}_\mathbf{0}|>u_{n}(\tau))=\tau\in[0,\infty]$ holds and the limit $\lim_{n\to\infty} \mathbb{P}(\max_{\mathbf{t}\in\Lambda_{n} }|\mathbf{X}_\mathbf{t}|\leq u_{n}(\tau)) = e^{-\theta^{\Lambda}_{X}\tau}$ exists for some $\theta_{X}^{\Lambda}\in[0, 1]$. Then $\theta_{X}^{\Lambda}$ is the $\Lambda$-extremal index of $(\mathbf{X}_\mathbf{t})$.
\end{definition}
As shown by Samorodnistky and Wu \cite{SW} when $\Lambda_{n}=[1,n]^k$, the extremal index is  connected with the so called \textit{block extremal index}. In particular, let \begin{equation*}
    \theta^{\Lambda}_{n}:=\frac{\mathbb{P}(\max_{\mathbf{t}\in\Lambda_{r_{n}} }|\mathbf{X}_\mathbf{t}|>u_{n}(\tau))}{|\Lambda_{r_{n}}|\mathbb{P}(|\mathbf{X}_\mathbf{0}|>u_{n}(\tau))} \text{ and }\theta^{\Lambda}_{b}:=\lim\limits_{n\to\infty}\theta^{\Lambda}_{n}
\end{equation*}
where by $u_{n}(\tau)$ is such that $\lim\limits_{n\to\infty}|\Lambda_n|\mathbb{P}(|\mathbf{X}_\mathbf{0}|>u_{n}(\tau))=\tau\in[0,\infty]$. 

For the sake of simplicity,  we provide our results for random fields with respect to the modulus $\rho_{\alpha}$ in this section. Thus, $D_{j}=1$ and $\sum_{i\in I^*}\gamma^*_{i}|{\cal E}_{j}|=1$. For every $j\in I^*$, generalizing the approach of Janssen \cite{Janssen} for processes to random fields, let $\mathbf T^{*}_{j}$ be defined as follows: for $\mathbf{t}\in\mathcal{L}_{j}$ define
\begin{multline*}
\{\omega:\mathbf T_{j}^{*}(\omega)=\mathbf{t}\}
=\{\omega:\max_{\mathbf{z}\in(\mathcal{E}_{j})_{\mathbf{t}}}|\mathbf\Theta_{\mathcal{E}_{j},\mathbf{z}}(\omega)|-\sup_{\mathbf{s}\in\mathcal{L}_{j},\mathbf{s}\prec\mathbf{t}}\max_{\mathbf{v}\in(\mathcal{E}_{j})_{\mathbf{s}}}|\mathbf\Theta(\mathbf{v})(\omega)|>0\}\\
\cap \{\omega:\max_{\mathbf{z}\in(\mathcal{E}_{j})_{\mathbf{t}}}|\mathbf\Theta_{\mathcal{E}_{j},\mathbf{z}}(\omega)|-\sup_{\mathbf{s}\in\mathcal{L}_{j},\mathbf{s}\succeq\mathbf{t}}\max_{\mathbf{v}\in(\mathcal{E}_{j})_{\mathbf{s}}}|\mathbf\Theta(\mathbf{v})(\omega)|=0\}
\end{multline*}
and for $\mathbf{t}\in\mathbb{Z}^{k}\setminus\mathcal{L}$ let $\{\omega:\mathbf T_j^{*}(\omega)=\mathbf{t}\}=\emptyset$. If (AC$^\Lambda_{\succ,I^*}$) is satisfied then $\mathbf T_{j}^{*}$ is well defined thanks to the summability proved in Proposition \ref{lem-bound-L-New} (see also the end of the proof of Lemma \ref{lem-AC1-L-New} for the connection between the summability of $\rho$ and the one of the max norm). If (AC$^\Lambda_{\succ}$) is satisfied and all $\mathcal{H}_j$'s are lattices (namely all $\mathcal{D}\cup\{{\bf 0}\}\cup-\mathcal{D}$'s and all the $\Xi$'s are lattices) then $\mathbf T_{j}^{*}$ is also well defined thanks to the summability proved in Proposition \ref{lem-bound-L}. Observe that when $\mathcal{H}_j$ is a lattice then $\mathcal{E}_j=\{\mathbf{0}\}$ and $\mathcal{H}_j=\mathcal{L}_j$.

\begin{theorem}\label{co-u}
	Consider an $\mathbb{R}^d$-valued stationary random field $(\mathbf{X}_\mathbf{t})_{\mathbf{t}\in\mathbb{Z}^{k}}$ with index $\alpha$ and a sequence $(\Lambda_n)$ satisfying the condition ($\mathcal{D}^{\Lambda}$).
	\\\textnormal{(1)} If the anti-clustering condition \textnormal{(AC$^{\Lambda}_{\succ}$)} holds, then the limit $\theta^{\Lambda}_{b}:=\lim\limits_{n\to\infty}\theta^{\Lambda}_{n}$ exists, is positive and has the representations
	\begin{multline}\label{u-index}
	\theta^{\Lambda}_{b}=\sum_{j=1}^{\infty}\lambda_{j}\mathbb{P}(Y\sup_{\mathbf{t}\in \mathcal{D}_{j} }|\mathbf\Theta_\mathbf{t}|\leq 1)=\sum_{j=1}^{\infty}\lambda_{j}\mathbb{E}\Big[\Big(\sup_{\mathbf{t}\in\mathcal{D}_{j}\cup\{\mathbf{0}\}}|\mathbf\Theta_\mathbf{t}|^{\alpha}-\sup_{\mathbf{t}\in\mathcal{D}_{j}}|\mathbf\Theta_\mathbf{t}|^{\alpha}\Big) \Big].
	\end{multline}
	\\\textnormal{(2)} If also all the $\Xi_j^*$'s are lattices then $\theta^{\Lambda}_{b}$ admits the representations 
	\begin{align}\nonumber
	\theta^{\Lambda}_{b}&=\sum_{j=1}^\infty\lambda_{j}\mathbb{E}\bigg[\sup\limits_{\mathbf{t}\in\Xi^*_{j}}|\mathbf{Q}_{\Xi^*_{j},\mathbf{t}}|^{\alpha}\bigg]=\sum_{j=1}^\infty\lambda_{j}\mathbb{E}\bigg[\frac{\sup_{\mathbf{t}\in\Xi^*_{j}}|\mathbf \Theta_{\mathbf{t}}|^{\alpha}}{\sum_{\mathbf{s}\in\Xi^*_{j}}|\mathbf \Theta_{\mathbf{s}}|^{\alpha}} \bigg]\\
	\nonumber&=\sum_{j=1}^\infty\lambda_{j}\mathbb{E}\bigg[|\mathbf\Theta_{\mathbf{0}}|^{\alpha} \mathbf{1}(\mathbf T^{*}_{j}=\mathbf{0})\bigg]=\sum_{j=1}^{\infty}\lambda_{j}\mathbb{E}\Big[\Big(\sup_{\mathbf{t}\in\mathcal{D}_{j}\cup\{\mathbf{0}\}}|\mathbf\Theta_\mathbf{t}|^{\alpha}-\sup_{\mathbf{t}\in\mathcal{D}_{j}}|\mathbf\Theta_\mathbf{t}|^{\alpha}\Big) \Big]\,.\\
	\label{extrindlattice}
	\end{align}
	\noindent\textnormal{(3)} In any case (1) or (2), if also the mixing condition
	\begin{equation}\label{mix-1}
	\mathbb{P}(\max_{\mathbf{t}\in\Lambda_{n} }|\mathbf{X}_\mathbf{t}|\leq a^\Lambda_{n}x)-\mathbb{P}(\max_{\mathbf{t}\in\Lambda_{r_{n}} }|\mathbf{X}_\mathbf{t}|\leq a^\Lambda_{n}x)^{k_{n}}\to0,\quad n\to\infty,
	\end{equation}
	where $(k_{n}$ and $(r_{n})$ are as in the anti-clustering condition \textnormal{(AC$^{\Lambda}_{\succ}$)}, is satisfied, then $\theta^{\Lambda}_{X}$ exists and coincides with $\theta^{\Lambda}_{b}$.
\end{theorem}

     It is possible to see that Theorem \ref{co-u} (2) only applies to $\mathcal{D}$'s that are lattices. It is natural to ask whether or not a similar result holds for any $\mathcal{D}$. The answer is positive if a different anti-clustering condition is assumed. In particular, we have the equivalent result when the anti-clustering condition \textnormal{(AC$^{\Lambda}_{\succ,I^*}$)} holds.

\begin{theorem}\label{co-u-I}
	Consider an $\mathbb{R}^d$-valued stationary random field $(\mathbf{X}_\mathbf{t})_{\mathbf{t}\in\mathbb{Z}^{k}}$ with index $\alpha$ and a sequence $(\Lambda_n)$ satisfying the condition ($\mathcal{D}^{\Lambda}$).
	\\\textnormal{(1)} If the anti-clustering condition \textnormal{(AC$^{\Lambda}_{\succ,I^*}$)} holds, then the limit $\theta^{\Lambda}_{b}:=\lim\limits_{n\to\infty}\theta^{\Lambda}_{n}$ exists, is positive and has the representations
	\begin{align}
\nonumber
	\theta_{b}^{\Lambda}&=\mathbb{E}\bigg[\sup\limits_{l\in\mathbb{N}}|\hat{\mathbf{Q}}_{l}|^{\alpha} \bigg]=\sum_{j\in I^*}\gamma^*_{j}|{\cal E}_{j}|\mathbb{E}\bigg[\sup\limits_{\mathbf{t}\in\Xi^*_{j}}|\mathbf{Q}_{\mathcal{E}_{j},\mathcal{L}_{j},\mathbf{t}}|^{\alpha}\bigg]\\
	&\nonumber
	=\sum_{j\in I^*}\gamma^*_{j}|{\cal E}_{j}|\mathbb{E}\bigg[\max\limits_{\mathbf{t}\in(\mathcal{E}_{j})_{\mathbf T^{*}_{j}}}|\mathbf{Q}_{\mathcal{E}_{j},\mathcal{L}_{j},\mathbf{t}}|^{\alpha}\bigg]\\
	&\label{index4}=\sum_{j\in I^*}\gamma^*_{j}|{\cal E}_{j}|\mathbb{E}\bigg[\max_{\mathbf{z}\in\mathcal{E}_{j}}|\mathbf\Theta_{\mathcal{E}_{j},\mathbf{z}}|^{\alpha} \mathbf{1}(\mathbf T^{*}_{j}=\mathbf{0})\bigg]\,.
	\end{align}
	\\\textnormal{(2)} If also the mixing condition \eqref{mix-1}
	where $(k_{n})$ and $(r_{n})$ are as in   the anti-clustering condition \textnormal{(AC$^{\Lambda}_{\succ,I^*}$)}, is satisfied, then $\theta^{\Lambda}_{X}$ exists and coincides with $\theta^{\Lambda}_{b}$.
\end{theorem}
From the two previous results we obtain the following immedaite corollary.
\begin{corollary}\label{co-u-last}
	Assume that $(\mathbf{X}_\mathbf{t})_{\mathbf{t}\in\mathbb{Z}^{k}}$ is an $\mathbb{R}^d$-valued stationary random field with index $\alpha$, a sequence $(\Lambda_n)$ satisfying the condition ($\mathcal{D}^{\Lambda}$) and either \textnormal{(AC$^{\Lambda}_{\succ}$)} or \textnormal{(AC$^{\Lambda}_{\succ,I^*}$)} and (\ref{mix-1}). Then the extremal index $\theta_{X}$ exists, is positive, and
	\begin{equation*}
	\lim\limits_{n\to\infty}\mathbb{P}(\max_{\mathbf{t}\in\Lambda_{n} }a_{n}^\Lambda |\mathbf{X}_{\mathbf{t}}|\leq x)=\Phi_{\alpha}^{\theta_{X}}(x),\quad x>0,
	\end{equation*}
	where $\Phi_{\alpha}(x) = e^{-x^{-\alpha}}$, $x > 0$, is the standard Fr\'{e}chet distribution function and $\theta_{X}$ is given in either (\ref{u-index}) or (\ref{index4}) depending on which anti-clustering and mixing conditions are satisfied, \eqref{extrindlattice} being available only when the $\Xi^*_j$ are lattices.
\end{corollary}

\subsection{Max-stable random fields}
 Consider a non negative stationary random field $X=(X_\mathbf{t})_{\mathbf{t}\in\mathbb{Z}^{k}}$ (with state space $E=\mathbb R_+$ and $d=1$). A fundamental representation theorem by de Haan \cite{DeHaan} states that any stochastically continuous max-stable (real valued) random field $X$ can be represented (in finite dimensional distributions) as
\begin{equation}\label{eq:maxstable}
X_{\textbf{t}}=\max_{i\in\mathbb{N}}U_{i}V_{i,\textbf{t}},\quad \textbf{t}\in\mathbb{Z}^{k},
\end{equation}
where $(U_{i})_{i\in\mathbb{N}}$ is a decreasing enumeration of the points of a Poisson point process on $(0,+\infty)$ with intensity measure $u^{-2}du$, $(V_{i})_{i\in\mathbb{N}}$ are i.i.d.~copies of a non-negative random field $(V_{\textbf{t}})_{\textbf{t}\in\mathbb{Z}^{k}}$ such that $\mathbb{E}[V_{\textbf{t}}]<+\infty$ for all $\textbf{t}\in\mathbb{Z}^{k}$, the sequences $(U_{i})_{i\in\mathbb{N}}$ and $(V_{i})_{i\in\mathbb{N}}$ are independent. Observe that the above definition implies that the marginal distributions of $X$ are 1-Fr\'{e}chet, that is $\mathbb{P}(X_{\textbf{t}}\leq z)=e^{-\mathbb{E}[V_{\textbf{t}}]/z}$ for all $z>0$, where $\mathbb{E}[V_{\textbf{t}}]>0$ is a scale parameter.

The aim of this Section is to find a necessary an sufficient condition for the anti-clustering condition  (AC$^{\Lambda}_{\succeq,I^*}$) to hold for stationary max-stable random fields. We recall some notation: $\mathcal{H}_{j}=\bigcup_{\mathbf{s}\in\mathcal{L}_{j}}(\mathcal{E}_{j})_{\mathbf{s}}$ where, for every $j\ge 1$,  $\mathcal{E}_j$ are finite subsets of $\mathbb Z^k$ including ${{\bf 0}}$ and $\mathcal{L}_{j}$ are any lattice of $\mathbb Z^k$ (possibly degenerate). 
The following result is an extension of results in Samarodnistky and Wu \cite{SW}. 	Notice that the limit \eqref{Rep-1} motivates the introduction of a mixing distribution on $V_{\bf t}$ as in the second assertion below.
\begin{proposition}\label{pro-BR-1}
	Let $(X_{\textbf{t}})_{\textbf{t}\in\mathbb{Z}^{k}}$ be a stationary max-stable random field with non-negative values. Consider a sequence $\Lambda_n$ of subsets of translated~$\bigcup_{j\ge 1}\mathcal{H}_{j}$ satisfying the condition ($\mathcal{D}^{\Lambda}$) then $(X_{\textbf{t}})_{\textbf{t}\in\mathbb{Z}^{k}}$ satisfies the (AC$^{\Lambda}_{\succeq,I^*}$) condition for any $r_{n}\to\infty$ s.t.~ $\lfloor n/r_{n}\rfloor\to\infty$ if for any $i\ge 1$ and $j\in I^*$ 
	\begin{equation}\label{BR}
	\lim\limits_{|\textbf{t}|\to\infty,\textbf{t}\in (\mathcal{H}_{i})^+}V_{\textbf{t}}\,\mathbf 1\big(\max\limits_{\mathbf{t}\in\mathcal{E}_j} V_{\mathbf{t}}\ne 0\big)=0,\quad a.s.
	\end{equation}
Consider $\mathcal L(V_{\bf t})=\sum_{j\in I^*} \lambda_j \mathcal L(V_{\bf t}^{(j)})$, $\lambda_j>0$, $j\in I^*$, $\sum_{j\in I^*} \lambda_j=1$ such that each component $V_{\bf t}^{(j)}$ is supported by a subset of a translation of a unique ${\cal H}_j$, $j\ge 1$ then the condition  \eqref{BR} simplifies to
\begin{equation}\label{BR2}
\lim\limits_{|\textbf{t}|\to\infty,\textbf{t}\in (\mathcal{H}_{j})^+}V_{\textbf{t}}\,\mathbf 1\big(\max\limits_{\mathbf{t}\in\mathcal{E}_j} V_{\mathbf{t}}\ne 0\big)=0,\quad a.s.
\end{equation}
for any $j\in I^*$. 
\end{proposition}
Notice that these specific max-stable random fields could be used to model any asymptotic clustering due to our result \eqref{Rep-1}.
\begin{remark}
Under Condition ($\mathcal{D}^{\Lambda}$), the index set $\Lambda_n$ fills up the translated asymptotic index set $\bigcup_{j\ge 1}\mathcal{H}_{j}$. Checking the conditions \eqref{BR} and \eqref{BR2} requires the knowledge of the  $\mathcal{E}_j$, $j\ge 1$, beforehand. Thus it requires some  prior knowledge on the grid of the observations. 
\end{remark}
\begin{example}
    [Continuing Example \ref{Example-1}] Max-stable random fields have been widely used for modeling extremal phenomena such as storm, starting with the pioneer work of \cite{Smith}. The spatial model \cite{Schlather}, called spectrally stationary, is widely used to model space dependence because of its simplicity. It is defined as follows: consider  an iid sequence of stationary random fields $V_{i,\bf s}$, ${\bf s}\in \mathbb Z^2$ which are not null. Then $X_{\bf s}^{space}=\max_{i\ge 1}U_iV_{i,{\bf s}}$ is a stationary max-stable random field. However it does not satisfy conditions \eqref{BR} nor \eqref{BR2} in any direction of $\mathbb Z^2$ because it would contradict the stationary assumption on $V_{i,\bf s}$, ${\bf s}\in \mathbb Z^2$. The M3 representation of \cite{haan:pereira:2006,Kab} was introduced to bypass this issue. Consider now a state-space model with space defined over $\mathcal H_0=\bigcup_{{\bf t}\in {\cal L}_0}({\cal C})_{\bf t}$ where $\mathcal C$ is a finite subset of $\mathbb Z^2$. It is sufficient to check \eqref{BR} and \eqref{BR2} where the limit is taken along the time direction only. Thus a stationary space-time process $X_t$ such that its space distribution is the one of $X_{\bf t}^{space}$ can satisfy conditions \eqref{BR} and \eqref{BR2} when its extremes are sufficiently independent over time. A basic example is an iid process in time for which the condition $\max\limits_{\mathbf{t}\in\mathcal{E}_0} V_{\mathbf{t}}\ne 0=\max\limits_{\mathbf{t}\in\mathcal C\times \{0\}} V_{\mathbf{t}}\ne 0$ forces that $V_{\mathbf{t}}= 0$ for any other component ${\bf t}=({\bf s},k)$, $k\ne 0$. Such spectrally stationary models  in space   were not attainable in previous studies, see Remark 5 (i)  of \cite{BK}, because they are not ergodic as shown in \cite{DK}.
\end{example}

\section{Proofs in Section \ref{sec:lattice}}\label{sec:proof}

\subsection{Proof of Proposition \ref{lem-AC1-L-2}}
		Assume that the first statement of point (I) is false for at least some $\mathcal{D}_{j}$ and $\mathcal{D}_{i}$ with $j\neq i$. Then, for every $p\in\mathbb{N}$, $\mathcal{D}_{j}\cap K_{p}= \mathcal{D}_{i}\cap K_{p}$, which implies that $\mathcal{D}_{j}=\mathcal{D}_{i}$ contradicting Condition ($\mathcal{D}^{\Lambda}$).  By Condition ($\mathcal{D}^{\Lambda}$) we infer on one hand that $\lambda_{j,p}\geq \sum_{i\in I_{p}^{(j)}}\lambda_{i}$. Indeed $\mathcal{D}_{j}\cap K_{p}\subseteq\mathcal{D}_{j}\cap K_{p'}$ for $p'>p$ thus we have the inclusion 
		$$
		\{\mathbf{t}\in\Lambda_{n}:\Lambda_{n}^{(\mathbf{t},p')}=\mathcal{D}_{j}\cap K_{p'}\}\subseteq \{\mathbf{t}\in\Lambda_{n}:\Lambda_{n}^{(\mathbf{t},p)}=\mathcal{D}_{j}\cap K_{p}\}
		$$
		and $\lambda_{j,p} \ge \lambda_j$, $p\ge 1$, $1\le j\le q$.
		Thus for any $p'>p$ we have 
		\begin{align*}
		\{\mathbf{t}\in\Lambda_{n}:\Lambda_{n}^{(\mathbf{t},p)}=\mathcal{D}_{j}\cap K_{p}\}&=\bigcup_{i\in I_p^{(j)}}\{\mathbf{t}\in\Lambda_{n}:\Lambda_{n}^{(\mathbf{t},p)}=\mathcal{D}_{i}\cap K_{p}\}\\
		&\supseteq\bigcup_{i\in I_p^{(j)}}\{\mathbf{t}\in\Lambda_{n}:\Lambda_{n}^{(\mathbf{t},p')}=\mathcal{D}_{i}\cap K_{p'}\}\,.
		\end{align*}
		Fix $\varepsilon>0$. As $\sum_{i=1}^q\lambda_i<\infty$, there exists some $m$ sufficiently large such that $\sum_{i\ge m}\lambda_i< \varepsilon$. Moreover, there exists $p'$ sufficiently large so that $\mathcal{D}_{j}\cap K_{p'}\neq \mathcal{D}_{i}\cap K_{p'}$ for any $i,$ $j\le m$ from the reasoning above. Thus
		\begin{align*}
		|\{\mathbf{t}\in\Lambda_{n}:\Lambda_{n}^{(\mathbf{t},p)}=\mathcal{D}_{j}\cap K_{p}\}|&\ge
		|\bigcup_{i\in I_p^{(j)}, i\le m}\{\mathbf{t}\in\Lambda_{n}:\Lambda_{n}^{(\mathbf{t},p')}=\mathcal{D}_{i}\cap K_{p'}\}|\\
		&\ge \sum_{i\in I_p^{(j)}, i\le m}
		|\{\mathbf{t}\in\Lambda_{n}:\Lambda_{n}^{(\mathbf{t},p')}=\mathcal{D}_{i}\cap K_{p'}\}|
		\end{align*}
		and dividing both sides by $|\Lambda_n|$ and letting $n\to \infty$ we obtain 
		$$
		\lambda_{j,p}\ge \sum_{i\in I_p^{(j)}, i\le m}\lambda_{i,p'}\ge \sum_{i\in I_p^{(j)}, i\le m}\lambda_{i}\ge \sum_{i\in I_p^{(j)}}\lambda_{i}-\varepsilon\,.
		$$
		As it holds for every $\varepsilon>0$ it implies the desired relation $\lambda_{j,p}\geq \sum_{i\in I_{p}^{(j)}}\lambda_{i}$.
		On the other hand $\lambda_{i,p}$ cannot be strictly greater than $\sum_{i\in I_{p}^{(j)}}\lambda_{i}$. Indeed, defining the equivalence relation $i\sim j  \Leftrightarrow i\in (I_p^{(j)})_{1\le j\le q}$, one considers the partition of  $\{1,\ldots,q\}$ generated by the equivalence classes $\mathfrak{P}_p=\{1,\ldots,q\}\setminus \sim $. If 
		$$ \lim\limits_{n\to\infty}|\{\mathbf{t}\in\Lambda_{n}:\Lambda_{n}^{(\mathbf{t},p)}=\mathcal{D}_{j}\cap K_{p}\}|/|\Lambda_n|> \sum_{i\in I_{p}^{(j)}}\lambda_{i}$$ for some $j\in\{1,\ldots,q\}$ belonging to the class $\ell\in\mathfrak{P}_p$ then one gets 
		\begin{align*}
		\lim\limits_{n\to\infty}\sum_{j\in \mathfrak{P}_p}&|\{\mathbf{t}\in\Lambda_{n}:\Lambda_{n}^{(\mathbf{t},p)}=\mathcal{D}_{j}\cap K_{p}\}|/|\Lambda_n|\\ >&\lim\limits_{n\to\infty}\sum_{j\in \mathfrak{P}_p\setminus{\ell}}|\{\mathbf{t}\in\Lambda_{n}:\Lambda_{n}^{(\mathbf{t},p)}=\mathcal{D}_{j}\cap K_{p}\}|/|\Lambda_n| + \sum_{j\in I_{p}^{(\ell)}}\lambda_{j}\\
		> & \sum_{j\in \mathfrak{P}_p} \sum_{j\in I_{p}^{(j)}}\lambda_{j}=1
		\end{align*} which yields a contradiction.
		
		For point (II), consider the case of $(\Lambda_{r_{n}})_{n\in\mathbb{N}}$ whose points have a distance between each other which increases as $n$ increases. The increase of the distance within the points of $\Lambda_{r_{n}}$ as $n$ increases allows the following fact: when we consider $K_{p}$ around one of the points, for every fixed $p$, all the other points are outside $K_{p}$ for every $n$ large enough. Then, in this case there is only one $\mathcal{D}$ for $(\Lambda_{r_{n}})_{n\in\mathbb{N}}$ and it is given by the empty set.
		
		For point (III) we need the following Lemma.
		\begin{lemma}\label{lem:transcopy}
			For any $1\le j\le q$ and $\mathbf{z}\in \mathcal{D}_{j}$, there exists $1\le i\le q$ such that $\mathcal{D}_{i}=((\mathcal{D}_{j})_{-\mathbf{z}})^+$ and $\lambda_i\ge \lambda_j$.
		\end{lemma}	
		\begin{proof}	
			Consider $\mathbf{z}\in \mathcal{D}_{j}$, $1\le j\le q$ and let $D:=((\mathcal{D}_{j})_{-\mathbf{z}})^+$. Notice that for every $q\in\mathbb{N}$ there exists a $p\in\mathbb{N}$ such that $K_{p}\supset ((K_{q})_{-\mathbf{z}})^{+}$. Thus $\mathcal{D}_j\cap ((K_{q})_{-\mathbf{z}})^{+}\subset \mathcal{D}_j\cap K_p$ and thus $\mathbf{z}+\mathbf{s}$ belongs to  $\mathcal{D}_j\cap K_p$ for any $\mathbf{s}\in D\cap K_{q}$.
			Thus for every $q\in\mathbb{N}$ we have
			$$
			\{\mathbf{t}\in\Lambda_{n}:\Lambda_{n}^{(\mathbf{t},q)}=D\cap K_{q}\}\supset \{\mathbf{t}\in\Lambda_{n}:\Lambda_{n}^{(\mathbf{t},p)}=\mathcal{D}_j\cap K_{p}\}
			$$
			so that
			\begin{equation}\label{tilde-D}
			\liminf\limits_{n\to\infty}|\{\mathbf{t}\in\Lambda_{n}:\Lambda_{n}^{(\mathbf{t},q)}=D\cap K_{q}\}|/|\Lambda_{n}|\geq\lambda_{j,p}\geq \lambda_{j}\,.
			\end{equation}
			Assume that $D$ does not coincide with any $\mathcal{D}_{i}$, $1\le i\le q$. Fix $\varepsilon>0$ so small that it satisfies $\lambda_j>\varepsilon$. Let  $m$ satisfies $\sum_{i>m}\lambda_{i}<\varepsilon$ as above in the proof of point (I)  (thus $j\le m$) and  $p$ sufficiently large such that $D\cap K_{p}\neq \mathcal{D}_{i}\cap K_{p}$ and $ \mathcal{D}_{i}\cap K_{p}\neq \mathcal{D}_{k}\cap K_{p}$  for every $i\ne k\leq m$. Using the notation introduced in the proof of point (I), we have
			\begin{align*}
			\limsup\limits_{n\to\infty}&|\{\mathbf{t}\in\Lambda_{n}:\Lambda_{n}^{(\mathbf{t},p)}=D\cap K_{p}\}|/|\Lambda_{n}|\\&\le
			\limsup\limits_{n\to\infty}|\{\mathbf{t}\in\Lambda_{n}:\Lambda_{n}^{(\mathbf{t},p)}\neq\mathcal{D}_{i}\cap K_{p},\forall i\leq m\}|/|\Lambda_{n}|\\
			&\le\limsup\limits_{n\to\infty}(|\Lambda_{n}|-|\bigcup_{i\leq m}\{\mathbf{t}\in\Lambda_{n}:\Lambda_{n}^{(\mathbf{t},p)}=\mathcal{D}_{i}\cap K_{p}\}|)/|\Lambda_{n}|\\
			&\le 1-\sum_{i\le m}\lambda_{i,p}\le1-\sum_{i\le m}\lambda_{i}\leq \varepsilon
			\end{align*}
			which is in contradiction with \eqref{tilde-D}. Therefore, $D=\mathcal{D}_{i}$ for some $1\le i\le q$ and we have
			\begin{equation*}
			\lim\limits_{q\to\infty}\lim\limits_{n\to\infty}|\{\mathbf{t}\in\Lambda_{n}:\Lambda_{n}^{(\mathbf{t},q)}=D\cap K_{q}\}|/|\Lambda_{n}|=\lambda_{i}
			\end{equation*}
			and the relation $\lambda_{i}\geq\lambda_{j}$ follows from \eqref{tilde-D}.
		\end{proof}

		To prove Point (III) observe first that for any $\mathcal{D}_{j}$ there exists $b_j\in\mathbb{N}\cup\{\infty\}$, with $b_j\leq|\mathcal{D}_j|+1$, distinct sets $(\mathcal{D}_{j})_{-\mathbf{z}}^+$, $\mathbf{z}\in \mathcal{D}_{j}$. The sum of the corresponding weights $\lambda_i$ being smaller than 1 and larger than $b_j\lambda_j$, by Lemma \ref{lem:transcopy} we get the constraint $b_j\lambda_j\le 1$ and Point (III) follows.

\subsection{Proof of Proposition \ref{prop:lattice1}}

	First, for any $1\le j\le q$, let us show that $\mathcal{G}_j$ is invariant by addition in the sense that if $\mathbf{z}\in \mathcal{G}_j$ and $\mathbf{z}'\in \mathcal{G}_j$ we infer that $\mathbf{z}+\mathbf{z}'\in \mathcal{G}_j\setminus \{\mathbf{0}\}$. Indeed, $\mathbf{z}'\in \mathcal{D}_{j}= ((\mathcal{D}_{j})_{-\mathbf{z}})^+$ so that necessarily $\mathbf{z}+\mathbf{z}'\in \mathcal{D}_{j}$. Moreover we have
	\begin{align*}
	((\mathcal{D}_{j})_{-\mathbf{z}-\mathbf{z}'})^+ &= ((((\mathcal{D}_{j})_{-\mathbf{z}})^+)_{-\mathbf{z}'})^+\\
	&= ((\mathcal{D}_{j})_{-\mathbf{z}'})^+\\
	&=\mathcal{D}_{j}\,.
	\end{align*} 
	It shows that $\mathcal{G}_j$ is invariant by addition on $(\mathbb{Z}^{k})^+$. Thus $\mathcal{G}_j$ is given by a lattice, namely given $k$ (not necessarily linearly independent) distinct vectors $\mathbf{v}_{1},...,\mathbf{v}_{k}\in\mathbb{Z}^{k}$  (\textit{i.e.}~a basis of $\mathbb{Z}^{l}$ for $l\in\{1,...,k\}$ called the rank)  we have the identity $\mathcal{G}_j=\{\sum_{l=1}^{k}a_{l}\mathbf{v}_{l}:a_{l}\in\mathbb{Z}\}\cap\{\mathbf{t}\in\mathbb{Z}^{k}:\mathbf{t}\succeq\mathbf{0}\} $. We will refer to the degenerate case $\mathcal{G}_j=\{\mathbf{0}\}$ as the case of null rank $k=0$. Thus, $\mathcal{L}_j$ is a lattice on $\mathbb{Z}^k$.
	
	Let us now we prove the existence of the partition (\ref{partition}). We have to show that for any $\mathbf{z}\in \mathcal{D}_j$ there exists a unique $1\le i\le b_j$ so that $\mathbf{z}-\mathbf{z}_{l_i}\in \mathcal L_{l_i}$. We know that we have a unique $\mathcal{D}_{l_i}$ such that $\mathcal{D}_{l_i}=((\mathcal{D}_j)_{- \mathbf{z}})^+=((\mathcal{D}_j)_{- \mathbf{z}_{l_i}})^+$ for some $1\le i\le b_j$. Assume without loss of generality that $\mathbf{z}\succeq\mathbf{z}_{l_i}$. Thus, either $\mathbf{z}=\mathbf{z}_{l_i}$ and then $\mathbf{z}-\mathbf{z}_{l_i}=\mathbf{0}\in \mathcal G_{l_i}$. Or $\mathbf{z}-\mathbf{z}_{l_i}\succ \mathbf{0}$ and for any $\mathbf{s}\in \mathcal{D}_{l_i}$ we have $\mathbf{s}+\mathbf{z}\in \mathcal{D}_{j}$ and thus $\mathbf{s}+\mathbf{z}-\mathbf{z}_{l_i}\in \mathcal{D}_{l_i}$. That $\mathbf{z}-\mathbf{z}_{l_i}\in \mathcal G_{l_i}$ follows by definition of $\mathcal G_{l_i}$. Since $\mathbf{z}$ is an arbitrary point in $\mathcal{D}_j$ and since $l_i$ is unique as $\mathcal{D}_{l_i}$, we obtain the desired partition.

	Consider any $\mathcal{D}_{l}$ so that  there exists $\mathbf{z}\in \mathcal{D}_j$ satisfying $\mathcal{D}_{l}=((\mathcal{D}_j)_{- \mathbf{z}})^+$. Then for any $\mathbf{s}\in \mathcal{G}_j\setminus\{\mathbf{0}\}$ we have  
	\begin{align*}
	((\mathcal{D}_{l})_{-\mathbf{s}})^+ &= ((\mathcal{D}_{j})_{-\mathbf{z}-\mathbf{s}})^+\\
	&= ((((\mathcal{D}_{j})_{-\mathbf{s}})^+)_{-\mathbf{z}})^+\\
	&= ((\mathcal{D}_{j})_{-\mathbf{z}})^+\\
	&=\mathcal{D}_{l}\,.
	\end{align*} 
	Since $\mathbf{z}\in \mathcal{D}_j=((\mathcal{D}_{j})_{-\mathbf{s}})^+$ then $\mathbf{z}+\mathbf{s}\in \mathcal{D}_{j}$ and $\mathbf{s}\in \mathcal{D}_{l}$. Thus we proved that 
	$$
	\mathbf{s}\in \{\mathbf{z}'\in \mathcal{D}_l\cup \{\mathbf{0}\}: ( (\mathcal{D}_{l})_{-\mathbf{z}'})^+= \mathcal{D}_l \}=: \mathcal{G}_l
	$$
	and that $\mathcal{G}_j\subseteq \mathcal{G}_l$.
	
	Further, we show now that $\mathcal{G}_j$ and $\mathcal{G}_l$, $l=l_1,...,l_{b_j}$, have the same rank. Assume the contrary. Thus, let $\mathcal{G}_j=\{\sum_{l=1}^{m}a_{l}\mathbf{v}_{l}:a_{l}\in\mathbb{Z}\}^+$ where $\mathbf{v}_{1},...,\mathbf{v}_{m}\in\mathbb{Z}^{k}$ are linearly independent and $\mathcal{G}_l=\{\sum_{l=1}^{p}a_{l}\mathbf{v}'_{l}:a_{l}\in\mathbb{Z}\}^+$ where $\mathbf{v}'_{1},...,\mathbf{v}'_{p}\in\mathbb{Z}^{k}$ are linearly independent, with $k\geq p>m$. Since $\mathcal{G}_j\subseteq \mathcal{G}_l$ we know that $\mathbf{v}_{i}=c_i\mathbf{v}'_{i}$ for some $c_i\in\mathbb{Z}$, for every $i=1,...,m$. Since $a_h\mathbf{v}'_h\in \mathcal{G}_l\setminus\mathcal{G}_j$ for any $a_h\in\mathbb{Z}$ such that $a_h\mathbf{v}'_h\succeq\mathbf{0}$ and since $\mathcal{G}_j\subseteq \mathcal{G}_l$ (and $\mathcal{G}_l$ is a lattice), we have that $(\mathcal{G}_j)_{a_h\mathbf{v}'_h}\subset \mathcal{G}_l$, and again by the lattice structure of $\mathcal{G}_l$  we have $(\mathcal{L}_j)_{a_h\mathbf{v}'_h}^+\subset \mathcal{G}_l$. By induction we obtain
	\begin{equation}\label{rank}
	\bigcup_{a_{m+1}\in\mathbb{Z}}\bigcup_{a_{m+2}\in\mathbb{Z}}\cdots\bigcup_{a_p\in\mathbb{Z}}(\mathcal{L}_j)_{a_{m+1}\mathbf{v}'_{m+1}+\cdots+a_{p}\mathbf{v}'_{p}}^+\subset \mathcal{G}_l.
	\end{equation} 
	
	Now, consider $(\mathcal{D}_{j}\cap K_{q})\setminus \bigcup_{\mathbf{i}\in\mathcal{G}_j\setminus\{\mathbf{0}\}}(\mathcal{D}_{j}\cap K_{q})_{\mathbf{i}}$, namely the points in $\mathcal{D}_{j}\cap K_{q}$ without $K_{q}^{+}(\mathbf{i})$ for every $\mathbf{i}\in\mathcal{G}_j\setminus\{\mathbf{0}\}$. By (\ref{rank}) we have 
	\begin{equation*}
	\Big(\bigcup_{a_{m+1}\in\mathbb{Z}}\bigcup_{a_{m+2}\in\mathbb{Z}}\cdots\bigcup_{a_p\in\mathbb{Z}}(\mathcal{L}_j)_{a_{m+1}\mathbf{v}'_{m+1}+\cdots+a_{p}\mathbf{v}'_{p}}^+\Big)_{\mathbf{z}_l}\subset (\mathcal{G}_l)_{\mathbf{z}_l}\subset\mathcal{D}_j,
	\end{equation*}
	where $\mathbf{z}_{l}$ is defined in the statement of Point (V). Thus, $|(\mathcal{D}_{j}\cap K_{q})\setminus \bigcup_{\mathbf{i}\in\mathcal{G}_j\setminus\{\mathbf{0}\}}(\mathcal{D}_{j}\cap K_{q})_{\mathbf{i}}|\to\infty$ as $q\to\infty$ monotonically. In particular, there is a $q^{*}$ large enough such that $|(\mathcal{D}_{j}\cap K_{q^*})\setminus \bigcup_{\mathbf{i}\in\mathcal{G}_j\setminus\{\mathbf{0}\}}(\mathcal{D}_{j}\cap K_{q^{*}})_{\mathbf{i}}|>1/\lambda_j$. Since $\lim\limits_{n\to\infty}|\{\mathbf{t}\in\Lambda_{n}:\Lambda_{n}^{(\mathbf{t},q^*)}=\mathcal{D}_j\cap K_{q^*}\}|/|\Lambda_n|=\lambda_{j,q^*}$, there are $\lambda_{j,q^*}|\Lambda_{n}||(\mathcal{D}_{j}\cap K_{q})\setminus \bigcup_{\mathbf{i}\in\mathcal{G}_j\setminus\{\mathbf{0}\}}(\mathcal{D}_{j}\cap K_{q^{*}})_{\mathbf{i}}|$ asymptotically many points in $\Lambda_n$, but since $\lambda_{j,q^*}\geq\lambda_j$ we have that $\lambda_{j,q^*}|\Lambda_{n}||(\mathcal{D}_{j}\cap K_{q})\setminus \bigcup_{\mathbf{i}\in\mathcal{G}_j\setminus\{\mathbf{0}\}}(\mathcal{D}_{j}\cap K_{q^{*}})_{\mathbf{i}}|>|\Lambda_{n}|$, which leads to a contradiction. Thus, $\mathcal{G}_j$ and $\mathcal{G}_l$, $l=l_1,...,l_{b_j}$, have the same rank.  
	
	Finally, if $\mathcal{D}_j$ is bounded then by definition of $\mathcal{G}_j$ we have that $\mathcal{G}_j=\{\mathbf{0}\}$. If $\mathcal{G}_j=\{\mathbf{0}\}$ then $\mathcal{G}_{l_i}=\{\mathbf{0}\}$ because $\mathcal{G}_j$ and $\mathcal{G}_{l_i}$ have the same rank, for every $i=1,...,b_j$. Since $b_j$ is finite, we conclude that $\mathcal{D}_j$ is finite.

\subsection{Proof of Proposition \ref{prop:lattice2}}

Since $\mathcal{G}_j\subseteq \mathcal{G}_{l_i}$, it remains to show that $\mathcal{G}_j\supseteq \mathcal{G}_{l_i}$ considering that $i$ satisfies (TIP$_j$). We notice that for $\mathbf{x}\in (\mathcal{L}_{l_i})_{\mathbf{z}_{l_i}}^+$ we have $((\mathcal D_j)_{-\mathbf{x}})^+=\mathcal D_{l_i}$ as
\begin{align*}
((\mathcal D_j)_{-\mathbf{x}})^+&=\Big((\mathcal{G}_j^+)_{-\mathbf{x}}\cup\bigcup_{h=1}^{b_j}((\mathcal L_{l_h})_{\mathbf{z}_{l_h}})^+_{-\mathbf{x}}\Big)^+\\
&=((\mathcal{G}_j)_{-\mathbf{x}})^+\cup\bigcup_{h=1}^{b_j}((\mathcal L_{l_h})_{\mathbf{z}_{l_h}-\mathbf{x}})^+\\
&=((\mathcal{G}_j)_{-\mathbf{x}})^+\cup \mathcal{G}_i ^+ \cup\bigcup_{h=1,h\neq i}^{b_j}((\mathcal L_{l_h})_{\mathbf{z}_{l_h}-\mathbf{x}})^+.
\end{align*}
So that $((\mathcal D_j)_{-\mathbf{x}})^+$ is the unique $\mathcal D_l$, $l=1,\ldots ,b_j$ associated to the lattice $\mathcal G_l = \mathcal G_{l_i}$ and it coincides with $\mathcal D_{l_i}$.
Then $\mathbf{y}-\mathbf{x}\in \mathcal D_{l_i}$, with $\mathbf{y}\in\mathcal{G}_j$, is such that $((\mathcal D_{l_i})_{-(\mathbf{y}-\mathbf{x})})^+=\mathcal D_j$. We then obtain that $\mathcal{G}_j\supseteq \mathcal{G}_{l_i}$ by exchanging the role of $\mathcal{D}_{l_i}$ with the one of $\mathcal{D}_j$ in the proof of $\mathcal{G}_j\subseteq \mathcal{G}_{l_i}$ in Point (V). Further, by applying Lemma \ref{lem:transcopy} to $\mathcal{D}_{l_i}$ we conclude that $\lambda_j=\lambda_{l_i}$.

If $\mathcal{G}_j$ is a full rank lattices then it is spanned by $k$ linearly independent vectors and there always exists a point $\mathbf{s}\in\mathcal{G}_j$ such that $\mathbf{s}\succ\mathbf{z}_{l}$ for every $l=l_1,...,l_{b_j}$. This implies that the (LC$_l$) condition must be satisfied and $\mathcal{G}_j=\mathcal{G}_l$ for every $l=l_1,...,l_{b_j}$. This concludes the proof of the first statement.

Let us now prove the second statement.  By (TIP$_j$), for every $i=0,...,b_j$, and points (V) and (VI) we have that there exists $\mathbf{z}\in \mathcal{D}_{l_i}$ such that $((\mathcal{D}_{l_i})_{-\mathbf{z}})^+=\mathcal{D}_{j}$. Hence, $\mathbf{z}\in \mathcal{D}_{l_i}$ contains a translated copy of $\mathcal{D}_{j}$ hence a translated copy of any $\mathcal{D}_{l_h}$, $h=1,...,b_j$, already contained in $\mathcal{D}_{j}$. Thus  $(l_h; h=0,...,b_j)=(l_h; h=0,...,b_{l_i})$ so that $(l_h; h\in W_j)=(l_h; h\in W_{l_i})$ and then $\hat{\mathcal{D}}_{l_i}$ is the union of the same sets than $\hat{\mathcal{D}}_{j}$
\begin{equation*}
\hat{\mathcal{D}}_{l_i}=\bigcup_{h\in W_{l_i}}\mathcal D_{l_h}=\bigcup_{h\in W_j}\mathcal D_{l_h}=\hat{\mathcal{D}}_{j}\,.
\end{equation*}

Concerning the translation invariance property, we need to check that $\hat{\mathcal{D}}_j\cup\{\mathbf{0}\}\cup-\hat{\mathcal{D}}_j$ is invariant to the translation by every point in the lattice $\mathcal{L}_j$, that is $\hat{\mathcal{D}}_j\cup\{\mathbf{0}\}\cup-\hat{\mathcal{D}}_j=(\hat{\mathcal{D}}_j\cup\{\mathbf{0}\}\cup-\hat{\mathcal{D}}_j)_{\mathbf{s}}$ for every $\mathbf{s}\in\mathcal{L}_j$. With no loss of generality consider $\mathbf{s}\in\mathcal{G}_j^+$ so that for any $h\in W_j$ we have $(\mathcal D_{l_h})_{-\mathbf{s}}\cap\{\mathbf{t}\in\mathbb{Z}^k:\mathbf{t}\succeq\mathbf{0}\}=\{\mathbf{0}\}\cup\mathcal D_{l_h}$ since $\mathcal G_{l_h}=\mathcal G_j$. Hence $(\hat{\mathcal{D}}_j)_{-\mathbf{s}}\cap\{\mathbf{t}\in\mathbb{Z}^k:\mathbf{t}\succeq\mathbf{0}\}=\{\mathbf{0}\}\cup\hat{\mathcal{D}}_j$. Moreover, for similar reason for any $\mathbf{z}\in\hat{\mathcal{D}}_j\cup\{\mathbf{0}\}$ we have $\mathbf{s}+\mathbf{z}\in\hat{\mathcal{D}}_j$ so that $-(\hat{\mathcal{D}}_j\cup\{\mathbf{0}\})_{-\mathbf{s}}\subset -\hat{\mathcal{D}}_j$. It remains to show that $-\hat{\mathcal{D}}_j\setminus -(\hat{\mathcal{D}}_j)_{-\mathbf{s}}=(\hat{\mathcal{D}}_j\cup\{\mathbf{0}\})_{-\mathbf{s}}\setminus (((\hat{\mathcal{D}}_j)_{-\mathbf{s}})^+\cup\{\mathbf{0}\})$.

\subsection{Proof of Propositon \ref{lem-Xi1}}

Let $\Xi$ be a subset of $\mathbb{Z}^k$ such that $\mathbf{0}\in\Xi$ and let $p\in\mathbb{N}$. Denote by $-\mathbf{v}$ be the lowest point (according to $\succ$) of $\Xi\cap K_p$ and denote the points of $K_p\setminus\{\mathbf{t}\in\mathbb{Z}^k:\mathbf{t}\succeq -\mathbf{v}\}$ by $-\mathbf{w}_1,...,-\mathbf{w}_{u}$, for some $u\in\mathbb{N}$ which depends on $p$. Let $\Phi_{1},...,\Phi_{v}$, for some $v\in\mathbb{N}$, be the subsets of $K_{2p}^+$ such that for each $h=1,...,v$ we have $\Phi_{h}\cap (K_{p})_{\mathbf{v}}=((\Xi\cap K_{p})_{\mathbf{v}})^+$. Similarly, for every $i=1,...,u$, let $\Psi_{i,1},...,\Psi_{i,v_i}$, for some $v_i\in\mathbb{N}$, be the subsets of $K_{2p}^+$ such that for each $h=1,...,v_i$ we have $\Psi_{i,h}\cap (K_{p})_{\mathbf{w}_{i}}=(\Xi\cap K_{p})_{\mathbf{w}_{i}}$. Moreover, for every $i=1,...,u$, denote by $\Pi_{i,1},...,\Pi_{i,s_i}$ the non-empty subsets of $K_p\setminus\{\mathbf{t}\in\mathbb{Z}^k:\mathbf{t}\succeq \mathbf{v}\}$ with highest point (according to $\succ$) given by $-\mathbf{w}_i$. Observe that $((\Xi\cap K_{p})_{\mathbf{v}})^+=(\Xi\cap K_{p})_{\mathbf{v}}\setminus\{\mathbf{0}\}$ and $((\Xi\cap K_{p})_{\mathbf{w}_i})^+=(\Xi\cap K_{p})_{\mathbf{w}_i}$ for every $i=1,...,u$. First, for every $n\in\mathbb{N}$ we have that
\begin{equation*}
|\{\mathbf{t}\in\Lambda_{n}:(\Lambda_{n})_{-\mathbf{t}}\cap K_{p}=\Xi\cap K_{p}\}|
=|\{\mathbf{t}\in\Lambda_{n}:(\Lambda_{n})_{-\mathbf{t}}\cap (K_{p})_{\mathbf{v}}=(\Xi\cap K_{p})_{\mathbf{v}}\}|.
\end{equation*}
Second, we have the following identities
\begin{align*}
\{\mathbf{t}\in\Lambda_{n}:(\Lambda_{n})_{-\mathbf{t}}\cap (K_{p})_{\mathbf{v}}&=(\Xi\cap K_{p})_{\mathbf{v}}\}\\
&=\{\mathbf{t}\in\Lambda_{n}:(\Lambda_{n})_{-\mathbf{t}}\cap ((K_{p})_{\mathbf{v}})^+\\
&=((\Xi\cap K_{p})_{\mathbf{v}})^+\}\setminus \bigcup_{i=1,...,u}\bigcup_{h=1,...,v_i}\{\mathbf{t}\in\Lambda_{n}:(\Lambda_{n})_{-\mathbf{t}}\cap (K_{p})_{\mathbf{v}}\\
&=(\Pi_{i,h}\cup\Xi\cap K_{p})_{\mathbf{v}}\}.
\end{align*}
Since 
\begin{equation*}
\bigcup_{i=1,...,u}\bigcup_{h=1,...,v_i}\{\mathbf{t}\in\Lambda_{n}:(\Lambda_{n})_{-\mathbf{t}}\cap (K_{p})_{\mathbf{v}}=(\Pi_{i,h}\cup\Xi\cap K_{p})_{\mathbf{v}}\}
\end{equation*}
are unions of disjoint sets, since
\begin{multline*}
\{\mathbf{t}\in\Lambda_{n}:(\Lambda_{n})_{-\mathbf{t}}\cap ((K_{p})_{\mathbf{v}})^+=((\Xi\cap K_{p})_{\mathbf{v}})^+\}\\
\supset \bigcup_{i=1,...,u}\bigcup_{h=1,...,v_i}\{\mathbf{t}\in\Lambda_{n}:(\Lambda_{n})_{-\mathbf{t}}\cap (K_{p})_{\mathbf{v}}=(\Pi_{i,h}\cup\Xi\cap K_{p})_{\mathbf{v}}\}
\end{multline*}
and since for every $i=1,...,u$
\begin{multline*}
|\bigcup_{h=1,...,v_i}\{\mathbf{t}\in\Lambda_{n}:(\Lambda_{n})_{-\mathbf{t}}\cap (K_{p})_{\mathbf{v}}=(\Pi_{i,h}\cup\Xi\cap K_{p})_{\mathbf{v}}\}|\\
=|\{\mathbf{t}\in\Lambda_{n}:(\Lambda_{n})_{-\mathbf{t}}\cap ((K_{p})_{\mathbf{w}_i})^+=(\Xi\cap K_{p})_{\mathbf{w}_i}\}|
\end{multline*}
we have that
\begin{align*}
&|\{\mathbf{t}\in\Lambda_{n}:(\Lambda_{n})_{-\mathbf{t}}\cap ((K_{p})_{\mathbf{v}})^+=((\Xi\cap K_{p})_{\mathbf{v}})^+\}\\
&\qquad\setminus \bigcup_{i=1,...,u}\bigcup_{h=1,...,v_i}\{\mathbf{t}\in\Lambda_{n}:(\Lambda_{n})_{-\mathbf{t}}\cap (K_{p})_{\mathbf{v}}=(\Pi_{i,h}\cup\Xi\cap K_{p})_{\mathbf{v}}\}|\\
&=|\{\mathbf{t}\in\Lambda_{n}:(\Lambda_{n})_{-\mathbf{t}}\cap ((K_{p})_{\mathbf{v}})^+=((\Xi\cap K_{p})_{\mathbf{v}})^+\}|\\
&\qquad- \sum_{i=1,...,u}\sum_{h=1,...,v_i}|\{\mathbf{t}\in\Lambda_{n}:(\Lambda_{n})_{-\mathbf{t}}\cap (K_{p})_{\mathbf{v}}=(\Pi_{i,h}\cup\Xi\cap K_{p})_{\mathbf{v}}\}|\\
&=|\{\mathbf{t}\in\Lambda_{n}:(\Lambda_{n})_{-\mathbf{t}}\cap ((K_{p})_{\mathbf{v}})^+=((\Xi\cap K_{p})_{\mathbf{v}})^+\}|\\
&\qquad- \sum_{i=1,...,u}|\{\mathbf{t}\in\Lambda_{n}:(\Lambda_{n})_{-\mathbf{t}}\cap ((K_{p})_{\mathbf{w}_i})^+=(\Xi\cap K_{p})_{\mathbf{w}_i}\}|\\
&
=\sum_{l=1,...,v}|\{\mathbf{t}\in\Lambda_{n}:\Lambda_{n}^{(\mathbf{t},2p)}=\Phi_{l}\}|- \sum_{i=1,...,u}\sum_{h=1,...,v_i}|\{\mathbf{t}\in\Lambda_{n}:\Lambda_{n}^{(\mathbf{t},2p)}=\Psi_{i,h}\}|,
\end{align*}
where the last equality follows by the definition of the $\Phi$'s and the $\Psi$'s and the fact that $\Phi_l\neq\Phi_m$ for every $l,m=1,...,v$ with $l\neq m$, and that $\Psi_{i,h}\neq\Psi_{i,k}$ for every $i=1,...,u$ and $h,k=1,...,v_i$ with $h\neq k$. Now, thanks to Point (I) in Proposition \ref{lem-AC1-L-2} we have that
\begin{equation*}
\lim\limits_{n\to\infty}|\{\mathbf{t}\in\Lambda_{n}:\Lambda_{n}^{(\mathbf{t},2p)}=\Phi_{l}\}|/|\Lambda_n|=\begin{cases}
\lambda_{2p,z}=\sum_{x\in I^{(z)}_{2p}}\lambda_x& \textnormal{if $\mathcal{D}_z\cap K_{2p}=\Phi_{l}$ for some $z\in\mathbb{N}$,}\\0&\textnormal{otherwise}.
\end{cases}
\end{equation*}
for $l=1,...,v$, and similarly for $\Psi_{i,h}$ for $i=1,...,u$ and $h=1,...,v_i$. Then, we obtain that the following limit exists
\begin{equation*}
\lim\limits_{n\to\infty}|\{\mathbf{t}\in\Lambda_{n}:(\Lambda_{n})_{-\mathbf{t}}\cap K_{p}=\Xi\cap K_{p}\}|/|\Lambda_n|.
\end{equation*}

Further, observe that for $p'>p$ we have the inclusion 
$$
\{\mathbf{t}\in\Lambda_{n}:(\Lambda_{n})_{-\mathbf{t}}\cap K_{p'}=\Xi\cap K_{p'}\}\subseteq \{\mathbf{t}\in\Lambda_{n}:(\Lambda_{n})_{-\mathbf{t}}\cap K_{p}=\Xi\cap K_{p}\},
$$
thus, the following limit exists
\begin{equation*}
\lim\limits_{p\to\infty}\lim\limits_{n\to\infty}|\{\mathbf{t}\in\Lambda_{n}:(\Lambda_{n})_{-\mathbf{t}}\cap K_{p}=\Xi\cap K_{p}\}|/|\Lambda_n|.
\end{equation*}
This concludes the first part of the statement.

Now, let $\Xi^+\neq\mathcal{D}_j$ for every $1\leq j\leq q$ and fix $\varepsilon>0$. As $\sum_{i=1}^q\lambda_i=1$, there exists some $m$ sufficiently large such that $\sum_{i\ge m}\lambda_i< \varepsilon$ and there exists $p$ sufficiently large so that $\Xi\cap K_{p}^+\neq \mathcal{D}_{i}\cap K_{p}$ for any $i\le m$. Then,
\begin{equation*}
\lim\limits_{n\to\infty}|\{\mathbf{t}\in\Lambda_{n}:(\Lambda_{n})_{-\mathbf{t}}\cap K_{p}=\Xi\cap K_{p}\}|/|\Lambda_n|<\varepsilon
\end{equation*}
Thus, 
\begin{equation*}
\lim\limits_{p\to\infty}\lim\limits_{n\to\infty}|\{\mathbf{t}\in\Lambda_{n}:(\Lambda_{n})_{-\mathbf{t}}\cap K_{p}=\Xi\cap K_{p}\}|/|\Lambda_n|=0,
\end{equation*}
which concludes the proof.

\subsection{Proof of Proposition \ref{lem-Xi1.5}}
Consider first the case of bounded $\Xi_b$. In this case, we have that $|\Xi_b|<1/\gamma_b+1$; otherwise we will have a contradiction because we will asymptotically end up with more points than the ones in $\Lambda_n$. Further, denote by $-\mathbf{z}$ its lowest point according to $\succ$, then we have
\begin{equation*}
\lim\limits_{p\to\infty}\lim\limits_{n\to\infty}|\{\mathbf{t}\in\Lambda_{n}:\Lambda_{n}^{(\mathbf{t},p)}=(\Xi_b)_{\mathbf{z}}\cap K_{p}^+\}|/|\Lambda_n|>0,
\end{equation*}
which implies that $(\Xi_b)_{\mathbf{z}}=\mathcal{D}_j$ for some $j=1,...,q$. Further, since 
\begin{equation*}
\lim\limits_{p\to\infty}\lim\limits_{n\to\infty}|\{\mathbf{t}\in\Lambda_{n}:(\Lambda_{n})_{-\mathbf{t}}\cap K_{p}=\{\mathbf{0}\}\cup\mathcal{D}_j\cap K_{p}\}|/|\Lambda_n|>0,
\end{equation*}
and since by (\ref{partition}) $\mathcal{L}_j\cup\bigcup_{i=1}^{b_j}(\mathcal L_{l_i})_{\mathbf{z}_{l_i}}=\{\mathbf{0}\}\cup\mathcal{D}_j$ we obtain that $\Xi_b=(\{\mathbf{0}\}\cup\mathcal{D}_j)_{-\mathbf{z}}$ and the first statement follows. 

Now, let $\Xi_b$ be unbounded. We show that $\Xi_b$ is a finite union of translated lattices. Consider any point $\mathbf{s}$ in $\Xi_b$. Let $g_p\in\mathbb{N}$ be such that $\Xi_b\cap K_{g_p}\supset\Xi_b\cap (K_p^+)_{-\mathbf{s}}$. Then, for every $n\in\mathbb{N}$
\begin{equation*}
|\{\mathbf{t}\in\Lambda_{n}:\Lambda_{n}^{(\mathbf{t},p)}=(\Xi_b)_{-\mathbf{s}}\cap K_{p}^+\}|\geq |\{\mathbf{t}\in\Lambda_{n}:(\Lambda_{n})_{-\mathbf{t}}\cap K_{g_p}=\Xi_b\cap K_{g_p}\}|
\end{equation*}
and since this holds for every $p$ large enough, we get
\begin{equation*}
\lim\limits_{p\to\infty}\lim\limits_{n\to\infty}|\{\mathbf{t}\in\Lambda_{n}:\Lambda_{n}^{(\mathbf{t},p)}=(\Xi_b)_{-\mathbf{s}}\cap K_{p}^+\}|/|\Lambda_n|\geq\gamma_b.
\end{equation*}
Then, we have that $((\Xi_b)_{-\mathbf{s}})^+=\mathcal{D}_k$ for some $k=1,...,q$. By Proposition \ref{prop:lattice1}, we deduce that $\Xi_b$ is a union of translated lattices. Further, this union is finite because $\gamma_b$ is strictly positive.

Now, consider a point $\mathbf{r}$ on the most preceding lattice of $\Xi_b$. Then, $((\Xi_b)_{-\mathbf{r}})^+=\mathcal{D}_j$, for some $j=1,...,q$, and so $((\Xi_b)_{-\mathbf{r}})=\bigcup_{i=0}^{b_j}(\mathcal L_{l_i})_{\mathbf{z}_{l_i}}$, which concludes the proof of the first statement.

Let us now prove the second statement. Let $p\in\mathbb{N}$. Define the equivalence relation $i\sim j  \Leftrightarrow i\in (F_p^{(j)})_{1\le j\le q'}$, one considers the partition of  $\{1,\ldots,q'\}$ generated by the equivalence classes $\mathfrak{P}'_p=\{1,\ldots,q'\}\setminus \sim $. Recall the definition of $\mathfrak{P}_p$ from the proof of point (I) in Proposition \ref{lem-AC1-L-2}. For every $l\in\mathfrak{P}_p$, let $\mathfrak{P}'_{p,l}\subset\mathfrak{P}'_p$ such that $i\in \mathfrak{P}'_{p,l}$ if $\Xi_i\cap K_p^+=\mathcal{D}_l\cap K_p^+$. Since 
\begin{equation*}
|\{\mathbf{t}\in\Lambda_{n}:\Lambda_{n}^{(\mathbf{t},p)}=\mathcal{D}_{j}\cap K_{p}\}|=\sum_{i=1,...,u}|\{\mathbf{t}\in\Lambda_{n}:(\Lambda_{n})_{-\mathbf{t}}\cap K_{p}=\Pi_{i}\cup\{\mathbf{0}\}\cup\mathcal{D}_{j}\cap K_{p}\}|
\end{equation*}
where $\Pi_1,...,\Pi_u$, $u\in\mathbb{N}$, are the subsets of $K_p^-\setminus\{\mathbf{0}\}$, we obtain that $\lambda_{p,j}=\sum_{i\in \mathfrak{P}'_{p,j}}\gamma_{p,i}$. Thus, we have that $1=\sum_{j\in\mathfrak{P}_p}\lambda_{p,j}=\sum_{j\in\mathfrak{P}_p}\sum_{i\in \mathfrak{P}'_{p,j}}\gamma_{p,i}=\sum_{i\in \mathfrak{P}'_p}\gamma_{p,i}$. By applying Fatou's lemma, we get that
\begin{multline*}
1=\lim\limits_{p\to\infty}\sum_{i\in \mathfrak{P}'_p}\gamma_{p,i}=\liminf\limits_{p\to\infty}\sum_{i\in \mathfrak{P}'_p}\gamma_{p,i}\\\geq\sum_{j=1}^{q'}\liminf\limits_{p\to\infty}\lim\limits_{n\to\infty}\frac{|\{\mathbf{t}\in\Lambda_{n}:(\Lambda_{n})_{-\mathbf{t}}\cap K_{p}=\Xi_j\cap K_{p}\}|}{|\Lambda_n|}=\sum_{j=1}^{q'}\gamma_j.
\end{multline*}
Hence, $1\geq \sum_{j=1}^{q'}\gamma_j$. By applying the same arguments as the ones used in the proof of point (I) in Proposition \ref{lem-AC1-L-2} we have that $\sum_{i\in F_p^{(j)}}\gamma_i\geq\gamma_{j,p}$ for every $j\in\mathfrak{P}'_p$, which implies that $\sum_{i=1}^{q'}\gamma_i=\sum_{j\in\mathfrak{P}'_p}\sum_{i\in F_p^{(j)}}\gamma_i\geq\sum_{j\in\mathfrak{P}'_p}\gamma_{j,p}=1$. Therefore, combining the two results we have that $\sum_{i=1}^{q'}\gamma_i=1$ and $\sum_{i\in F_p^{(j)}}\gamma_i=\gamma_{j,p}$ for every $j\in\mathfrak{P}'_p$.

\subsection{Proof of Proposition \ref{lem-Xi2}}

	Let $j\in I^*$. Consider $(\Xi^{*}_j)_{-\mathbf{x}}$ for every $\mathbf{x}\in\mathcal{E}_j$ and observe that these are the only possible $\Xi$s that can be formed by translations of $\Xi^{*}_j$. The weights of the sets $(\Xi^{*}_j)_{-\mathbf{x}}$, $\mathbf{x}\in\mathcal{E}_j$, are all equal to $\gamma_j^*$. This is because of the following arguments. Let $\gamma_b$ be the weight of $(\Xi^{*}_j)_{-\mathbf{x}}$. Consider a point $\mathbf{x}\in\mathcal{E}_j\setminus\{\mathbf{0}\}$. For any $p\in\mathbb{N}$, let $g_p\in\mathbb{N}$ be such that $\Xi^{*}_j\cap K_{g_p}\supset\Xi^{*}_j\cap (K_p)_{\mathbf{x}}$. Then, for every $n\in\mathbb{N}$
	\begin{equation*}
	|\{\mathbf{t}\in\Lambda_{n}:(\Lambda_{n})_{-\mathbf{t}}\cap K_{p}=(\Xi^{*}_j)_{-\mathbf{x}}\cap K_{p}\}|\geq |\{\mathbf{t}\in\Lambda_{n}:(\Lambda_{n})_{-\mathbf{t}}\cap K_{g_p}=\Xi^{*}_j\cap K_{g_p}\}|
	\end{equation*}
	which implies that $\gamma_b\geq\gamma_j^*$. Conversely, for any $p\in\mathbb{N}$, let $f_p\in\mathbb{N}$ be such that $(\Xi^{*}_j)_{-\mathbf{x}}\cap K_{f_p}\supset(\Xi^{*}_j)_{-\mathbf{x}}\cap (K_p)_{-\mathbf{x}}$. Then, for every $n\in\mathbb{N}$
	\begin{equation*}
	|\{\mathbf{t}\in\Lambda_{n}:(\Lambda_{n})_{-\mathbf{t}}\cap K_{p}=\Xi^{*}_j\cap K_{p}\}|\geq |\{\mathbf{t}\in\Lambda_{n}:(\Lambda_{n})_{-\mathbf{t}}\cap K_{f_p}=(\Xi^{*}_j)_{-\mathbf{x}}\cap K_{f_p}\}|
	\end{equation*}
	which implies that $\gamma_j^*\geq\gamma_b$, hence $\gamma_j^*=\gamma_b$. Thus, for each $j\in I^*$ we have that the sum of the weights of the $\Xi$s composed by the translations of $\Xi^{*}_k$ is $|\mathcal{E}_j|\gamma^*_j$. Since each $\Xi_b$, $b=1,...,q'$, is the translation of a $\Xi^{*}_k$ for some $k\in I^*$, we obtain that $\sum_{j\in I^*}\gamma^*_j|\mathcal{E}_j|=\sum_{i=1}^{q'}\gamma_i=1$, where the last equality comes from Proposition \ref{lem-Xi1.5}.

\subsection{Proof of Proposition \ref{prop:latlambd}}

First, we have that $m_l\to\infty$ as $l\to\infty$ and, since by Lemma \ref{lem-Xi2} we know that $\sum_{i\in I^*}\gamma_i^*|\mathcal{E}_i|=1$, we obtain that $\sum_{i\in I^*,i>m_l}\gamma_i^*|\mathcal{E}_i|\to 0$ as $l\to\infty$.
	
	For every $n\in\mathbb{N}$ and $j\in I^*$ with $j<m_{4l}$, consider the set $S_{j,4l}$. We let the dependency on $n$ be implicit. Let $j,i\in I^*$ with $j,i<m_{4l}$ and $i\neq j$. In the following we show that for every $\mathbf{t}\in S_{j,4l}$ and $\mathbf{s}\in S_{i,4l}$ we have that $(\mathcal{D}_j\cap K_{2l})_{\mathbf{t}}\cap (\mathcal{D}_i\cap K_{2l})_{\mathbf{s}}=\emptyset$. The idea behind the following proof is that by taking points in $\Lambda_{n}$ with certain structure on $K_{4l}$ around them (\textit{i.e.}~$\Xi\cap K_{4l}$ for $i\in I^*$ with $i<m_{4l}$) where $l$ is large enough (see above), we ensure that the sets $K^+_{2l}$ around them do not intersect for different structures (\textit{i.e.}~$(\mathcal{D}_j\cap K_{2l})_{\mathbf{t}}\cap (\mathcal{D}_i\cap K_{2l})_{\mathbf{s}}=\emptyset$, for every $\mathbf{t}\in S_{j,4l}$ and $\mathbf{s}\in S_{i,4l}$ and every $i,j\in I^*$ with $i,j<m_{4l}$ and $i\neq j$).
	
	First, consider the case of $\mathcal{D}_i$ and $\mathcal{D}_j$ bounded. Notice that $\mathbf{t}\neq\mathbf{s}$ because $\mathcal{D}_j\cap K_{4l}\neq \mathcal{D}_i\cap K_{4l}$ and so $\Xi^{*}_j\cap K_{4l}\neq \Xi^{*}_i\cap K_{4l}$. Thus, if $(\mathcal{D}_j\cap K_{2l})_{\mathbf{t}}$ and $(\mathcal{D}_i\cap K_{2l})_{\mathbf{s}}$ have an intersection then one of the two $\Xi^*\cap K_{4l}$'s will have at least one point in $K_{4l}^-\setminus\{\mathbf{0}\}$ (in particular at $\mathbf{s}-\mathbf{t}$ if $\mathbf{t}\succ\mathbf{s}$ or at $\mathbf{t}-\mathbf{s}$ if $\mathbf{s}\succ\mathbf{t}$) which is impossible by definition of bounded $\Xi^*$'s because its lowest point (according to $\succ$) is $\{\mathbf{0}\}$.
	
	Second, consider the case of $\mathcal{D}_i$ bounded and $\mathcal{D}_j$ unbounded. Then, as before $\mathbf{t}\neq\mathbf{s}$. Moreover, if $(\mathcal{D}_j\cap K_{2l})_{\mathbf{t}}$ and $(\mathcal{D}_i\cap K_{2l})_{\mathbf{s}}$ have an intersection and $\mathbf{s}\succ\mathbf{t}$ then $\Xi^{*}_i\cap K_{4l}$ will have at least one point in $K_{4l}^-\setminus\{\mathbf{0}\}$ which is impossible. If they have an intersection and $\mathbf{t}\succ\mathbf{s}$, then we have $\Xi^{*}_i\cap K_{2l}=\mathcal{D}_{l_h}\cap K_{2l}$ for some $l_h=l_1,...,l_{b_j}$, because $(K_{4l})_{\mathbf{t}}\supset(K_{2l})_{\mathbf{s}}$ and so the structure of $(\Xi^{*}_j\cap K_{4l})_{\mathbf{t}}$ implies that $(\Lambda_n)_{-\mathbf{s}}\cap K^+_{2l}=\mathcal{D}_{l_h}\cap K_{2l}$ for some $l_h=l_1,...,l_{b_j}$. However, the equality $\Xi^{*}_i\cap K_{2l}=\mathcal{D}_{l_h}\cap K_{2l}$ is impossible by construction.
	
	Third, consider the case of $\mathcal{D}_i$ and $\mathcal{D}_j$ unbounded. Then, as before $\mathbf{t}\neq\mathbf{s}$. Further, if  $\mathbf{t}\succ\mathbf{s}$, then we have $\Xi^{*}_i\cap K_{2l}=\mathcal{D}_{l_h}\cap K_{2l}$ for some $l_h=l_1,...,l_{b_j}$ as in the previous paragraph, which is impossible by construction. We conclude by observing that the case $\mathbf{s}\succ\mathbf{t}$ is specular to the case $\mathbf{t}\succ\mathbf{s}$.
	
	Now, we bound the number of points in $S_{i,4l}$ for every $i<m_{4l}$ and $i\in I^*$. For every $\mathcal{D}_i$ bounded with $i<m_{4l}$ and $i\in I^*$, let $S'_{i,4l}=S_{i,4l}$ if $|S_{i,4l}|\leq\gamma^{*}_i|\Lambda_{n}|$ and let $S'_{i,4l}$ be a subset of $S_{i,4l}$ with $|S'_{i,4l}|=\gamma^{*}_i|\Lambda_{n}|$ if $|S_{i,4l}|>\gamma^{*}_i|\Lambda_{n}|$.
	
	For the unbounded case we have the following. Consider any $\mathcal{D}_i$ unbounded with $i<m_{4l}$ and $i\in I^*$. Notice that $(\mathcal{E}_{i})_{\mathbf{s}}\cap(\mathcal{E}_{i})_{\mathbf{t}}=\emptyset$ for every $\mathbf{s},\mathbf{t}\in S_{i,4l}$ because $\mathcal{E}_{i}\subset K_l^+\cup\{\mathbf{0}\}$ and because we are considering $\Xi_i\cap K_{4l}$ and thus an intersection would violate the structure of $\Xi_i\cap K_{4l}$. For any $\mathbf{t}\in S_{i,4l}$, consider the set $(\mathcal{D}_{i}\cap K_{2l})\setminus\bigcup_{\mathbf{s}\in (S_{i,4l})_{-\mathbf{t}},\mathbf{s}\prec\mathbf{0}}(\mathcal{E}_{i})_{\mathbf{s}}$. Consider the set of points $\mathbf{t}\in S_{i,4l}$ such that 
	\begin{equation*}
	(\mathcal{D}^*_{i}\cap K_{2l})\setminus\bigcup_{\mathbf{s}\in (S_{i,4l})_{-\mathbf{t}},\mathbf{s}\prec\mathbf{0}}(\mathcal{E}_{i})_{\mathbf{s}}=(\mathcal{D}^*_{i}\cap K_{2l})\setminus \bigcup_{\mathbf{s}\in -\mathcal{G}_i\setminus\{\mathbf{0}\}}(\mathcal{E}_{i})_{\mathbf{s}}
	\end{equation*}
	and denote it by $\tilde{S}_{i,4l}$. We remark that 
	\begin{equation*}
	(\mathcal{D}_{i}\cap K_{2l})\setminus\bigcup_{\mathbf{s}\in (S_{i,4l})_{-\mathbf{t}},\mathbf{s}\prec\mathbf{0}}(\mathcal{E}_{i})_{\mathbf{s}}=(\mathcal{D}_{i}\cap K_{2l})\setminus\bigcup_{\mathbf{s}\in (\cup_{i\in I^*,i<m_{4l}}S_{i,4l})_{-\mathbf{t}},\mathbf{s}\prec\mathbf{0}}(\mathcal{E}_{i})_{\mathbf{s}}
	\end{equation*}
	because by construction for every $\mathbf{t}\in S_{i,4l}$ and $\mathbf{s}\in S_{j,4l}$, where $j\in I^*$ with $j\neq i$ and $j<m_{4l}$, we have that $(\mathcal{D}_i\cap K_{2l})_{\mathbf{t}}\cap (\mathcal{D}_j\cap K_{2l})_{\mathbf{s}}=\emptyset$ and so that $(\mathcal{D}_i\cap K_{2l})_{\mathbf{t}}\cap (\mathcal{E}_j)_{\mathbf{s}}=\emptyset$.
	
	Now, if $|\tilde{S}_{i,4l}|>\gamma^{*}_i|\Lambda_{n}|$ and we arbitrarily take out a point $\mathbf{x}\in S_{i,4l}$ then we might end up taking out more than one point in $\tilde{S}_{i,4l}$ because the points in $\tilde{S}_{i,4l}$ need the existence of certain points in $S_{i,4l}$ around them. Thus, we need to show that it is possible to find a procedure in which by taking out a certain point in $S_{i,4l}$ we only take out one (and only one) point in $\tilde{S}_{i,4l}$.
	
	Let $v:=|\{\mathbf{s}\in-\mathcal{G}_i\setminus\{\mathbf{0}\}:(\mathcal{D}_{i}\cap K_{2l})\cap(\mathcal{E}_{i})_{\mathbf{s}}\neq\emptyset\}|$. For each $\mathbf{t}\in\tilde{S}_{i,4l}$, let $\mathbf{s}_{\mathbf{t},1}\prec...\prec\mathbf{s}_{\mathbf{t},v}\prec\mathbf{0}$ be the points in $(S_{i,4l})_{-\mathbf{t}}$ such that $(\mathcal{D}_{i}\cap K_{2l})\cap(\mathcal{E}_{i})_{\mathbf{s}_{\mathbf{t},h}}\neq\emptyset$, $h=1,...,v$. Consider the lowest point in $\tilde{S}_{i,4l}$ according to $\succ$ and denote it by $\mathbf{w}$. Then, by taking out $\mathbf{s}_{\mathbf{w},1}$ from $S_{i,4l}$ we only take out $\mathbf{w}$ from $\tilde{S}_{i,4l}$ (but not from $S_{i,4l}$). This is because $\mathbf{s}_{\mathbf{w},1}$ is the lowest among $\mathbf{s}_{\mathbf{w},1},...,\mathbf{s}_{\mathbf{w},v}$ and since $\mathbf{s}_{\mathbf{w},1}$ is the lowest point in $\tilde{S}_{i,4l}$, this implies that $\mathbf{s}_{\mathbf{w},1}\neq \mathbf{s}_{\mathbf{t},h}$ for every $\mathbf{t}\in\tilde{S}_{i,4l}\setminus\{\mathbf{w}\}$ and every $h=1,...,v$. Thus, by taking out $\mathbf{s}_{\mathbf{w},1}$ from $S_{i,4l}$ we are not taking out any other point in $\tilde{S}_{i,4l}$ apart from $\mathbf{w}$.
	
	Now, if $|\tilde{S}_{i,4l}|\leq\gamma^{*}_i|\Lambda_{n}|$ let $S'_{i,4l}=\tilde{S}_{i,4l}$, while if $|\tilde{S}_{i,4l}|>\gamma^{*}_i|\Lambda_{n}|$ then, following the above procedure, reduces the points in $S_{i,4l}$ to obtain a set, which we denote $S^{(reduced)}_{i,4l}$, such that $|\tilde{S}^{(reduced)}_{i,4l}|=\gamma^{*}_i|\Lambda_{n}|$ and let $S'_{i,4l}=\tilde{S}^{(reduced)}_{i,4l}$.
	
	Concerning the asymptotic behaviour of $S'_{i,4l}$, in the bounded case, since $\lim\limits_{n\to\infty}|\{\mathbf{t}\in\Lambda_{n}:(\Lambda_{n})_{-\mathbf{t}}\cap K_{4l}=\Xi^{*}_i\cap K_{4l}\}|/|\Lambda_{n}|\geq\gamma^*_i$, by continuity of the minimum function we obtain that 
	\begin{equation*}
	\lim\limits_{n\to\infty}|S'_{i,4l}|/|\Lambda_{n}|=\lim\limits_{n\to\infty}(|\{\mathbf{t}\in\Lambda_{n}:(\Lambda_{n})_{-\mathbf{t}}\cap K_{4l}=\Xi^{*}_i\cap K_{4l}\}|\wedge \gamma^*_i|\Lambda_{n}|)/|\Lambda_{n}|=\gamma^*_i.
	\end{equation*}
	In the unbounded case, notice that $S_{i,4l}\supset\tilde{S}_{i,4l}\supset S_{i,p}$ for every $p\geq8l$ and every $n\in\mathbb{N}$. Since $\lim\limits_{n\to\infty}|S_{i,p}|/|\Lambda_{n}|=\lim\limits_{n\to\infty}|\{\mathbf{t}\in\Lambda_{n}:(\Lambda_{n})_{-\mathbf{t}}\cap K_{p}=\Xi^{*}_i\cap K_{p}\}|/|\Lambda_{n}|\geq \gamma^*_i$ for every $p\in\mathbb{N}$, then we have that
	\begin{equation*}
	\gamma^*_i=\lim\limits_{n\to\infty}(|\{\mathbf{t}\in\Lambda_{n}:(\Lambda_{n})_{-\mathbf{t}}\cap K_{4l}=\Xi^{*}_i\cap K_{4l}\}|\wedge \gamma^*_i|\Lambda_{n}|)/|\Lambda_{n}|\leq\lim\limits_{n\to\infty}(|\tilde{S}_{i,4l}|\wedge \gamma^*_i|\Lambda_{n}|)/|\Lambda_{n}|
	\end{equation*}
	\begin{equation*}
	\leq\lim\limits_{n\to\infty}(|\{\mathbf{t}\in\Lambda_{n}:(\Lambda_{n})_{-\mathbf{t}}\cap K_{8l}=\Xi^{*}_i\cap K_{8l}\}|\wedge \gamma^*_i|\Lambda_{n}|)/|\Lambda_{n}|=\gamma^*_i.
	\end{equation*}
	Since $|\tilde{S}_{i,4l}|\wedge \gamma^*_i|\Lambda_{n}|=|S'_{i,4l}|$ we obtain that $|S'_{i,4l}|/|\Lambda_{n}|\to\gamma^*_i$ as $n\to\infty$.
	
	Finally, since $\sum_{i\in I^*,i<m_{4l}}|S'_{i,4l}||\mathcal{E}_i|\to\sum_{i\in I^*,i<m_{4l}}\gamma^*_i|\mathcal{E}_i|$ as $n\to\infty$ for every fixed $l$ and since $\sum_{j<m_{4l}}\gamma^*_{j}|\mathcal{E}_{j}|\to\sum_{j\in I^*}\gamma^*_{j}|\mathcal{E}_{j}|=1$ monotonically as $l\to\infty$, we conclude that $\lim\limits_{l\to\infty}\lim\limits_{n\to\infty}\frac{|\Lambda_{n}|-\sum_{i\in I^*,i<m_{4l}}|S'_{i,4l}||\mathcal{E}_i|}{|\Lambda_{n}|}=1$.

\section{Proofs in Sections \ref{sec:spectral} and \ref{sec:Upsilon}}\label{sec:proof2}

\subsection{Proof of Theorem \ref{t1-L}}

First, by $\mathcal{A}^{\Lambda}(a_{n}^\Lambda)$ it suffices to show that for any $g\in \mathbb{C}^{+}_{K}$, $(\Psi_{\tilde{N}^{\Lambda}_{r_{n}}}(g))^{k_{n}}$ converges to (\ref{Psi-L}) as $n\to\infty$. Then, by regular variation of $|\mathbf{X}|$ and the definition of $(a_{n} )$
	\begin{equation*}
	1-\Psi_{\tilde{N}^{\Lambda}_{r_{n}}}(g)
	\leq \mathbb{P}(\max_{\mathbf{t}\in \Lambda_{r_{n}}}|\mathbf{X}_{\mathbf{t}}|>\delta a_{n})\leq \frac{|\Lambda_{r_{n}}|}{|\Lambda_n|}[|\Lambda_n|\mathbb{P}(|\mathbf{X}|>\delta a_{n})]=O(1/k_{n})
	\end{equation*}
	as $n\to\infty$.
	So by Taylor expansion it suffices to prove that $k_{n}(1-\Psi_{\tilde{N}^{\Lambda}_{r_{n}}}(g))$ converges to the logarithm of (\ref{Psi-L}) as $n\to\infty$. Denote $t_{|\Lambda_{r_{n}}|}$ the highest element of $\Lambda_{r_{n}}$ according to $\prec$, by $t_{|\Lambda_{r_{n}}|-1}$ the second highest one,..., by $t_{1}$ the lowest one. Let
	\begin{equation*}
	\tilde{\Psi}_{m}(g)=\begin{cases}
	\mathbb{E}\Big[\exp\Big(-\sum_{j=m}^{|\Lambda_{r_{n}}|}g({a_{n}^\Lambda}^{-1}\mathbf{X}_{t_{j}}) \Big) \Big],& 1\leq m\leq |\Lambda_{r_{n}}|,\\1, & m=|\Lambda_{r_{n}}|+1.
	\end{cases}
	\end{equation*}
	Recall that $K_{l}=\{\mathbf{x}\in\mathbb{Z}^{k}:\mathbf{x}\in\{-l,...,l\}^{k}\}$. Using the stationarity of $\mathbf{X}$ and the fact that $\prec$ is shift invariant, we have
	\begin{align*}
	&\tilde{\Psi}_{m+1}(g)-\tilde{\Psi}_{m}(g)\\=&\mathbb{E}\Big[e^{-\sum_{j=m+1}^{|\Lambda_{r_{n}}|}g({a_{n}^\Lambda}^{-1}\mathbf{X}_{t_{j}-t_{m}})}\Big(1-e^{-g({a_{n}^\Lambda}^{-1}\mathbf{X}_{\mathbf{0}})} \Big)  \Big]\\
	=&\mathbb{E}\Big[e^{-\sum_{\mathbf{t}\in\{t_{m+1}-t_{m},...,t_{|\Lambda_{r_{n}}|}-t_{m}\}\cap K_{l}}g({a_{n}^\Lambda}^{-1}\mathbf{X}_{\mathbf{t}})}\Big(1-e^{-g({a_{n}^\Lambda}^{-1}\mathbf{X}_{\mathbf{0}})} \Big) \mathbf{1}(\hat{M}_{l+1,r_{n}}^{\Lambda,|\mathbf{X}|}\leq\delta a_{n}) \Big]\\
	&+\mathbb{E}\Big[e^{-\sum_{j=m+1}^{|\Lambda_{r_{n}}|}g({a_{n}^\Lambda}^{-1}\mathbf{X}_{t_{j}-t_{m}})}\Big(1-e^{-g({a_{n}^\Lambda}^{-1}\mathbf{X}_{\mathbf{0}})} \Big)  \mathbf{1}(\hat{M}_{l+1,r_{n}}^{\Lambda,|\mathbf{X}|}>\delta a_{n}) \Big]\\
	=&\mathbb{E}\Big[e^{-\sum_{\mathbf{t}\in\{t_{m+1}-t_{m},...,t_{|\Lambda_{r_{n}}|}-t_{m}\}\cap K_{l}}g({a_{n}^\Lambda}^{-1}\mathbf{X}_{\mathbf{t}})}\Big(1-e^{-g({a_{n}^\Lambda}^{-1}\mathbf{X}_{\mathbf{0}})} \Big) \Big]\\
	&-\mathbb{E}\Big[e^{-\sum_{\mathbf{t}\in\{t_{m+1}-t_{m},...,t_{|\Lambda_{r_{n}}|}-t_{m}\}\cap K_{l}}g({a_{n}^\Lambda}^{-1}\mathbf{X}_{\mathbf{t}})}\Big(1-e^{-g({a_{n}^\Lambda}^{-1}\mathbf{X}_{\mathbf{0}})} \Big) \mathbf{1}(\hat{M}_{l+1,r_{n}}^{\Lambda,|\mathbf{X}|}>\delta a_{n})\mathbf{1}(|\mathbf{X}_{\mathbf{0}}|>\delta a_{n}) \Big]\\
	&+\mathbb{E}\Big[e^{-\sum_{j=m+1}^{|\Lambda_{r_{n}}|}g({a_{n}^\Lambda}^{-1}\mathbf{X}_{t_{j}-t_{m}})}\Big(1-e^{-g({a_{n}^\Lambda}^{-1}\mathbf{X}_{\mathbf{0}})} \Big)  \mathbf{1}(\hat{M}_{l+1,r_{n}}^{\Lambda,|\mathbf{X}|}>\delta a_{n})\mathbf{1}(|\mathbf{X}_{\mathbf{0}}|>\delta a_{n}) \Big]\\
	=&\mathbb{E}\Big[e^{-\sum_{\mathbf{t}\in\{t_{m+1}-t_{m},...,t_{|\Lambda_{r_{n}}|}-t_{m}\}\cap K_{l}}g({a_{n}^\Lambda}^{-1}\mathbf{X}_{\mathbf{t}})}\Big(1-e^{-g({a_{n}^\Lambda}^{-1}\mathbf{X}_{\mathbf{0}})} \Big) \Big]+J^{(r_{n})}_{l,m}
	\end{align*}
	where $J^{(r_{n})}_{l,m}$ is such that 
	\begin{equation*}
	\lim\limits_{l\to\infty}\limsup_{n\to\infty}k_{n}\sum_{m=1}^{|\Lambda_{r_{n}}|}|J^{(r_{n})}_{l,m}|\leq 2 \lim\limits_{l\to\infty}\limsup_{n\to\infty}\mathbb{P}(\hat{M}^{\Lambda,|\mathbf{X}|}_{l+1,r_{n}}>\delta a_{n}\,\big| |\mathbf{X}_{\mathbf{0}}|>\delta a_{n})|\Lambda_n|\mathbb{P}(|\mathbf{X}|>\delta a_{n})=0.
	\end{equation*}
	Now, for the every point in $\Lambda_{r_{n}}$ we have that
	\begin{align*}
	&\mathbb{E}\Big[e^{-\sum_{\mathbf{t}\in\{t_{m+1}-t_{m},...,t_{|\Lambda_{r_{n}}|}-t_{m}\}\cap K_{l}}g({a_{n}^\Lambda}^{-1}\mathbf{X}_{\mathbf{t}})}\Big(1-e^{-g({a_{n}^\Lambda}^{-1}\mathbf{X}_{\mathbf{0}})} \Big) \Big]\\
	&=\mathbb{E}\Big[e^{-\sum_{\mathbf{t}\in\{t_{m+1}-t_{m},...,t_{|\Lambda_{r_{n}}|}-t_{m}\}\cap K_{l}}g({a_{n}^\Lambda}^{-1}\mathbf{X}_{\mathbf{t}})}\Big(1-e^{-g({a_{n}^\Lambda}^{-1}\mathbf{X}_{\mathbf{0}})} \Big)\Big||\mathbf{X}_{\mathbf{0}}|>\delta a_{n} \Big]\mathbb{P}(|\mathbf{X}|>\delta a_{n})\\
	&\leq\mathbb{P}(|\mathbf{X}|>\delta a_{n})
	\end{align*}
	To lighten the notation assume that $q$ in Condition ($\mathcal{D}^{\Lambda}$) is $\infty$, so that there are infinitely many $\mathcal{D}$s. By point (ii) in the construction of $\Lambda_{r_{n}}$ only the points $\bigcup_{i=1}^{\infty}\{\mathbf{t}\in\Lambda_{r_{n}}:\Lambda_{r_{n}}^{(\mathbf{t},l)}=\mathcal{D}_{i}\cap K_{l}\}$ are asymptotically relevant, because by (i) and (ii) we have that $|\Lambda_{r_{n}}\setminus \bigcup_{i=1}^{\infty}\{\mathbf{t}\in\Lambda_{r_{n}}:\Lambda_{r_{n}}^{(\mathbf{t},l)}=\mathcal{D}_{i}\cap K_{l}\}|\to 0$.
	
	Observe that there are finitely many different subsets of $K_{l}$. We denote their total number by $\tau_{l}$ and denote them by $\Xi_{l}^{(1)},....,\Xi_{l}^{(\tau_{l})}$. Thus, we have
\begin{align*}
k_{n}\sum_{m=1}^{|\Lambda_{r_{n}}|}&\mathbb{E}\Big[\exp\Big(-\sum_{\mathbf{t}\in\{t_{m+1}-t_{m},...,t_{|\Lambda_{r_{n}}|}-t_{m}\}\cap K_{l}}g({a_{n}^\Lambda}^{-1}\mathbf{X}_{\mathbf{t}}) \Big)\Big(1-e^{-g({a_{n}^\Lambda}^{-1}\mathbf{X}_{\mathbf{0}})} \Big) \Big]\\
&=\sum_{j=1}^{\tau_{l}}k_{n}\mu_{r_{n}}^{(j)}\mathbb{E}\Big[\exp\Big(-\sum_{\mathbf{t}\in\Xi_{l}^{(j)}}g({a_{n}^\Lambda}^{-1}\mathbf{X}_{\mathbf{t}}) \Big)\Big(1-e^{-g({a_{n}^\Lambda}^{-1}\mathbf{X}_{\mathbf{0}})} \Big) \Big]
\end{align*}
where $\mu_{r_{n}}^{(j)}:=|\{t_{m},m=1,...,|\Lambda_{r_{n}}|:\{t_{m+1}-t_{m},...,t_{|\Lambda_{r_{n}}|}-t_{m}\}\cap K_{l}=\Xi_{l}^{(j)}\}|$. Recall from the proof of Point (I) in Proposition \ref{lem-AC1-L-2}, that by defining the equivalence relation $i\sim j  \Leftrightarrow i\in (I_p^{(j)})_{1\le j\le q}$, one considers the partition of $\{1,\ldots,q\}$ generated by the equivalence classes $\mathfrak{P}_p=\{1,\ldots,q\}\setminus \sim $. Then, by (i) and (ii) and in particular by point (I) in Proposition \ref{lem-AC1-L-2} we have
\begin{equation*}
\lim\limits_{n\to\infty}\frac{k_{n}\mu_{r_{n}}^{(j)}}{|\Lambda_{r_{n}}|}=\begin{cases}
\lambda_{i,l}, & \textnormal{if $\Xi_{l}^{(j)}= \mathcal{D}_{i}\cap K_{l}$ for some $i\in\mathfrak{P}_l$,}\\ 0, & \textnormal{otherwise.}
\end{cases}
\end{equation*}
Therefore, we have that 
\begin{align*}
\lim\limits_{n\to\infty}&\sum_{j=1}^{\tau_{l}}k_{n}\mu_{r_{n}}^{(j)}\mathbb{E}\Big[\exp\Big(-\sum_{\mathbf{t}\in\Xi_{l}^{(j)}}g({a_{n}^\Lambda}^{-1}\mathbf{X}_{\mathbf{t}}) \Big)\Big(1-e^{-g({a_{n}^\Lambda}^{-1}\mathbf{X}_{\mathbf{0}})} \Big) \Big]\\
&=\delta^{-\alpha}\sum_{i\in\mathfrak{P}_l}\lambda_{i,l}\mathbb{E}\Big[\exp\Big(-\sum_{\mathbf{t}\in\mathcal{D}_{i}\cap K_{l}}g(\delta \mathbf{Y}_{\mathbf{t}}) \Big)\Big(1-e^{-g(\delta \mathbf{Y}_{\mathbf{0}})} \Big) \Big]\\
&=\delta^{-\alpha}\sum_{i=1}^{\infty}\lambda_{i}\mathbb{E}\Big[\exp\Big(-\sum_{\mathbf{t}\in\mathcal{D}_{i}\cap K_{l}}g(\delta \mathbf{Y}_{\mathbf{t}}) \Big)\Big(1-e^{-g(\delta \mathbf{Y}_{\mathbf{0}})} \Big) \Big]\\
&=\int_{0}^{\infty}\sum_{i=1}^{\infty}\lambda_{i}\mathbb{E}\Big[\exp\Big(-\sum_{\mathbf{t}\in\mathcal{D}_{i}\cap K_{l}}g(y\mathbf{\Theta}_{\mathbf{t}}) \Big)\Big(1-e^{-g(y\mathbf{\Theta}_{\mathbf{0}})} \Big) \Big]d(-y^{-\alpha}).
\end{align*}
Notice that the above arguments hold for every $l$ large enough. By monotone convergence theorem we have that
\begin{align*}
\lim\limits_{l\to\infty}&\int_{0}^{\infty}\sum_{i=1}^{\infty}\lambda_{i}\mathbb{E}\Big[\exp\Big(-\sum_{\mathbf{t}\in\mathcal{D}_{i}\cap K_{l}}g(y\mathbf{\Theta}_{\mathbf{t}}) \Big)\Big(1-e^{-g(y\mathbf{\Theta}_{\mathbf{0}})} \Big) \Big]d(-y^{-\alpha})\\
&=\int_{0}^{\infty}\sum_{i=1}^{\infty}\lambda_{i}\mathbb{E}\Big[\exp\Big(-\sum_{\mathbf{t}\in\mathcal{D}_{i}}g(y\mathbf{\Theta}_{\mathbf{t}}) \Big)\Big(1-e^{-g(y\mathbf{\Theta}_{\mathbf{0}})} \Big) \Big]d(-y^{-\alpha}).
\end{align*}
Finally, the existence of the limiting random measure $N^{\Lambda}$ is ensured by Corollary 4.14 in \cite{Kallenberg2}.

\subsection{Proof of Proposition \ref{lem-bound-L}}

In order to prove Proposition \ref{lem-bound-L} we need the following Lemma.
\begin{lemma}\label{lem-AC1-L}
	Let $(\mathbf Y_{\mathbf{t}}:\mathbf{t}\in\mathbb{Z}^{k})$ be an $\mathbb{R}^{d}$-valued random field such that the time change formula (\ref{timechangeY}) is satisfied. Let $\mathbf{\Theta}_{\mathbf{t}}=\mathbf Y_{\mathbf{t}}/|{\mathbf Y}_{\mathbf{0}}|$, $\mathbf{t}\in\mathbb{Z}^{k}$. Let $\Upsilon$ be a subset of $\{\mathbf{t}\in\mathbb{Z}^{k}:\mathbf{t}\succeq\mathbf{0}\}$ containing $\{\mathbf{0}\}$ and assume that $\Upsilon\cup-\Upsilon$ is translation invariant along the points of a (not necessarily full rank) lattice. Then $|\mathbf{\Theta}_{\mathbf{t}}|\to0$ a.s.~as $|\mathbf{t}|\to\infty$ for $\mathbf{t}\in\Upsilon$ implies that $\sum_{\mathbf{t}\in\Upsilon\cup-\Upsilon}|\mathbf{\Theta}_{\mathbf{t}}|^{\alpha}<\infty$ a.s..
\end{lemma}
\begin{proof}
	The proof is divided in two parts. In the first part we show that  $|\mathbf{\Theta}_{\mathbf{t}}|\to0$ a.s.~as $|\mathbf{t}|\to\infty$ for $\mathbf{t}\in\Upsilon\cup-\Upsilon$ and then that $\sum_{\mathbf{t}\in\Upsilon\cup-\Upsilon}|\mathbf{\Theta}_{\mathbf{t}}|^{\alpha}<\infty$ a.s.

	Denote by $\mathcal{L}$ the lattice and let $\mathcal{G}:=\mathcal{L}\cap\{\mathbf{t}\in\mathbb{Z}^{k}:\mathbf{t}\succeq\mathbf{0}\}$. We stress that $\{\mathbf{0}\}\in\mathcal{G}$. Let $\epsilon>0$. Suppose that $\mathbb{P}(\sum_{\mathbf{h}\in-\Upsilon}\mathbf{1}(|\mathbf{Y}_{\mathbf{h}}|>\epsilon)=\infty)>0$. Recall that $|\mathbf{Y}_{\mathbf{0}}|$ follows a Pareto($\alpha$) distribution, thus $\mathbb{P}(|\mathbf{Y}_{\mathbf{0}}|\geq 1)=1$, and observe that the sets $\Big\{|\mathbf{Y}_{\mathbf{t}}|\geq C>\sup\limits_{\mathbf{t}\prec\mathbf{s},\mathbf{s}\in\mathcal{G}}|\mathbf{Y}_{\mathbf{s}}|\Big\}$, $\mathbf{t}\in\mathcal{G}$, are disjoint for every $C>0$. Then, we have that for every $0<D\leq1$
\begin{equation*}
\mathbb{P}\bigg(\bigcup_{\mathbf{t}\in\mathcal{G}}\Big\{|\mathbf{Y}_{\mathbf{t}}|\geq D>\sup\limits_{\mathbf{t}\prec\mathbf{s},\mathbf{s}\in\mathcal{G}}|\mathbf{Y}_{\mathbf{s}}|\Big\}\bigg)=\sum_{\mathbf{t}\in\mathcal{G}}\mathbb{P}\Big(|\mathbf{Y}_{\mathbf{t}}|\geq D>\sup\limits_{\mathbf{t}\prec\mathbf{s},\mathbf{s}\in\mathcal{G}}|\mathbf{Y}_{\mathbf{s}}|\Big)= 1,
\end{equation*}
and for  every $D'>1$
\begin{equation*}
\mathbb{P}\bigg(\bigcup_{\mathbf{t}\in\mathcal{G}}\Big\{|\mathbf{Y}_{\mathbf{t}}|\geq D'>\sup\limits_{\mathbf{t}\prec\mathbf{s},\mathbf{s}\in\mathcal{G}}|\mathbf{Y}_{\mathbf{s}}|\Big\}\bigg)=\sum_{\mathbf{t}\in\mathcal{G}}\mathbb{P}\Big(|\mathbf{Y}_{\mathbf{t}}|\geq D'>\sup\limits_{\mathbf{t}\prec\mathbf{s},\mathbf{s}\in\mathcal{G}}|\mathbf{Y}_{\mathbf{s}}|\Big)\leq 1.
\end{equation*}
 we have that $\mathbb{P}(\sum_{\mathbf{h}\in-\Upsilon}\mathbf{1}(|\mathbf{Y}_{\mathbf{h}}|>\epsilon)=\infty)=\sum_{\mathbf{t}\in\mathcal{G}}\mathbb{P}(\sum_{\mathbf{h}\in-\Upsilon}\mathbf{1}(|\mathbf{Y}_{\mathbf{h}}|>\epsilon)=\infty,|\mathbf{Y}_{\mathbf{t}}|\geq1>\sup\limits_{\mathbf{t}\prec\mathbf{s},\mathbf{s}\in\mathcal{G}}|\mathbf{Y}_{\mathbf{s}}|)$. Consider any $\mathbf{t}\in\mathcal{G}$ s.t.~$\mathbb{P}(\sum_{\mathbf{h}\in-\Upsilon}\mathbf{1}(|\mathbf{Y}_{\mathbf{h}}|>\epsilon)=\infty,|\mathbf{Y}_{\mathbf{t}}|\geq1>\sup\limits_{\mathbf{t}\prec\mathbf{s},\mathbf{s}\in\mathcal{G}}|\mathbf{Y}_{\mathbf{s}}|)>0$. By the time change formula (\ref{timechangeY})  we get
	\begin{align*}
	\infty&=\mathbb{E}\bigg[\sum_{\mathbf{h}\in-\Upsilon}\mathbf{1}(|\mathbf{Y}_{\mathbf{h}}|>\epsilon,|\mathbf{Y}_{\mathbf{t}}|\geq1>\sup\limits_{\mathbf{t}\prec\mathbf{s},\mathbf{s}\in\mathcal{G}}|\mathbf{Y}_{\mathbf{s}}|)\bigg]\\
&=\sum_{\mathbf{h}\in-\Upsilon}\mathbb{E}\Big[\mathbf{1}(|\mathbf{Y}_{\mathbf{h}}|>\epsilon,|\mathbf{Y}_{\mathbf{t}}|\geq1>\sup\limits_{\mathbf{t}\prec\mathbf{s},\mathbf{s}\in\mathcal{G}}|\mathbf{Y}_{\mathbf{s}}|)\Big]\\
&=\sum_{\mathbf{h}\in-\Upsilon}\mathbb{P}\Big(|\mathbf{Y}_{\mathbf{h}}|>\epsilon,|\mathbf{Y}_{\mathbf{t}}|\geq1>\sup\limits_{\mathbf{t}\prec\mathbf{s},\mathbf{s}\in\mathcal{G}}|\mathbf{Y}_{\mathbf{s}}|\Big)\\
&=\sum_{\mathbf{h}\in-\Upsilon}\int_{\epsilon}^{\infty}\mathbb{P}\Big(r|\mathbf{\Theta}_{-\mathbf{h}}|>1,r|\mathbf{\Theta}_{\mathbf{t}-\mathbf{h}}|\geq1>r\sup\limits_{\mathbf{t}-\mathbf{h}\prec\mathbf{s},\mathbf{s}\in\mathcal{G}}|\mathbf{\Theta}_{\mathbf{s}}|\Big)d(-r^{-\alpha})\\
&\stackrel{(r=q\epsilon)}{=}\epsilon^{-\alpha}\sum_{\mathbf{h}\in-\Upsilon}\int_{1}^{\infty}\mathbb{P}\Big(q\epsilon|\mathbf{\Theta}_{-\mathbf{h}}|>1,q\epsilon|\mathbf{\Theta}_{\mathbf{t}-\mathbf{h}}|\geq 1>q\epsilon\sup\limits_{\mathbf{t}-\mathbf{h}\prec\mathbf{s},\mathbf{s}\in\mathcal{G}}|\mathbf{\Theta}_{\mathbf{s}}|\Big)d(-q^{-\alpha})
\\
	&\leq\epsilon^{-\alpha}\int_{1}^{\infty}\sum_{\mathbf{h}\in-\Upsilon}\mathbb{P}\Big(|\mathbf{\Theta}_{\mathbf{t}-\mathbf{h}}|\geq\frac{1}{q\epsilon}>\sup\limits_{\mathbf{t}-\mathbf{h}\prec\mathbf{s},\mathbf{s}\in\mathcal{G}_{-\mathbf{h}}}|\mathbf{\Theta}_{\mathbf{s}}|\Big)d(-q^{-\alpha})\\
	&\leq b\epsilon^{-\alpha}\int_{1}^{\infty}d(-q^{-\alpha})=\epsilon^{-\alpha}<\infty,
	\end{align*}
	where $b$ is the number of points of $\Upsilon$ inside the fundamental parallelotope of $\mathcal{L}$. Notice that we used that for every $\mathbf{t}\in\mathcal{G}$ and $\mathbf{h}\in-\Upsilon$ (\textit{i.e.}~$-\mathbf{h}\in\Upsilon$) we have that $\mathbf{t}-\mathbf{h}\in\Upsilon$, that is $\mathbf{t}-\mathbf{h}\in(\mathcal{G})_{\mathbf{x}}$ where $\mathbf{x}$ is one of the $b$ different points in the fundamental parallelotope, which we denote by $\hat{B}$ to be consistent with the notation of the proof of Proposition \ref{lem-AC1-L-2}. Thus, we have a contradiction and so $|\mathbf{\Theta}_{\mathbf{t}}|\to0$ a.s.~as $|\mathbf{t}|\to\infty$ for $\mathbf{t}\in\Upsilon\cup-\Upsilon$.

	Now, suppose that the event $\{|\mathbf{\Theta}_{\mathbf{t}}|\to 0\textnormal{ as }|\mathbf{t}|\to\infty,\mathbf{t}\in\Upsilon\cup-\Upsilon\} $ has probability 1. Denote this event by $E$. Observe that $\sup_{\mathbf{t}\in\Upsilon\cup-\Upsilon}|\mathbf{\Theta}_{\mathbf{t}}|$ is a well defined random variable since it is the supremum of measurable functions over a countable set. Since $|\mathbf{\Theta}_{\mathbf{t}}|\to0$ a.s.~as $|\mathbf{t}|\to\infty$ for $\mathbf{t}\in\Upsilon\cup-\Upsilon$ and since we are in (a subset of) $\mathbb{Z}^{k}$, for every $\omega\in E$ there exist finitely many $t_{1},...,t_{m}\in\Upsilon\cup-\Upsilon$ such that $|\mathbf{\Theta}(t_{1})(\omega)|=...=|\mathbf{\Theta}(t_{m})(\omega)|=\sup_{\mathbf{t}\in\Upsilon\cup-\Upsilon}|\mathbf{\Theta}_{\mathbf{t}}(\omega)|$. For every $\omega\in E$, let $T^{*}(\omega)$ be such that $|\mathbf{\Theta}(T^{*}(\omega))(\omega)|=\sup_{\mathbf{t}\in\Upsilon\cup-\Upsilon}|\mathbf{\Theta}_{\mathbf{t}}(\omega)|$ with $T^{*}(\omega)$ being the smallest of these finitely many points according to $\succ$. That is for every $\mathbf{t}\in\Upsilon\cup-\Upsilon$ we have
	\begin{align*}
	&\{\omega:T^{*}(\omega)=\mathbf{t}\}\\
	 &=\{\omega:\mathbf{\Theta}_{\mathbf{t}}(\omega)-\sup_{\mathbf{s}\in\Upsilon\cup-\Upsilon,\mathbf{s}\prec\mathbf{t}}|\mathbf{\Theta}_{\mathbf{s}}(\omega)|>0\}\cap \{\omega:\mathbf{\Theta}_{\mathbf{t}}(\omega)-\sup_{\mathbf{s}\in\Upsilon\cup-\Upsilon,\mathbf{s}\succeq\mathbf{t}}|\mathbf{\Theta}_{\mathbf{s}}(\omega)|=0\}
	\end{align*}
	and for $\mathbf{t}\in\mathbb{Z}^{k}\setminus(\Upsilon\cup-\Upsilon)$ we have $\{\omega:T^{*}(\omega)=\mathbf{t}\}=\emptyset$.
	
	By construction $|\mathbf{\Theta}_{T^{*}}|$ is a measurable function.	Since the difference of two measurable functions is measurable and the intersection of two measurable sets is also measurable we have that $\{\omega:T^{*}(\omega)=\mathbf{t}\}$ is a measurable set. Further, for any subset $A$ of $\mathbb{Z}^{k}$, since $(T^{*})^{-1}(A)=\cup_{\mathbf{t}\in A}	\{\omega:T^{*}(\omega)=\mathbf{t}\}$ and since the union of measurable sets is measurable we have that $(T^{*})^{-1}(A)$ is measurable. Thus, $T^{*}$ is a well defined random variable. Using the same arguments we can construct $T_{\mathcal{L}}^{*}$, where the supremum is taken over $\mathcal{L}$ instead of $\Upsilon\cup-\Upsilon$. In the same way we can construct $T_{(\mathcal{L})_{\mathbf{x}}}^{*}$ where the supremum is taken over $(\mathcal{L})_{\mathbf{x}}$ where $\mathbf{x}\in\hat{B}$.
	
	Consider any $\mathbf{x}\in\hat{B}$. Assume that $\mathbb{P}(\sum_{\mathbf{t}\in(\mathcal{L})_{\mathbf{x}}}|\mathbf{\Theta}_{\mathbf{t}}|^{\alpha}=\infty)>0$. We have that $\mathbb{P}(\sum_{\mathbf{t}\in(\mathcal{L})_{\mathbf{x}}}|\mathbf{\Theta}_{\mathbf{t}}|^{\alpha}=\infty)=\sum_{\mathbf{i}\in\mathcal{L}}\mathbb{P}(\sum_{\mathbf{t}\in(\mathcal{L})_{\mathbf{x}}}|\mathbf{\Theta}_{\mathbf{t}}|^{\alpha}=\infty,T^{*}_{\mathcal{L}}=\mathbf{i})=\sum_{\mathbf{i}\in H}\mathbb{P}(\sum_{\mathbf{t}\in(\mathcal{L})_{\mathbf{x}}}|\mathbf{\Theta}_{\mathbf{t}}|^{\alpha}=\infty,T^{*}_{\mathcal{L}}=\mathbf{i})$, where $H$ is the subset of $\mathcal{L}$ s.t.~$\mathbb{P}(\sum_{\mathbf{t}\in(\mathcal{L})_{\mathbf{x}}}|\mathbf{\Theta}_{\mathbf{t}}|^{\alpha}=\infty,T^{*}_{\mathcal{L}}=\mathbf{i})>0$ for every $\mathbf{i}\in H$.  Let $\mathbf{i}\in H$, then
	\begin{equation*}
	\infty=\mathbb{E}\bigg[\sum_{\mathbf{t}\in(\mathcal{L})_{\mathbf{x}}}|\mathbf{\Theta}_{\mathbf{t}}|^{\alpha}\mathbf{1}(T^{*}_{\mathcal{L}}=\mathbf{i})\bigg]=\sum_{\mathbf{t}\in(\mathcal{L})_{\mathbf{x}}}\mathbb{E}\bigg[|\mathbf{\Theta}_{\mathbf{t}}|^{\alpha}\mathbf{1}(T_{\mathcal{L}}^{*}=\mathbf{i})\bigg].
	\end{equation*}
	Now, we generalise the arguments adopted in the proof of Lemma 3.3 in \cite{SW}. For each $\mathbf{i}\in\mathcal{L}$ define a function $g_{\mathbf{i}}:(\bar{\mathbb{R}}^{d})^{\mathbb{Z}^{k}}\to\mathbb{R}$ as follows. If $(\mathbf{\Theta}_{\mathbf{s}}, \mathbf{s}\in\mathbb{Z}^{k})$ is such that
	\begin{equation*}
|\mathbf{\Theta}_{\mathbf{j}}|<|\mathbf{\Theta}_{\mathbf{i}}|\quad\textnormal{for $\mathbf{j}\prec\mathbf{i}$ and $\mathbf{j}\in\mathcal{L}$},\quad |\mathbf{\Theta}_{\mathbf{j}}|\leq|\mathbf{\Theta}_{\mathbf{i}}|\quad\textnormal{for $\mathbf{j}\succeq\mathbf{i}$ and $\mathbf{j}\in\mathcal{L}$},
	\end{equation*}
	then set $g_{\mathbf{i}}(\mathbf{\Theta}_{\mathbf{s}}, \mathbf{s}\in\mathbb{Z}^{k})=1$. Otherwise set $g_{\mathbf{i}}(\mathbf{\Theta}_{\mathbf{s}}, \mathbf{s}\in\mathbb{Z}^{k})=0$. Observe that $g_{\mathbf{i}}$ is a
	bounded measurable function and observe that, for any $\mathbf{t}\in\mathbb{Z}^{k}$, $g_{\mathbf{i}}(\cdot-\mathbf{t})=1$ when
	\begin{equation*}
	|\mathbf{\Theta}_{\mathbf{j}-\mathbf{t}}|<|\mathbf{\Theta}_{\mathbf{i}-\mathbf{t}}|\quad\textnormal{for $\mathbf{j}\prec\mathbf{i}$ and $\mathbf{j}\in\mathcal{L}$},\quad |\mathbf{\Theta}_{\mathbf{j}-\mathbf{t}}|\leq|\mathbf{\Theta}_{\mathbf{i}-\mathbf{t}}|\quad\textnormal{for $\mathbf{j}\succeq\mathbf{i}$ and $\mathbf{j}\in\mathcal{L}$},
	\end{equation*}
	\begin{equation*}
	\Leftrightarrow|\mathbf{\Theta}_{\mathbf{j}}|<|\mathbf{\Theta}_{\mathbf{i}-\mathbf{t}}|\quad\textnormal{for $\mathbf{j}\prec\mathbf{i}-\mathbf{t}$ and $\mathbf{j}\in(\mathcal{L})_{-\mathbf{t}}$},\quad |\mathbf{\Theta}_{\mathbf{j}}|\leq|\mathbf{\Theta}_{\mathbf{i}-\mathbf{t}}|\quad\textnormal{for $\mathbf{j}\succeq\mathbf{i}-\mathbf{t}$ and $\mathbf{j}\in(\mathcal{L})_{-\mathbf{t}}$},
	\end{equation*}
	and zero otherwise. Then, by time change formula we have
	\begin{align*}
	\infty&=\sum_{\mathbf{t}\in(\mathcal{L})_{\mathbf{x}}}\mathbb{E}\bigg[|\mathbf{\Theta}_{\mathbf{t}}|^{\alpha}\mathbf{1}(T_{\mathcal{L}}^{*}=\mathbf{i})\bigg]=\sum_{\mathbf{t}\in(\mathcal{L})_{\mathbf{x}}}\mathbb{E}\bigg[|\mathbf{\Theta}_{\mathbf{t}}|^{\alpha}g_{\mathbf{i}}(\mathbf{\Theta}_{\mathbf{s}}, \mathbf{s}\in\mathbb{Z}^{k})\bigg]\\
	&=\sum_{\mathbf{t}\in(\mathcal{L})_{\mathbf{x}}}\mathbb{E}\bigg[g_{\mathbf{i}}(\mathbf{\Theta}(\mathbf{s}-\mathbf{t}), \mathbf{s}\in\mathbb{Z}^{k})\mathbf{1}(\mathbf{\Theta}(-\mathbf{t})\neq\mathbf{0})\bigg]\leq \sum_{\mathbf{t}\in(\mathcal{L})_{\mathbf{x}}}\mathbb{E}\bigg[g_{\mathbf{i}}(\mathbf{\Theta}(\mathbf{s}-\mathbf{t}), \mathbf{s}\in\mathbb{Z}^{k})\bigg]\\
	&=\sum_{\mathbf{t}\in(\mathcal{L})_{\mathbf{x}}}\mathbb{E}\bigg[\mathbf{1}(\mathbf T_{(\mathcal{L})_{-t}}^{*}=\mathbf{i}-\mathbf{t})\bigg]=\sum_{\mathbf{t}\in(\mathcal{L})_{\mathbf{x}}}\mathbb{E}\bigg[\mathbf{1}(\mathbf T_{(\mathcal{L})_{-x}}^{*}=\mathbf{i}-\mathbf{t})\bigg]=\sum_{\mathbf{t}\in(\mathcal{L})_{-\mathbf{x}}}\mathbb{E}\bigg[\mathbf{1}(\mathbf T_{(\mathcal{L})_{-x}}^{*}=\mathbf{i}+\mathbf{t})\bigg]\\
	&=\sum_{\mathbf{t}\in(\mathcal{L})_{\mathbf{r}-\mathbf{x}}}\mathbb{E}\bigg[\mathbf{1}(\mathbf T_{(\mathcal{L})_{\mathbf{r}-x}}^{*}=\mathbf{i}+\mathbf{t})\bigg]=\sum_{\mathbf{t}\in(\mathcal{L})_{\mathbf{r}-\mathbf{x}}}\mathbb{E}\bigg[\mathbf{1}(\mathbf T_{(\mathcal{L})_{\mathbf{r}-x}}^{*}=\mathbf{t})\bigg]=1,
	\end{align*}
	which is a contradiction. Notice that we used the fact that by construction, for every $\mathbf{t}\in(\mathcal{L})_{\mathbf{x}}$, we have $(\mathcal{L})_{-t}=(\mathcal{L})_{-x}$, $-\mathbf{t}\in(\mathcal{L})_{-\mathbf{x}}$, $(\mathcal{L})_{-\mathbf{x}}=(\mathcal{L})_{\mathbf{r}-\mathbf{x}}$ where $\mathbf{r}$ is the highest point in the closure of the fundamental parallelotope (as defined in the proof of Proposition \ref{lem-AC1-L-2}), and that $\mathbf{t}+\mathbf{i}\in (\mathcal{L})_{\mathbf{r}-\mathbf{x}}$ for any $\mathbf{i}\in\mathcal{L}$. 
	
	Thus, we have $\sum_{\mathbf{t}\in(\mathcal{L})_{\mathbf{x}}}|\mathbf{\Theta}_{\mathbf{t}}|^{\alpha}<\infty$ almost surely. The same arguments can be repeated for every $\mathbf{x}\in\hat{B}$ and use the fact proven in the proof of Proposition \ref{lem-AC1-L-2} that for every $\mathbf{x}\in\hat{B}$ we know that $\mathbf{r}-\mathbf{x}\in\hat{B}$. Therefore, since $\hat{B}$ is finite we conclude that $\sum_{\mathbf{t}\in\Upsilon\cup-\Upsilon}|\mathbf{\Theta}_{\mathbf{t}}|^{\alpha}=\sum_{\mathbf{x}\in\hat{B}}\sum_{\mathbf{t}\in(\mathcal{L})_{\mathbf{x}}}|\mathbf{\Theta}_{\mathbf{t}}|^{\alpha}<\infty$ a.s..
\end{proof}

	We first prove that for every $\mathbf{t}\in\bigcup_{j=1}^{\infty}\mathcal{D}_{j}$ there exists a $n_{\mathbf{t}}\in\mathbb{N}$ s.t.~$\mathbf{t}\in R_{0,\Lambda_{m}}$ for every $m>n_{\mathbf{t}}$. First, notice that if $\mathbf{t}\in \bigcup_{j=1}^{\infty}\mathcal{D}_{j}$ then $\mathbf{t}\in \mathcal{D}_{i}$ for some $i\in\mathbb{N}$. Then, by condition point (I) in Proposition \ref{lem-AC1-L-2} for every $i\in\mathbb{N}$ and every $p\in\mathbb{N}$ there exists an $n_{i,p}^{*}\in\mathbb{N}$ such that for every $m>n_{i,p}^{*}$ we have $|\{\mathbf{t}\in\Lambda_{m}:\Lambda_{m}^{(\mathbf{t},p)}=\mathcal{D}_{i}\cap K_{p}\}|\geq 1$. Thus, for every $i,p\in\mathbb{N}$ there exists an $n^{*}_{i,p}$ s.t.~$\mathcal{D}_{i}\cap K_{p}\subset R_{0,\Lambda_{m}}$ for every $m>n^{*}_{i,p}$. Therefore, for every $\mathbf{t}\in\bigcup_{j=1}^{\infty}\mathcal{D}_{j}$ (notice that for each $\mathbf{t}$ we have $\mathbf{t}\in\mathcal{D}_{i}\cap K_{p}$ for some $i,p\in\mathbb{N}$) there exists a $n_{\mathbf{t}}\in\mathbb{N}$ (namely $n^{*}_{i,p}$) s.t.~$\mathbf{t}\in R_{0,\Lambda_{m}}$ for every $m>n_{\mathbf{t}}$. 
	
	Now, choose $(d_{n})_{n\in\mathbb{N}}$ such that $d_{n}$ is the highest integer s.t.~$\max_{|\mathbf{t}|\leq d_{n},\mathbf{t}\in\bigcup_{j=1}^{\infty}\mathcal{D}_{j}}n_{\mathbf{t}}< r_{n}$. Notice that $\{|\mathbf{t}|\leq d_{n},\mathbf{t}\in\bigcup_{j=1}^{\infty}\mathcal{D}_{j}\}\subset R_{0,\Lambda_{r_{n}}}$. It is possible to see that $d_{n}\to\infty$ as $n\to\infty$ and that for every $l<r_{n}$
	\begin{align*}
	\mathbb{P}&\Big(\max_{l\leq |\mathbf{t}|\leq d_{n}\,|\,\mathbf{t}\in \bigcup_{j=1}^{\infty}\mathcal{D}_{j} }|\mathbf{X}_{\mathbf{t}}|>a_{n}^\Lambda x\,\big||\mathbf{X}_{\mathbf{0}}|>a_{n}^\Lambda x\Big)\\
	&=\mathbb{P}\Big(\max_{l\leq |\mathbf{t}|\leq d_{n}\,|\,\mathbf{t}\in \bigcup_{j=1}^{\infty}\mathcal{D}_{j} \,|\,\mathbf{t}\in R_{l,\Lambda_{r_{n}}} }|\mathbf{X}_{\mathbf{t}}|>a_{n}^\Lambda x\,\big||\mathbf{X}_{\mathbf{0}}|>a_{n}^\Lambda x\Big)\\
	&\leq\mathbb{P}\Big(\max_{\mathbf{t}\in R_{l,\Lambda_{r_{n}}} }|\mathbf{X}_{\mathbf{t}}|>a_{n}^\Lambda x\,\big||\mathbf{X}_{\mathbf{0}}|>a_{n}^\Lambda x\Big).
	\end{align*}	
	Therefore, condition (AC$^{\Lambda}_{\succeq}$) implies the following anti-clustering condition:
	\begin{equation}\label{AC-UionionD}
	\lim\limits_{l\to\infty}\limsup_{n\to\infty}\mathbb{P}\Big(\max_{l\leq |\mathbf{t}|\leq d_{n}\,|\,\mathbf{t}\in \bigcup_{j=1}^{\infty}\mathcal{D}_{j} }|\mathbf{X}_{\mathbf{t}}|>a_{n}^\Lambda x\,\big||\mathbf{X}_{\mathbf{0}}|>a_{n}^\Lambda x\Big)=0.
\end{equation}
	
	Now, for any $z>0$, by the regular variation of $|\mathbf{X}_{\mathbf{0}}|$ and by (\ref{AC-UionionD}) we have that 
	\begin{equation*}
	\lim\limits_{l\to\infty}\limsup_{n\to\infty}\mathbb{P}\bigg(\max_{l\leq |\mathbf{t}|\leq d_{n}\,|\,\mathbf{t}\in \bigcup_{j=1}^{\infty}\mathcal{D}_{j} }|\mathbf{X}_{\mathbf{t}}|>za_{n}^\Lambda x\,\big||\mathbf{X}_{\mathbf{0}}|>a_{n}^\Lambda x\bigg)=0.
	\end{equation*}
	In other words, for any $\epsilon>0$ and $z>0$, there exists $l>0$ such that for all $w>l$
	\begin{equation*}
	\mathbb{P}\bigg(\max_{l\leq |\mathbf{t}|\leq w\,|\,\mathbf{t}\in \bigcup_{j=1}^{\infty}\mathcal{D}_{j}  }|\mathbf{Y}_{\mathbf{t}}|>z\bigg)\leq \epsilon.
	\end{equation*}
	This, implies that $\mathbb{P}(\lim\limits_{|\mathbf{t}|\to\infty}|\mathbf{Y}_{\mathbf{t}}|=0)=1$ and so $\mathbb{P}(\lim\limits_{|\mathbf{t}|\to\infty}|\mathbf{\Theta}_{\mathbf{t}}|=0)=1$ for $\mathbf{t}\in \bigcup_{j=1}^{\infty}\mathcal{D}_{j}$. Then, from Lemmas \ref{lem-AC1-L-2} and \ref{lem-AC1-L} we obtain the statement.

\subsection{Proof of Theorem \ref{t2-L}}
By changes of variables we have
	\begin{align*}
	&\int_{0}^{\infty}\mathbb{E}\bigg[e^{-\sum_{\mathbf{t}\in\mathcal{D}_{j}}g(y\mathbf{\Theta}_{\mathbf{t}})}\Big(1-e^{-g(y\mathbf{\Theta}_{\mathbf{0}})} \Big) \bigg]d(-y^{-\alpha})\\
	&=\mathbb{E}\Big[\|\mathbf{\Theta}\|_{\mathcal{D}_{j}\cup-\mathcal{D}_{j},\alpha}^{\alpha}\int_{0}^{\infty}e^{-\sum_{\mathbf{t}\in\mathcal{D}_{j}}g(y\mathbf{\Theta}_{\mathbf{t}}/\|\mathbf{\Theta}\|_{\mathcal{D}_{j}\cup-\mathcal{D}_{j},\alpha})}\Big(1-e^{-g(y\mathbf{\Theta}_{\mathbf{0}}/\|\mathbf{\Theta}\|_{\mathcal{D}_{j}\cup-\mathcal{D}_{j},\alpha})} \Big) d(-y^{-\alpha})\Big]\\
	&=\sum_{\mathbf{h}\in\mathcal{D}_{j}\cup-\mathcal{D}_{j}}\mathbb{E}\Big[|\mathbf{\Theta}_{\mathbf{h}}|^{\alpha}\int_{0}^{\infty}e^{-\sum_{\mathbf{t}\in\mathcal{D}_{j}}g(y\mathbf{\Theta}_{\mathbf{t}}/\|\mathbf{\Theta}\|_{\mathcal{D}_{j}\cup-\mathcal{D}_{j},\alpha})}\Big(1-e^{-g(y\mathbf{\Theta}_{\mathbf{0}}/\|\mathbf{\Theta}\|_{\mathcal{D}_{j}\cup-\mathcal{D}_{j},\alpha})} \Big) d(-y^{-\alpha})\Big]
	\end{align*}
	From the time-change formula, we obtain that for any $\mathbf{h}\in\mathcal{D}_{j}\cup-\mathcal{D}_{j}$,
	\begin{align*}
	\mathbb{E}&\Big[|\mathbf{\Theta}_{\mathbf{h}}|^{\alpha}\int_{0}^{\infty}e^{-\sum_{\mathbf{t}\in\mathcal{D}_{j}}g(y\mathbf{Q}_{\mathcal{D}_{j}\cup-\mathcal{D}_{j},\mathbf{t}}) }- e^{-\sum_{\mathbf{t}\in\mathcal{D}_{j}\cup\{\mathbf{0}\}}g(y\mathbf{Q}_{\mathcal{D}_{j}\cup-\mathcal{D}_{j},\mathbf{t}})} d(-y^{-\alpha})\Big]\\
	&=\mathbb{E}\Big[\int_{0}^{\infty}e^{-\sum_{\mathbf{t}\in(\mathcal{D}_{j})_{-\mathbf{h}}}g(y\mathbf{Q}_{\mathcal{D}_{j}\cup-\mathcal{D}_{j},\mathbf{t}})}- e^{-\sum_{\mathbf{t}\in(\mathcal{D}_{j})_{-\mathbf{h}}\cup\{-\mathbf{h}\}}g(y\mathbf{Q}_{\mathcal{D}_{j}\cup-\mathcal{D}_{j},\mathbf{t}}) }d(-y^{-\alpha})\Big].
	\end{align*}
	Since $\mathbf{h}$ is a lattice point then $(\mathcal{D}_{j})_{\mathbf{-h}}=(\mathcal{D}_{j}\cup-\mathcal{D}_{j})\cap\{\mathbf{t}\in\mathbb{Z}^{k}:\mathbf{t}\succ-\mathbf{h}\}$ and since $-\mathbf{h}$ is the first point of $(\mathcal{D}_{j}\cup-\mathcal{D}_{j})\cap\{\mathbf{t}\in\mathbb{Z}^{k}:\mathbf{t}\succeq-\mathbf{h}\}$.	This leads to a telescoping sum structure for any $\mathbf{k}\in\mathcal{D}_{j}\cup-\mathcal{D}_{j}$
	\begin{align}
	&\sum_{\{-\mathbf{h}\in\mathcal{D}_{j}\cup-\mathcal{D}_{j}:-\mathbf{k}\preceq\mathbf{h}\preceq\mathbf{k}\}}\mathbb{E}\Big[\int_{0}^{\infty}e^{-\sum_{\mathbf{t}\in(\mathcal{D}_{j})_{-\mathbf{h}}}g(y\mathbf{Q}_{\mathcal{D}_{j}\cup-\mathcal{D}_{j},\mathbf{t}})}- e^{-\sum_{\mathbf{t}\in(\mathcal{D}_{j})_{-\mathbf{h}}\cup\{-\mathbf{h}\}}g(y\mathbf{Q}_{\mathcal{D}_{j}\cup-\mathcal{D}_{j},\mathbf{t}})} d(-y^{-\alpha})\Big]\nonumber\\
	&\label{Q-L}
	=\mathbb{E}\Big[\int_{0}^{\infty}e^{-\sum_{\mathbf{t}\in(\mathcal{D}_{j})_{\mathbf{k}}}g(y\mathbf{Q}_{\mathcal{D}_{j}\cup-\mathcal{D}_{j},\mathbf{t}})}- e^{-\sum_{\mathbf{t}\in(\mathcal{D}_{j})_{-\mathbf{k}}\cup\{-\mathbf{k}\}}g(y\mathbf{Q}_{\mathcal{D}_{j}\cup-\mathcal{D}_{j},\mathbf{t}})}d(-y^{-\alpha})\Big].
	\end{align}
	Since any function $g\in\mathbb{C}^{+}_{K}$ vanishes in some neighbourhood of the origin and $\mathbf{\Theta}_{\mathbf{t}}\stackrel{a.s}{\to}0$ and $\mathbf{Q}_{\mathcal{D}_{j}\cup-\mathcal{D}_{j},\mathbf{t}}\stackrel{a.s}{\to}0$ as $\mathbf{t}\to\boldsymbol{\infty}$ for $\mathbf{t}\in\mathcal{D}_{j}\cup-\mathcal{D}_{j}$, we have monotonically
	$
	\sum_{\mathbf{t}\in(\mathcal{D}_{j})_{\mathbf{k}}}g(y\mathbf{Q}_{\mathcal{D}_{j}\cup-\mathcal{D}_{j},\mathbf{t}})\to0,$ and $\sum_{\mathbf{t}\in(\mathcal{D}_{j})_{-\mathbf{k}}\cup\{-\mathbf{k}\}}g(y\mathbf{Q}_{\mathcal{D}_{j}\cup-\mathcal{D}_{j},\mathbf{t}})\to\sum_{\mathbf{t}\in\mathcal{D}_{j}\cup-\mathcal{D}_{j}}g(y\mathbf{Q}_{\mathcal{D}_{j}\cup-\mathcal{D}_{j},\mathbf{t}})
	$
	a.s., as $\mathbf{k}\to\boldsymbol{\infty}$. Thus, by monotone convergence theorem the right-hand side in (\ref{Q-L}) converges, as $\mathbf{k}\to\boldsymbol{\infty}$, to
	\begin{equation*}
	\mathbb{E}\Big[\int_{0}^{\infty}1- \exp\Big(-\sum_{\mathbf{t}\in\mathcal{D}_{j}\cup-\mathcal{D}_{j}}g(y\mathbf{Q}_{\mathcal{D}_{j}\cup-\mathcal{D}_{j},\mathbf{t}}) \Big) d(-y^{-\alpha})\Big].
	\end{equation*}
	One deduces the following expression of the Laplace transform of $N^ \lambda$
		\begin{equation*}
	\Psi_{N^{\Lambda}}(g)=\exp\bigg(-\sum_{j=1}^{\infty}\lambda_{j}\int_{0}^{\infty}\mathbb{E}\Big[1
	-e^{-\sum_{\mathbf{t}\in\hat{\mathcal{D}_{j}}}g( y\mathbf{Q}_{\hat{\mathcal{D}_{j}},\mathbf{t}})}\Big]d(-y^{-\alpha})\bigg).
	\end{equation*}

\subsection{Poof of Proposition \ref{pro-New-1}}

The first statement follows from similar arguments as the ones used in the proof of Theorem 2.1 in \cite{BS} and in Lemma 3.1 in \cite{SZM}. In particular, it is easy to see that for all $\mathbf{s}, \mathbf{t} \in \mathbb{Z}^{k}$ with $\mathbf{s}\leq\mathbf{t}$
\begin{equation*}
\frac{\mathbb{P}((x^{-1} \mathbf{X}_{\mathbf{s}} , . . . , x^{-1} \mathbf{X}_{\mathbf{t}})\in\cdot)}{\mathbb{P}(\rho_{\Upsilon}(\mathbf{X})>x)}=\frac{\mathbb{P}(|\mathbf{X}_{\mathbf{0}}|>x)}{\mathbb{P}(\rho_{\Upsilon}(\mathbf{X})>x)}\frac{\mathbb{P}((x^{-1} \mathbf{X}_{\mathbf{s}}, . . . , x^{-1} \mathbf{X}_{\mathbf{t}})\in\cdot)}{\mathbb{P}(|\mathbf{X}_{\mathbf{0}}|>x)}.
\end{equation*}
Let $\tilde{\mu}$ be the tail measure of $\mathbf{X}$ with auxiliary regularly varying function $\mathbb{P}(|\mathbf{X}_{\mathbf{0}}|>x)$. By the definition of regular variation of $\mathbf{X}$, by homogeneity of the tail measure, and assuming w.l.o.g.~that $\{\mathbf{0}\}\in\Upsilon$ we have (see also Lemma 3.1 in \cite{SZM})
 \begin{equation*}
 \frac{\mathbb{P}(|\mathbf{X}_{\mathbf{0}}|>x)}{\mathbb{P}(\rho_{\Upsilon}(\mathbf{X})>x)}\to\frac{1}{\tilde{\mu}_{\Upsilon}(z\in\mathbb{R}^{|\Upsilon|}:\rho(z)>1)}\in(0,\infty).
 \end{equation*}
Thus, we have 
\begin{equation*}
\frac{\mathbb{P}((x^{-1} \mathbf{X}_{\mathbf{s}} , . . . , x^{-1} \mathbf{X}_{\mathbf{t}})\in\cdot)}{\mathbb{P}(\rho_{\Upsilon}(\mathbf{X})>x)}\to\frac{\tilde{\mu}_{\mathbf{s},\mathbf{t}}(\cdot)}{\tilde{\mu}_{\Upsilon}(z\in\mathbb{R}^{|\Upsilon|}:\rho(z)>1)}=:\mu_{\mathbf{s},\mathbf{t}}(\cdot).
\end{equation*}
Notice that the function $x \mapsto \mathbb{P}(\rho_{\Upsilon}(\mathbf{X})>x)$ is regularly varying of index $-\alpha$ and that $\mu$ restricted to the set $\{y_{\mathbf{s}},....,y_{\mathbf{t}}:\rho(\mathbf Y)>1\}$ is a probability measure, call it $\nu_{\mathbf{s},\mathbf{t}}$. Here we have used that w.l.o.g.~$\Upsilon\subset\{\mathbf{s},...,\mathbf{t}\}$, indeed if some $\mathbf{z}\in\Upsilon$ is not contained in $\{\mathbf{s},...,\mathbf{t}\}$ then we can consider $\tilde{\mu}_{\{\mathbf{s},\mathbf{t}\}\cup\{\mathbf{z}\}}$ because by consistency of the measure $\tilde{\mu}$ we have $\tilde{\mu}_{\{\mathbf{s},\mathbf{t}\}\cup\{\mathbf{z}\}}(\cdot,\mathbb{R})=\tilde{\mu}_{\mathbf{s},\mathbf{t}}(\cdot)$.

It is possible to see that $(\nu_{\mathbf{s},\mathbf{t}})_{\mathbf{s},\mathbf{t}\in\mathbb{Z}^{k}}$ is a family of consistent probability measures and by Kolmogorov extension theorem we obtain the first statement. The second statement follows from the first and the continuous mapping theorem.

\subsection{Proof of Proposition \ref{pro-New-2}}

	This follows from similar arguments used in the proofs of Theorem 3.1 in \cite{BS} and of Theorem 3.2 in \cite{SW}. Consider any $\mathbf{s}\in\mathbb{Z}^{k}$ and any $\mathbf{i},\mathbf{j}\in\mathbb{Z}^{k}$ such that $\Upsilon\subset\{\mathbf{t}\in\mathbb{Z}^{k}:\mathbf{i}\leq\mathbf{t}\leq\mathbf{j}\}$. Then, following the proof of Theorem 3.1 in \cite{BS} we define the spaces:
	\begin{equation*}
	\mathbb{E}_{\mathbf{i},\mathbf{j}}=\{(\mathbf{y}_{\mathbf{i}},...,\mathbf{y}_{\mathbf{j}})\,\,|\,\,0<\rho_{\Upsilon}(\mathbf{y})<\infty\},\qquad
	\mathbb{S}_{\mathbf{i},\mathbf{j}}=\{(\mathbf{y}_{\mathbf{i}},...,\mathbf{y}_{\mathbf{j}})\,\,|\,\,\rho_{\Upsilon}(\mathbf{y})=1\}.
	\end{equation*}
	Define the bijection $T:\mathbb{E}_{\mathbf{i},\mathbf{j}}\to(0,\infty)\times\mathbb{S}_{\mathbf{i},\mathbf{j}}$ by
	\begin{equation*}
	T(\mathbf{y}_{\mathbf{i}},...,\mathbf{y}_{\mathbf{j}})=\bigg(\rho_{\Upsilon}(\mathbf{y}),\frac{\mathbf{y}_{\mathbf{i}}}{\rho_{\Upsilon}(\mathbf{y})},...,\frac{\mathbf{y}_{\mathbf{j}}}{\rho_{\Upsilon}(\mathbf{y})}\bigg).
	\end{equation*}
	Let $\mu_{\mathbf{i},\mathbf{j}}$ be the tail measure of $(\mathbf{X}_{\mathbf{i}},...,\mathbf{X}_{\mathbf{j}})$ with auxiliary regularly varying function $\mathbb{P}(\rho_{\Upsilon}(\mathbf{X})>x)$ (as defined in the proof of Proposition \ref{pro-New-1}). Define the measure $\varPhi_{\mathbf{i},\mathbf{j}}$ on $\mathbb{S}_{\mathbf{i},\mathbf{j}}$ by
	\begin{equation*}
	\varPhi_{\mathbf{i},\mathbf{j}}(B)=\mu_{\mathbf{i},\mathbf{j}}\big(T^{-1}((1,\infty)\times B)\big)
	\end{equation*}
	for Borel-measurable $B\subset\mathbb{S}_{\mathbf{i},\mathbf{j}}$. Since the law of $(\mathbf{Y}_{\Upsilon,\mathbf{i}}, . . . , \mathbf{Y}_{\Upsilon,\mathbf{j}})$ is equal to the restriction of $\mu_{\mathbf{i},\mathbf{j}}$ to
	$T^{-1}((1,\infty)\times \mathbb{S}_{\mathbf{i},\mathbf{j}})$, the measure $\varPhi_{\mathbf{i},\mathbf{j}}$ is in fact equal to the law of $(\mathbf{\Theta}_{\Upsilon,\mathbf{i}}, . . . , \mathbf{\Theta}_{\Upsilon,\mathbf{j}})$. Furthermore, as $\mu_{\mathbf{i},\mathbf{j}}$ is homogeneous of order $-\alpha$, for $u\in(0, \infty)$ and Borel sets $B\subset\mathbb{S}_{\mathbf{i},\mathbf{j}}$
	\begin{equation}\label{Psi-New}
	\mu_{\mathbf{i},\mathbf{j}}\big(T^{-1}((u,\infty)\times B)\big)=\mu_{\mathbf{i},\mathbf{j}}\big(uT^{-1}((1,\infty)\times B)\big)=u^{-\alpha}\varPhi_{\mathbf{i},\mathbf{j}}(B)
	\end{equation}
	For $u\geq1$, the left-hand side is equal to $\mathbb{P}(\rho_{\Upsilon}(\mathbf{Y})>u,(\mathbf{\Theta}_{\Upsilon,\mathbf{i}}, . . . , \mathbf{\Theta}_{\Upsilon,\mathbf{j}})\in B)$, while the right hand side is equal to $\mathbb{P}(\rho_{\Upsilon}(\mathbf{Y})>u)\mathbb{P}((\mathbf{\Theta}_{\Upsilon,\mathbf{i}}, . . . , \mathbf{\Theta}_{\Upsilon,\mathbf{j}})\in B)$. Thus, $\rho_{\Upsilon}(\mathbf{Y})$ and $(\mathbf{\Theta}_{\Upsilon,\mathbf{i}}, . . . , \mathbf{\Theta}_{\Upsilon,\mathbf{j}})$ are independent and so $\rho_{\Upsilon}(\mathbf{Y})$ and $(\mathbf{\Theta}_{\Upsilon,\mathbf{i}})_{\mathbf{i}\in I}$ are independent, where $I$ is any subset of $\{\mathbf{t}\in\mathbb{Z}^{k}:\mathbf{i}\leq\mathbf{t}\leq\mathbf{j}\}$. Since $\mathbf{i}$ and $\mathbf{j}$ were arbitrary, point (i) follows. 
	
	Concerning point (ii), consider any $\mathbf{i},\mathbf{j}\in\mathbb{Z}^{k}$ and let $g:(\mathbb{R}^{d})^{\mathbb{Z}^{k}}\to\mathbb{R}$ be a bounded \textit{and continuous} function. By stationarity 
	\begin{align}
	\mathbb{E}&[g(\mathbf{Y}_{\Upsilon,\mathbf{i}-\mathbf{s}},...,\mathbf{Y}_{\Upsilon,\mathbf{j}-\mathbf{s}})\mathbf{1}(\rho_{(\Upsilon)_{-\mathbf{s}}}(\mathbf{Y})>\varepsilon)]\nonumber\\
&=\lim\limits_{x\to\infty}\frac{\mathbb{E}[g(x^{-1} \mathbf{X}_{\mathbf{i}-\mathbf{s}}, . . . , x^{-1} \mathbf{X}_{\mathbf{j}-\mathbf{s}})\mathbf{1}(\rho_{(\Upsilon)_{-\mathbf{s}}}(\mathbf{X})>x\varepsilon)\mathbf{1}(\rho_{\Upsilon}(\mathbf{X})>x)]}{\mathbb{P}(\rho_{\Upsilon}(\mathbf{X})>x)}\nonumber\\
&=\lim\limits_{x\to\infty}\frac{\mathbb{E}[g(x^{-1} \mathbf{X}_{\mathbf{i}} , . . . , x^{-1} \mathbf{X}_{\mathbf{j}})\mathbf{1}(\rho_{\Upsilon}(\mathbf{X})>x\epsilon)\mathbf{1}(\rho_{(\Upsilon)_{\mathbf{s}}}(\mathbf{X})>x)]}{\mathbb{P}(\rho_{\Upsilon}(\mathbf{X})>x)}\nonumber\\
&\label{mu}
=\int g(\mathbf{y}_{\mathbf{i}},...,\mathbf{y}_{\mathbf{j}})\mathbf{1}(\rho_{\Upsilon}(\mathbf{y})>\varepsilon)\mathbf{1}(\rho_{(\Upsilon)_{\mathbf{s}}}(\mathbf{y})>1)\mu_{\mathbf{l},\mathbf{k}}(d\mathbf{y})
\end{align}
where $\mathbf{l},\mathbf{k}\in\mathbb{Z}^{k}$ are such that $\{\mathbf{t}\in\mathbb{Z}^{k}:\mathbf{l}\leq\mathbf{t}\leq\mathbf{k}\}\supset\{\mathbf{t}\in\mathbb{Z}^{k}:\mathbf{i}\leq\mathbf{t}\leq\mathbf{j}\}\cup\Upsilon\cup(\Upsilon)_{\mathbf{s}}$. The last equality follows from the consistency of the measures $\mu_{\mathbf{l},\mathbf{k}}$, $\mathbf{l},\mathbf{k}\in\mathbb{Z}^{k}$ and the fact that $\mu_{\mathbf{l},\mathbf{k}}$ restricted on the set $\{(\mathbf{y}_{\mathbf{l}},...,\mathbf{y}_{\mathbf{k}}):\rho_{(\Upsilon)_{\mathbf{s}}}(\mathbf{y})>1\}$ is a probability measure, call it $\tilde{\nu}_{\mathbf{s},\mathbf{t}}$, and this holds for any $\mathbf{s}\in\mathbb{Z}^{k}$; indeed by stationarity
\begin{equation*}
\frac{\mathbb{P}((x^{-1} \mathbf{X}_{\mathbf{l}} , . . . , x^{-1} \mathbf{X}_{\mathbf{k}})\in\cdot)}{\mathbb{P}(\rho_{(\Upsilon)_{s}}(\mathbf{X})>x)}=\frac{\mathbb{P}((x^{-1} \mathbf{X}_{\mathbf{l}} , . . . , x^{-1} \mathbf{X}_{\mathbf{k}})\in\cdot)}{\mathbb{P}(\rho_{\Upsilon}(\mathbf{X})>x)}\to\mu_{\mathbf{l},\mathbf{k}}(\cdot)
\end{equation*}
As a side note observe that $\mathbf{Y}_{\Upsilon}$ is not necessarily stationary because different restrictions (\textit{i.e.}~different $\mathbf{s}$) of $\mu$ correspond to potentially different probability measures. That is $\nu_{\mathbf{s},\mathbf{t}}$, which is the probability measure given by $\mu_{\mathbf{l},\mathbf{k}}$ restricted on $\{(\mathbf{y}_{\mathbf{l}},...,\mathbf{y}_{\mathbf{k}}):\rho_{(\Upsilon)_{\mathbf{s}}}(\mathbf{y})>1\}$ (as introduced in the proof of Proposition \ref{pro-New-1}) is potentially a different from $\tilde{\nu}_{\mathbf{s},\mathbf{t}}$.

Now, by $(\ref{Psi-New})$ applied to $\mu_{\mathbf{l},\mathbf{k}}$ we obtain that $(\ref{mu})$ is equal to
\begin{equation*}
\int_{\varepsilon}^{\infty}\mathbb{E}[g(r\mathbf{\Theta}_{\Upsilon,\mathbf{i}},...,r\mathbf{\Theta}_{\Upsilon,\mathbf{j}})\mathbf{1}(r\rho_{(\Upsilon)_{\mathbf{s}}}(\mathbf{\Theta})>1)]d(-r^{-\alpha}).
\end{equation*}
Following the monotone convergence arguments of the proof of Theorem 3.2 in \cite{SW} we send $\varepsilon\to0$ and drop the continuity assumption of $g$. Since $\mathbf{i}$ and $\mathbf{j}$ were arbitrary we obtain the result.

Finally point (iii) follows from $(\ref{timechangeY-New})$ applied to the function $\tilde{g}(\mathbf{y})=g(\mathbf{y}/\rho_{(\Upsilon)_{\mathbf{s}}}(\mathbf{y}))\mathbf{1}(\rho_{(\Upsilon)_{\mathbf{s}}}(\mathbf{y})\neq 0)$.

\subsection{Proof of Lemma \ref{lem-CD}}

	Let $c^{-}:=\inf_{\mathbf{x}:\rho_{\Upsilon}(\mathbf{x})\geq1}|\mathbf{x}|$ and $c^{+}:=\sup_{\mathbf{x}:\rho_{\Upsilon}(\mathbf{x})<1}|\mathbf{x}|$.	By homogeneity we have that $\inf_{\mathbf{x}:\rho_{\Upsilon}(\mathbf{x})\geq\epsilon}|\mathbf{x}|=\epsilon c^{-}$ for every $\epsilon>0$ (and the same holds for $c^{+}$). This implies that if $|\mathbf{x}|<\epsilon$ then $\rho_{\Upsilon}(\mathbf{x})<\epsilon/c^{-}$, and if $\rho_{\Upsilon}(\mathbf{x})<\epsilon$ then $|\mathbf{x}|<c^{+}\epsilon$. From the latter we deduce that for every $b>0$ we have that $|\mathbf{x}|\geq bc^{+}$ implies $\rho_{\Upsilon}(\mathbf{x})\geq b$ (or equivalently by homogeneity that $|\mathbf{x}|\geq b$ implies $\rho_{\Upsilon}(\mathbf{x})\geq b/c^{+}$) and from the former that $\rho_{\Upsilon}(\mathbf{x})\geq b$ implies that $|\mathbf{x}|\geq bc^{-}$.
	
	Furthermore, it is easy to see that $\max_{\mathbf{t}\in\Upsilon}|\mathbf{x}_{\mathbf{t}}|$ is a norm on $\mathbb{R}^{|\Upsilon|k}$ and since on any finite dimensional vector space any norm is equivalent to any other norm, we have that there exists two constants $A$ and $B$ such that $A|\mathbf{x}|\leq \max_{\mathbf{t}\in\hat{\Upsilon}}|\mathbf{x}_{\mathbf{t}}| \leq B|\mathbf{x}|$. Therefore, for every $\epsilon>0$ we have that $\max_{\mathbf{t}\in\Upsilon}|\mathbf{x}_{\mathbf{t}}|<\epsilon$ implies that $A|\mathbf{x}|<\epsilon$ which in turn implies that $\rho_{\Upsilon}(\mathbf{x})<\frac{\epsilon}{Ac^{-}}$. Moreover, for every $\epsilon>0$ we have that $\max_{\mathbf{t}\in\Upsilon}|\mathbf{x}_{\mathbf{t}}|\geq\epsilon$ implies that $B|\mathbf{x}|\geq\epsilon$ which in turn implies that $\rho_{\Upsilon}(\mathbf{x})\geq\frac{\epsilon}{Bc^{+}}$. 
	
	Similarly for the other direction we have that, for every $\epsilon>0$, $\rho_{\Upsilon}(\mathbf{x})<\epsilon$ implies that $|\mathbf{x}|<c^{+}\epsilon$ which implies that $\max_{\mathbf{t}\in\Upsilon}|\mathbf{x}_{\mathbf{t}}|<c^{+}B\epsilon$. Moreover, for every $\epsilon>0$, $\rho_{\Upsilon}(\mathbf{x})\geq\epsilon$ implies that $|\mathbf{x}|\geq c^{-}\epsilon$ which implies that $\max_{\mathbf{t}\in\Upsilon}|\mathbf{x}_{\mathbf{t}}|\geq c^{-}A\epsilon$. Thus by setting $C=Ac^{-}$ and $D=Bc^{+}$ we obtain the result.

\subsection{Proof of Theorem \ref{t1-L-E-Pro}}
Before proving Theorem \ref{t1-L-E-Pro} we present the following result on the connection between tail random fields for different sets $\Upsilon_{1}$ and $\Upsilon_{2}$ and different moduli of continuity $\rho_{1}$ and $\rho_{2}$.

\begin{lemma}\label{lem:difmod}
Let $\Upsilon_{1}$ and $\Upsilon_{2}$ be two finite subset of $\mathbb{Z}^{k}$ and consider $\rho_{1}$ and $\rho_{2}$ be two moduli of continuity on $\mathbb{R}^{d\mathbb{Z}^{k}}$. Let $C_{1}$ and $C_{2}$ be the constants such that, for every $\epsilon>0$, $\rho_{1,\Upsilon_{1}}(x)>\epsilon$ implies $\rho_{2,\Upsilon_{2}}(x)>\frac{\epsilon}{C_{2}}$, and $\rho_{2,\Upsilon_{2}}(x)>\epsilon$ implies $\rho_{1,\Upsilon_{1}}(x)>C_{1}\epsilon$. Then,
\begin{equation*}
\mathbf{Y}_{\Upsilon_{2}}\stackrel{d}{=}C_{1}\mathbf{Y}_{\Upsilon_{1}}\big|\rho_{2,\Upsilon_{2}}(\mathbf{Y}_{\Upsilon_{1}})>\frac{1}{C_{1}}
\end{equation*}
and
\begin{equation*}
\mathbf{Y}_{\Upsilon_{1}}\stackrel{d}{=}\frac{\mathbf{Y}_{\Upsilon_{2}}}{C_{2}}\Big|\rho_{1,\Upsilon_{1}}(\mathbf{Y}_{\Upsilon_{2}})>C_{2}.
\end{equation*}
\end{lemma}
\begin{proof}
Let $\Xi$ be a finite subset of $\mathbb{Z}^{k}$ and let $g:\mathbb{R}^{d|\Xi|}\to\mathbb{R}$ be a bounded and continuous function. Then, by homogeneity we have
\begin{align*}
&\mathbb{E}\big[g\big(\big(\mathbf{Y}_{\Upsilon_{1}}(\mathbf{t})\big)_{\mathbf{t}\in\Xi}\big)\big]\\
&=\lim\limits_{x\to\infty}\frac{\mathbb{E}\Big[g\Big(\Big(\frac{\mathbf{X}_{\mathbf{t}}}{x}\Big)_{\mathbf{t}\in\Xi}\Big)\mathbf{1}(\rho_{1,\Upsilon_{1}}(\mathbf{X})>x)\Big]}{\mathbb{P}(\rho_{1,\Upsilon_{1}}(\mathbf{X})>x)}\\
&=\lim\limits_{x\to\infty}\frac{\mathbb{E}\Big[g\Big(\Big(\frac{\mathbf{X}_{\mathbf{t}}}{x}\Big)_{\mathbf{t}\in\Xi}\Big)\mathbf{1}(\rho_{1,\Upsilon_{1}}(\mathbf{X})>x)[\mathbf{1}(\rho_{2,\Upsilon_{2}}(\mathbf{X})>\frac{x}{C_{1}})+\mathbf{1}(\rho_{2,\Upsilon_{2}}(\mathbf{X})\leq\frac{x}{C_{1}})]\Big]}{\mathbb{P}(\rho_{1,\Upsilon_{1}}(\mathbf{X})>x)}\frac{\mathbb{P}(\rho_{2,\Upsilon_{2}}(\mathbf{X})>\frac{x}{C_{1}})}{\mathbb{P}(\rho_{2,\Upsilon_{2}}(\mathbf{X})>\frac{x}{C_{1}})}\\
&=K\mathbb{E}\Big[g\Big(\Big(\dfrac{\mathbf{Y}_{\Upsilon_{2},\mathbf{t}}}{C_{1}}\Big)_{\mathbf{t}\in\Xi}\Big)\Big]+\mathbb{E}\big[g\big(\big(\mathbf{Y}_{\Upsilon_{1},\mathbf{t}}\big)_{\mathbf{t}\in\Xi}\big)\mathbf{1}(\rho_{2,\Upsilon_{2}}(\mathbf{Y}_{\Upsilon_{1}})\leq\frac{1}{C_{1}})\big]
\end{align*}
where
\begin{equation*}
K=\lim\limits_{x\to\infty}\frac{\mathbb{P}(\rho_{2,\Upsilon_{2}}(\mathbf{X})>\frac{x}{C_{1}})}{\mathbb{P}(\rho_{1,\Upsilon_{1}}(\mathbf{X})>x)}=\lim\limits_{x\to\infty}\frac{\mathbb{P}(\rho_{2,\Upsilon_{2}}(\mathbf{X})>\frac{x}{C_{1}},\rho_{1,\Upsilon_{1}}(\mathbf{X})>x)}{\mathbb{P}(\rho_{1,\Upsilon_{1}}(\mathbf{X})>x)}=\mathbb{P}\Big(\rho_{2,\Upsilon_{2}}(\mathbf{Y}_{\Upsilon_{1}})>\frac{1}{C_{1}}\Big).
\end{equation*}
Thus, we have
\begin{equation*}
\mathbb{E}\Big[g\Big(\Big(\dfrac{\mathbf{Y}_{\Upsilon_{2},\mathbf{t}}}{C_{1}}\Big)_{\mathbf{t}\in\Xi}\Big)\Big]=\mathbb{E}\Big[g\big(\big(\mathbf{Y}_{\Upsilon_{1},\mathbf{t}}\big)_{\mathbf{t}\in\Xi}\big)|\rho_{2,\Upsilon_{2}}(\mathbf{Y}_{\Upsilon_{1}})>\frac{1}{C_{1}}\Big].
\end{equation*}
Since $g$ and $\Xi$ were arbitrary we obtain the first stated result. The same arguments apply to the second one.
\end{proof}

	For notational purpose we consider the general case of countably many $\mathcal{D}$s and so of $\mathcal{E}$s. For the first part of this proof we follow similar arguments as the ones used in the proof of Theorem \ref{t1-L}. By $\mathcal{A}^{\Lambda}(a_{n})$ it suffices to show that for any $g\in \mathbb{C}^{+}_{K}$, $(\Psi_{\tilde{N}^{\Lambda}_{r_{n}}}(g))^{k_{n}}$ converges to (\ref{Psi-L-E-Pro}) as $n\to\infty$: It suffices to prove that $k_{n}(1-\Psi_{\tilde{N}^{\Lambda}_{r_{n}}}(g))$ converges to the logarithm of (\ref{Psi-L-E-Pro}) as $n\to\infty$.

	Recall the sets $S_{i,4l}$, $S'_{i,4l}$, and $\tilde{S}_{i,4l}$ from Proposition \ref{prop:latlambd} and its proof. Further, let $\bar{S}_{i,4l}:=S_{i,4l}\setminus S'_{i,4l}$. For notational consistency let $\bar{S}_{i,4l}:=\emptyset$ for $i$ corresponding to $\mathcal{D}_i$ bounded. Now, we apply a telescoping sum argument which generalises the one used in the proof of Theorem \ref{t1-L}. 
	
	Let $u:=|\bigcup_{i\in I^*,i<m_{4l}}S'_{i,4l}|$ (we omit the dependency on $l$ and on $n$ in $u$) and let $s_{1}\prec s_{2}\prec...\prec s_{u}$ denote the points in $\bigcup_{i\in I^*,i<m_{4l}}S'_{i,4l}$. Denote by $\mathcal{E}_{j_{1}},...,\mathcal{E}_{j_{u}}$ the $\mathcal{E}$s associated to $s_{1},...,s_{u}$. Let $\bar{u}:=|\bigcup_{i\in I^*,i<m_{4l}}\bar{S}_{i,4l}|$ and let $\bar{s}_{1}\prec \bar{s}_{2}\prec...\prec \bar{s}_{\bar{u}}$ denote the points in $\bigcup_{i\in I^*,i<m_{4l}}\bar{S}_{i,4l}$. Denote by $\mathcal{E}_{\bar{j}_{1}},...,\mathcal{E}_{\bar{j}_{\bar{u}}}$ the $\mathcal{E}$s associated to $\bar{s}_{1},...,\bar{s}_{\bar{u}}$. Let $\tilde{s}_{1},...,\tilde{s}_{\tilde{u}}$, for some $\tilde{u}\in\mathbb{N}\cup\{\mathbf{0}\}$, be the ordered points in $\Lambda_{r_n}\setminus\Big(\bigcup_{i=1}^{u}(\mathcal{E}_{j_{i}})_{s_{i}}\cup \bigcup_{i=1}^{\bar{u}}(\mathcal{E}_{\bar{j}_{i}})_{\bar{s}_{i}}\Big)$. In this case we associate the set $\{\mathbf{0}\}$ to any point $\tilde{s}_{h}$, $h=1,...,\tilde{u}$. Denote by $\hat{u}:=u+\bar{u}+\tilde{u}$ and by $\hat{s}_{1},...,\hat{s}_{\hat{u}}$ the points $s_{1},...s_{u},\bar{s}_{1},...\bar{s}_{u},\tilde{s}_{1},...,\tilde{s}_{\tilde{u}}$ indexed such that $\hat{s}_{1}\prec\hat{s}_{2}\prec...\prec\hat{s}_{\hat{u}}$, and denote by $\hat{\mathcal{E}}_{1},...,\hat{\mathcal{E}}_{\hat{u}}$, the corresponding sets $\mathcal{E}_{j_{1}},...,\mathcal{E}_{j_{u}},\mathcal{E}_{\bar{j}_{1}},...,\mathcal{E}_{\bar{j}_{\bar{u}}},\underbrace{\{\mathbf{0}\},...,\{\mathbf{0}\}}_{\tilde{u}\textnormal{ times}}$; for example if $\hat{s}_{1}=\bar{s}_{\bar{u}}$ then $\hat{\mathcal{E}}_{1}=\mathcal{E}_{\bar{j}_{\bar{u}}}$.
	
Let
	\begin{equation*}
	\tilde{\Psi}_{l,m}(g)=\begin{cases}
\exp\Big(-\sum_{j=m}^{\hat{u}}\sum_{\mathbf{t}\in(\hat{\mathcal{E}}_{j})_{\hat{s}_{j}}}g(a_{n}^{-1}\mathbf{X}_{\mathbf{t}}) \Big) \Big],& 1\leq m\leq \hat{u},\\1, & m=\hat{u}+1\,,
	\end{cases}
	\end{equation*}
so that for every $l\in\mathbb{N}$ we obtain $
	1-\Psi_{\tilde{N}_{r_{n}}}(g)=\sum_{m=1}^{\hat{u}}\tilde{\Psi}_{l,m+1}(g)-\tilde{\Psi}_{l,m}(g)$.
	 By the stationarity of $\mathbf{X}$, we have
	\begin{align}\label{m}
	\lefteqn{\tilde{\Psi}_{m+1}(g)-\tilde{\Psi}_{m}(g)}\\
\nonumber	=&\mathbb{E}\Big[\Big(1-e^{-\sum_{\mathbf{t}\in\hat{\mathcal{E}}_{m}}g(a_{n}^{-1}\mathbf{X}_{\mathbf{t}})} \Big) \exp\Big(-\sum_{j=m+1}^{\hat{u}}\sum_{\mathbf{t}\in(\hat{\mathcal{E}}_{j})_{\hat{s}_{j}}}g(a_{n}^{-1}\mathbf{X}_{\mathbf{t}-\hat{s}_{m}}) \Big) \Big]\\
\nonumber=&\mathbb{E}\Big[\Big(1-e^{-\sum_{\mathbf{t}\in\hat{\mathcal{E}}_{m}}g(a_{n}^{-1}\mathbf{X}_{\mathbf{t}})} \Big) \exp\Big(-\sum_{\mathbf{t}\in\bigcup_{j=m+1}^{\hat{u}}(\hat{\mathcal{E}}_{j})_{\hat{s}_{j}-\hat{s}_{m}}}g(a_{n}^{-1}\mathbf{X}_{\mathbf{t}}) \Big) \Big]\\
\nonumber	=&\mathbb{E}\Big[\Big(1-e^{-\sum_{\mathbf{t}\in\hat{\mathcal{E}}_{m}}g(a_{n}^{-1}\mathbf{X}_{\mathbf{t}})} \Big)\exp\Big(-\sum_{\mathbf{t}\in \bigcup_{j=m+1}^{\hat{u}}(\hat{\mathcal{E}}_{j})_{\hat{s}_{j}-\hat{s}_{m}}\cap K_{2l}}g(a_{n}^{-1}\mathbf{X}_{\mathbf{t}}) \Big) \mathbf{1}(\max_{\mathbf{t}\in \bigcup_{j=m+1}^{\hat{u}}(\hat{\mathcal{E}}_{j})_{\hat{s}_{j}-\hat{s}_{m}}\setminus K_{2l}}|\mathbf{X}_{\mathbf{t}}|\leq\delta a_{n}) \Big]\\
\nonumber&	+\mathbb{E}\Big[\Big(1-e^{-\sum_{\mathbf{t}\in\hat{\mathcal{E}}_{m}}g(a_{n}^{-1}\mathbf{X}_{\mathbf{t}})} \Big)\exp\Big(-\sum_{\mathbf{t}\in \bigcup_{j=m+1}^{\hat{u}}(\hat{\mathcal{E}}_{j})_{\hat{s}_{j}-\hat{s}_{m}}}g(a_{n}^{-1}\mathbf{X}_{\mathbf{t}}) \Big) 	\mathbf{1}(\max_{\mathbf{t}\in \bigcup_{j=m+1}^{\hat{u}}(\hat{\mathcal{E}}_{j})_{\hat{s}_{j}-\hat{s}_{m}}\setminus K_{2l}}|\mathbf{X}_{\mathbf{t}}|>\delta a_{n}) \Big]\\
\nonumber	=&\mathbb{E}\Big[\Big(1-e^{-\sum_{\mathbf{t}\in\hat{\mathcal{E}}_{m}}g(a_{n}^{-1}\mathbf{X}_{\mathbf{t}})} \Big)\exp\Big(-\sum_{\mathbf{t}\in \bigcup_{j=m+1}^{\hat{u}}(\hat{\mathcal{E}}_{j})_{\hat{s}_{j}-\hat{s}_{m}}\cap K_{2l}}g(a_{n}^{-1}\mathbf{X}_{\mathbf{t}}) \Big) \Big]\\
\nonumber&	-\mathbb{E}\Big[\Big(1-e^{-\sum_{\mathbf{t}\in\hat{\mathcal{E}}_{m}}g(a_{n}^{-1}\mathbf{X}_{\mathbf{t}})} \Big)\exp\Big(-\sum_{\mathbf{t}\in \bigcup_{j=m+1}^{\hat{u}}(\hat{\mathcal{E}}_{j})_{\hat{s}_{j}-\hat{s}_{m}}\cap K_{2l}}g(a_{n}^{-1}\mathbf{X}_{\mathbf{t}}) \Big)\\
\nonumber&\mathbf{1}(\max_{\mathbf{t}\in \bigcup_{j=m+1}^{\hat{u}}(\hat{\mathcal{E}}_{j})_{\hat{s}_{j}-\hat{s}_{m}}\setminus K_{2l}}|\mathbf{X}_{\mathbf{t}}|>\delta a_{n})\mathbf{1}(\max_{\mathbf{t}\in\hat{\mathcal{E}}_{m}}|\mathbf{X}_{\mathbf{t}}|>\delta a_{n})  \Big]\\
\nonumber&+\mathbb{E}\Big[\Big(1-e^{-\sum_{\mathbf{t}\in\hat{\mathcal{E}}_{m}}g(a_{n}^{-1}\mathbf{X}_{\mathbf{t}})} \Big)\exp\Big(-\sum_{\mathbf{t}\in \bigcup_{j=m+1}^{\hat{u}}(\hat{\mathcal{E}}_{j})_{\hat{s}_{j}-\hat{s}_{m}}}g(a_{n}^{-1}\mathbf{X}_{\mathbf{t}}) \Big)\\
\nonumber&
\mathbf{1}(\max_{\mathbf{t}\in \bigcup_{j=m+1}^{\hat{u}}(\hat{\mathcal{E}}_{j})_{\hat{s}_{j}-\hat{s}_{m}}\setminus K_{2l}}|\mathbf{X}_{\mathbf{t}}|>\delta a_{n})\mathbf{1}(\max_{\mathbf{t}\in\hat{\mathcal{E}}_{m}}|\mathbf{X}_{\mathbf{t}}|>\delta a_{n}) \Big]\\
	\nonumber=&\mathbb{E}\Big[\Big(1-e^{-\sum_{\mathbf{t}\in\hat{\mathcal{E}}_{m}}g(a_{n}^{-1}\mathbf{X}_{\mathbf{t}})} \Big)\exp\Big(-\sum_{\mathbf{t}\in \bigcup_{j=m+1}^{\hat{u}}(\hat{\mathcal{E}}_{j})_{\hat{s}_{j}-\hat{s}_{m}}\cap K_{2l}}g(a_{n}^{-1}\mathbf{X}_{\mathbf{t}}) \Big) \Big]+J^{(r_n)}_{l,m}.
	\end{align}
	Consider now only the $J^{(r_n)}_{l,m}$ where $m$ is such that $\hat{s}_{m}=s_i$ for some $i=1,...,u$. We have that
	\begin{multline}
	J^{(r_n)}_{l,m}\leq 2\mathbb{E}\Big[\mathbf{1}(\max_{\mathbf{t}\in \bigcup_{j=m+1}^{\hat{u}}(\hat{\mathcal{E}}_{j})_{\hat{s}_{j}-\hat{s}_{m}}\setminus K_{2l}}|\mathbf{X}_{\mathbf{t}}|>\delta a_{n})\mathbf{1}(\max_{\mathbf{t}\in\mathcal{E}_{i}}|\mathbf{X}_{\mathbf{t}}|>\delta a_{n})\Big]\\
	\label{ineq-J}
	\leq 2\mathbb{P}(\hat{M}_{2l,r_{n}}^{\Lambda,|\mathbf{X}|,(i)}>\delta a_{n},\max_{\mathbf{t}\in\mathcal{E}_{i}}|\mathbf{X}_{\mathbf{t}}|>\delta a_{n}).
	\end{multline}
	By Proposition \ref{prop:latlambd} we have that $|S'_{i,4l}|/|\Lambda_{r_n}|\to\gamma^*_i$ as $n\to\infty$, and so that $\lim\limits_{n\to\infty}\frac{u}{|\Lambda_{r_n}|}=\sum_{j\in I^*,j<m_{4l}}\gamma^*_{j}$. Further, since the inequality $(\ref{ineq-J})$ holds for every $n,l\in\mathbb{N}$ and since $\sum_{j\in I^*}\gamma^*_{j}|\mathcal{E}_{j}|=1$, we obtain that
	\begin{align}
\nonumber\lefteqn{	\lim\limits_{l\to\infty}\limsup_{n\to\infty}k_{n}\sum_{m=1}^{u}|J^{(r_n)}_{l,m}|}
\\\nonumber&=\lim\limits_{l\to\infty}\limsup_{n\to\infty}k_{n}\sum_{i\in I^*,i<m_{4l}}2|S'_{i,4l}|\mathbb{P}(\hat{M}_{2l,r_{n}}^{\Lambda,|\mathbf{X}|,(i)}>\delta a_{n},\max_{\mathbf{t}\in\mathcal{E}_{i}}|\mathbf{X}_{\mathbf{t}}|>\delta a_{n})\\
\nonumber&	\leq2\lim\limits_{l\to\infty}\sum_{i\in I^*,i<m_{4l}}\gamma_{i}^{*}\limsup_{n\to\infty}|\Lambda_{r_{n}}|k_{n}\mathbb{P}(\hat{M}_{2l,r_{n}}^{\Lambda,|\mathbf{X}|,(i)}>\delta a_{n},\max_{\mathbf{t}\in\mathcal{E}_{i}}|\mathbf{X}_{\mathbf{t}}|>\delta a_{n})\\
	\nonumber &=2\lim\limits_{l\to\infty}\sum_{i\in I^*,i<m_{4l}}\gamma_{i}^{*}\limsup_{n\to\infty}|\Lambda_{n}|\mathbb{P}(|\mathbf{X}_{\mathbf{0}}|>\delta a_{n})\mathbb{P}(\hat{M}_{2l,r_{n}}^{\Lambda,|\mathbf{X}|,(i)}>\delta a_{n}|\max_{\mathbf{t}\in\mathcal{E}_{i}}|\mathbf{X}_{\mathbf{t}}|>\delta a_{n})\frac{\mathbb{P}(\max_{\mathbf{t}\in\mathcal{E}_{i}}|\mathbf{X}_{\mathbf{t}}|>\delta a_{n})}{\mathbb{P}(|\mathbf{X}_{\mathbf{0}}|>\delta a_{n})}\\
\nonumber &	=2\lim\limits_{l\to\infty}\sum_{i\in I^*,i<m_{4l}}\gamma_{i}^{*}\limsup_{n\to\infty}\mathbb{P}(\hat{M}_{2l,r_{n}}^{\Lambda,|\mathbf{X}|,(i)}>\delta a_{n}|\max_{\mathbf{t}\in\mathcal{E}_{i}}|\mathbf{X}_{\mathbf{t}}|>\delta a_{n})\frac{\mathbb{P}(\max_{\mathbf{t}\in\mathcal{E}_{i}}|\mathbf{X}_{\mathbf{t}}|>\delta a_{n})}{\mathbb{P}(|\mathbf{X}_{\mathbf{0}}|>\delta a_{n})}\\\label{bound-J}
&	\leq2\lim\limits_{l\to\infty}\sum_{i\in I^*,i<m_{4l}}\gamma_{i}^{*}|\mathcal{E}_{i}|\limsup_{n\to\infty}\mathbb{P}(\hat{M}_{2l,r_{n}}^{\Lambda,|\mathbf{X}|,(i)}>\delta a_{n}|\max_{\mathbf{t}\in\mathcal{E}_{i}}|\mathbf{X}_{\mathbf{t}}|>\delta a_{n})
	\end{align}	
	and since 
	\begin{equation*}
	\sum_{i\in I^*,i<m_{4l}}\gamma_{i}^{*}|\mathcal{E}_{i}|\limsup_{n\to\infty}\mathbb{P}(\hat{M}_{2l,r_{n}}^{\Lambda,|\mathbf{X}|,(i)}>\delta a_{n}|\max_{\mathbf{t}\in\mathcal{E}_{i}}|\mathbf{X}_{\mathbf{t}}|>\delta a_{n})	\leq \sum_{i\in I^*,i<m_{4l}}\gamma_{i}^{*}|\mathcal{E}_{i}|\leq \sum_{i\in I^*}\gamma_{i}^{*}|\mathcal{E}_{i}|=1
	\end{equation*}
	by dominated convergence theorem we have that (\ref{bound-J}) is equal to
	\begin{equation}\label{m-1}
	2 \sum_{i\in I^*}\gamma_{i}^{*}|\mathcal{E}_{i}|\lim\limits_{l\to\infty}\limsup_{n\to\infty}\mathbb{P}(\hat{M}_{2l,r_{n}}^{\Lambda,|\mathbf{X}|,(i)}>\delta a_{n}|\max_{\mathbf{t}\in\mathcal{E}_{i}}|\mathbf{X}_{\mathbf{t}}|>\delta a_{n})=0
	\end{equation}
	where the last equality follows by the anti-clustering condition (AC$^{\Lambda}_{\succeq,I^*}$).
	
	Now, let us focus on (\ref{m}) where $m$ is such that $\hat{s}_{m}=\bar{s}_i$ for some $i=1,...,\bar{u}$ or $\hat{s}_{m}=\tilde{s}_l$ for some $l=1,...,\tilde{u}$. Since by Proposition \ref{prop:latlambd} $\sum_{j<m_{4l}}\gamma^*_{j}|\mathcal{E}_{j}|\to\sum_{j\in I^*}\gamma^*_{j}|\mathcal{E}_{j}|=1$ monotonically as $l\to\infty$, then $\lim\limits_{l\to\infty}\lim\limits_{n\to\infty}\frac{|\Lambda_{r_{n}}|-\sum_{i\in I^*,i<m_{4l}}|S'_{i,4l}||\mathcal{E}_i|}{|\Lambda_{r_{n}}|}=0$. Since $\tilde{u}+\sum_{i\in I^*,i<m_{4l}}|\bar{S}_{i,4l}||\mathcal{E}_i|=|\Lambda_{r_{n}}|-\sum_{i\in I^*,i<m_{4l}}|S'_{i,4l}||\mathcal{E}_i|$, we obtain that
	\begin{equation}\label{u-tilde}
	\lim\limits_{l\to\infty}\lim\limits_{n\to\infty}\frac{\tilde{u}+\sum_{i\in I^*,i<m_{4l}}|\bar{S}_{i,4l}||\mathcal{E}_i|}{|\Lambda_{r_{n}}|}=0.
	\end{equation}
	By combining this with the fact that (\ref{m}) is bounded by
	\begin{equation*}
	\mathbb{E}\Big[\Big(1-e^{-\sum_{\mathbf{t}\in\hat{\mathcal{E}}_{m}}g(a_{n}^{-1}\mathbf{X}_{\mathbf{t}})} \Big) \Big]\leq \mathbb{P}(\max_{\mathbf{t}\in\hat{\mathcal{E}}_{m}}|\mathbf{X}_{\mathbf{t}}|>\delta a_{n})\leq |\hat{\mathcal{E}}_{m}|\mathbb{P}(|\mathbf{X}_{\mathbf{0}}|>\delta a_{n})
	\end{equation*}
	we conclude that
	\begin{multline}
	\lim\limits_{l\to\infty}\limsup_{n\to\infty}k_{n}\bigg(\tilde{u}\mathbb{P}(|\mathbf{X}_{\mathbf{0}}|>\delta a_{n})+\sum_{i\in I^*,i<m_{4l}}|\bar{S}_{i,4l}||\mathcal{E}_{i}|\mathbb{P}(|\mathbf{X}_{\mathbf{0}}|>\delta a_{n})\bigg)\label{m-2}\\
	=\lim\limits_{l\to\infty}\lim\limits_{n\to\infty}\frac{\tilde{u}+\sum_{i\in I^*,i<m_{4l}}|\bar{S}_{i,4l}||\mathcal{E}_i|}{|\Lambda_{r_{n}}|}=0.
	\end{multline}
	Now, by construction (recall (\ref{tilde-point}))  when $\hat{s}_{m}=s_{k}$ for some $k=1,...,u$ we get
	\begin{multline}\label{expectation}
	\mathbb{E}\Big[\Big(1-e^{-\sum_{\mathbf{t}\in\hat{\mathcal{E}}_{m}}g(a_{n}^{-1}\mathbf{X}_{\mathbf{t}})} \Big)\exp\Big(-\sum_{\mathbf{t}\in \bigcup_{j=m+1}^{\hat{u}}(\hat{\mathcal{E}}_{j})_{\hat{s}_{j}-\hat{s}_{m}}\cap K_{2l}}g(a_{n}^{-1}\mathbf{X}_{\mathbf{t}}) \Big) \Big]\\
	=\mathbb{E}\Big[\Big(1-e^{-\sum_{\mathbf{t}\in\mathcal{E}_{j_k}}g(a_{n}^{-1}\mathbf{X}_{\mathbf{t}})} \Big)\exp\Big(-\sum_{\mathbf{t}\in \tilde{\mathcal{D}}_{j_k}\cap K_{2l}}g(a_{n}^{-1}\mathbf{X}_{\mathbf{t}}) \Big) \Big],
	\end{multline}	
	where we used that by definition $(\mathcal{D}^*_{i}\cap K_{2l})\setminus \bigcup_{\mathbf{s}\in\{\mathbf{0}\}\cup -\mathcal{G}_i}(\mathcal{E}_{i})_{\mathbf{s}}=\tilde{\mathcal{D}}_i\cap K_{2l}$.	Therefore, by combining (\ref{m-1}) and (\ref{m-2}) we have that
	\begin{equation*}
	\lim\limits_{l\to\infty}\limsup_{n\to\infty}k_{n}\bigg|\sum_{m=1}^{\hat{u}}\tilde{\Psi}_{l,m+1}(g)-\tilde{\Psi}_{l,m}(g)-\sum_{m=1}^{u}\mathbb{E}\Big[\Big(1-e^{-\sum_{\mathbf{t}\in\mathcal{E}_{j_m}}g(a_{n}^{-1}\mathbf{X}_{\mathbf{t}})} \Big)\exp\Big(-\sum_{\mathbf{t}\in \tilde{\mathcal{D}}_{j_m}\cap K_{2l}}g(a_{n}^{-1}\mathbf{X}_{\mathbf{t}}) \Big) \Big]\bigg|=0.
	\end{equation*}
	
	Thus, for the remaining part of the proof we focus on
	\begin{equation*}
	k_n\sum_{m=1}^{u}\mathbb{E}\Big[\Big(1-e^{-\sum_{\mathbf{t}\in\mathcal{E}_{j_m}}g(a_{n}^{-1}\mathbf{X}_{\mathbf{t}})} \Big)\exp\Big(-\sum_{\mathbf{t}\in \tilde{\mathcal{D}}_{j_m}\cap K_{2l}}g(a_{n}^{-1}\mathbf{X}_{\mathbf{t}}) \Big) \Big].
	\end{equation*}
From Lemma \ref{lem-CD} we have that for every $i\in I^*$
\begin{equation}\label{relation max rho}
\Big\{\max_{\mathbf{t}\in\mathcal{E}_{i}}|\mathbf{X}_{\mathbf{t}}|>\delta a_{n}\Big\}\subseteq\Big\{\rho_{\mathcal{E}_{i}}(\mathbf{X})>\frac{\delta}{D_{i}} a_{n}\Big\},
\end{equation}
and thus we have
	\begin{align*}
	\lefteqn{\lim\limits_{n\to\infty}k_{n}\sum_{m=1}^{u}\mathbb{E}\Big[\Big(1-e^{-\sum_{\mathbf{t}\in\mathcal{E}_{j_m}}g(a_{n}^{-1}\mathbf{X}_{\mathbf{t}})} \Big)\exp\Big(-\sum_{\mathbf{t}\in \tilde{\mathcal{D}}_{j_m}\cap K_{2l}}g(a_{n}^{-1}\mathbf{X}_{\mathbf{t}}) \Big) \Big]}\\
	&=\lim\limits_{n\to\infty}k_{n}\sum_{m=1}^{u}\mathbb{E}\Big[\Big(1-e^{-\sum_{\mathbf{t}\in\mathcal{E}_{j_m}}g(a_{n}^{-1}\mathbf{X}_{\mathbf{t}})} \Big)\exp\Big(-\sum_{\mathbf{t}\in \tilde{\mathcal{D}}_{j_m}\cap K_{2l}}g(a_{n}^{-1}\mathbf{X}_{\mathbf{t}}) \Big)
	\mathbf{1}\Big(\max_{\mathbf{t}\in\mathcal{E}_{j_m}}|\mathbf{X}_{\mathbf{t}}|>\delta a_{n}\Big) \Big]\\
	& =\lim\limits_{n\to\infty}k_{n}\sum_{m=1}^{u}\mathbb{E}\Big[\Big(1-e^{-\sum_{\mathbf{t}\in\mathcal{E}_{j_m}}g(a_{n}^{-1}\mathbf{X}_{\mathbf{t}})} \Big)\exp\Big(-\sum_{\mathbf{t}\in \tilde{\mathcal{D}}_{j_m}\cap K_{2l}}g(a_{n}^{-1}\mathbf{X}_{\mathbf{t}}) \Big)\mathbf{1}\Big(\rho_{\mathcal{E}_{j_m}}(\mathbf{X})>\frac{\delta a_{n}}{D_{j_m}} \Big) \Big]
	\end{align*}
	\begin{align*}
	=&\lim\limits_{n\to\infty}k_{n}\sum_{m=1}^{u}\mathbb{E}\Big[\Big(1-e^{-\sum_{\mathbf{t}\in\mathcal{E}_{j_m}}g(a_{n}^{-1}\mathbf{X}_{\mathbf{t}})} \Big)\exp\Big(-\sum_{\mathbf{t}\in \tilde{\mathcal{D}}_{j_m}\cap K_{2l}}g(a_{n}^{-1}\mathbf{X}_{\mathbf{t}}) \Big)\\
	&\Big|\rho_{\mathcal{E}_{j_m}}(\mathbf{X})>\frac{\delta a_{n}}{D_{j_m}}  \Big]\frac{\mathbb{P}(\rho_{\mathcal{E}_{j_m}}(\mathbf{X})> \frac{\delta a_{n}}{D_{j_m}} )}{\mathbb{P}(\rho_{\mathcal{E}_{j_m}}(\mathbf{X})> a_{n})}\frac{\mathbb{P}(\rho_{\mathcal{E}_{j_m}}(\mathbf{X})> a_{n})}{\mathbb{P}(|\mathbf{X}_{0}|> a_{n})}\mathbb{P}(|\mathbf{X}_{0}|> a_{n})\\
	=&\sum_{j\in I^*,j<m_{4l}}\Big(\frac{\delta}{D_{j}}\Big)^{-\alpha}\gamma^*_{j}c_{j}\mathbb{E}\Big[\Big(1-e^{-\sum_{\mathbf{t}\in\mathcal{E}_{j}}g(\frac{\delta}{D_{j}} Y\mathbf{\Theta}_{\mathcal{E}_{j},\mathbf{t}})} \Big)\exp\Big(-\sum_{\mathbf{t}\in\tilde{\mathcal{D}}_{j}\cap K_{2l}}g(\frac{\delta}{D_{j}} Y\mathbf{\Theta}_{\mathcal{E}_{j},\mathbf{t}}) \Big)\Big]\\
	=&\sum_{j\in I^*,j<m_{4l}}\int_{\delta}^{\infty}\frac{\gamma^*_{j}c_{j}}{(D_{j})^{-\alpha}}\mathbb{E}\Big[\Big(1-e^{-\sum_{\mathbf{t}\in\mathcal{E}_{j}}g( \frac{y}{D_{j}}\mathbf{\Theta}_{\mathcal{E}_{j},\mathbf{t}})} \Big)\exp\Big(-\sum_{\mathbf{t}\in\tilde{\mathcal{D}}_{j}\cap K_{2l}}g( \frac{y}{D_{j}}\mathbf{\Theta}_{\mathcal{E}_{j},\mathbf{t}}) \Big)\Big]d(-y^{-\alpha})\\
	=&\int_{\delta}^{\infty}\sum_{j\in I^*,j<m_{4l}}\frac{\gamma^*_{j}c_{j}}{(D_{j})^{-\alpha}}\mathbb{E}\Big[\Big(1-e^{-\sum_{\mathbf{t}\in\mathcal{E}_{j}}g( \frac{y}{D_{j}}\mathbf{\Theta}_{\mathcal{E}_{j},\mathbf{t}})} \Big)\exp\Big(-\sum_{\mathbf{t}\in\tilde{\mathcal{D}}_{j}\cap K_{2l}}g( \frac{y}{D_{j}}\mathbf{\Theta}_{\mathcal{E}_{j},\mathbf{t}}) \Big)\Big]d(-y^{-\alpha}).
	\end{align*}
		
	By assumption $A^{\Lambda}_{\rho}$ we can apply the dominated convergence theorem (twice) as follows. First since the integrand is bounded by the constant $\sum_{j\in I^*}\gamma^*_{j}c_{j}D_{j}^{\alpha}$ for every $l\in\mathbb{N}$ and since this constant is bounded by condition $A^{\Lambda}_{\rho}$, which makes it an integrable function for the integral $\int_{\delta}^{\infty}d(-y^{-\alpha})$ for any $\delta>0$, we can put the limit of $l$ going to infinity inside the integral. Second, consider the finite counting measure $\sum_{j\in I^*}\gamma^*_{j}c_{j}D_{j}^{\alpha}\varepsilon_{j}(\cdot)$, where $\varepsilon$ is the Dirac delta measure. Since the integrand is bounded by 1 and 1 is an integrable function with respect to this finite counting measure, then we can apply again the dominated convergence theorem. Hence, we obtain
\begin{equation*}
\lim\limits_{l\to\infty}\int_{\delta}^{\infty}\sum_{j\in I^*,j<m_{4l}}\frac{\gamma^*_{j}c_{j}}{(D_{j})^{-\alpha}}\mathbb{E}\Big[\Big(1-e^{-\sum_{\mathbf{t}\in\mathcal{E}_{j}}g( \frac{y}{D_{j}}\mathbf{\Theta}_{\mathcal{E}_{j},\mathbf{t}})} \Big)\exp\Big(-\sum_{\mathbf{t}\in\tilde{\mathcal{D}}_{j}\cap K_{2l}}g( \frac{y}{D_{j}}\mathbf{\Theta}_{\mathcal{E}_{j},\mathbf{t}}) \Big)\Big]d(-y^{-\alpha})
\end{equation*}	
\begin{equation}\label{delta-new}
=\int_{\delta}^{\infty}\sum_{j\in I^*}\frac{\gamma^*_{j}c_{j}}{(D_{j})^{-\alpha}}\mathbb{E}\Big[\Big(1-e^{-\sum_{\mathbf{t}\in\mathcal{E}_{j}}g( \frac{y}{D_{j}}\mathbf{\Theta}_{\mathcal{E}_{j},\mathbf{t}})} \Big)\exp\Big(-\sum_{\mathbf{t}\in\tilde{\mathcal{D}}_{j}}g( \frac{y}{D_{j}}\mathbf{\Theta}_{\mathcal{E}_{j},\mathbf{t}}) \Big)\Big]d(-y^{-\alpha}).
\end{equation}

Since $\rho_{\mathcal{E}_{j}}(\mathbf{\Theta})=1$ a.s., by Corollary \ref{co-D-C-norm} we have that $\max_{\mathbf{s}\in\mathcal{E}_{j}}|\mathbf{\Theta}_{\mathbf{s}}|\leq D_{j}$ a.s.. Further, recall that $g$ has compact support, in particular for any $x\in\mathbb{R}^{d}$ with $|x|<\delta$ we have that $g(x)=0$. Then, we obtain that if $y<\delta$ then $\frac{y}{D_{j}}\mathbf{\Theta}_{\mathcal{E}_{j},\mathbf{t}}<\delta$ a.s., and so
$g(\frac{y}{D_{j}}\mathbf{\Theta}_{\mathcal{E}_{j},\mathbf{t}})=0$ a.s.. Thus, (\ref{delta-new}) is equal to
\begin{equation*}
\int_{0}^{\infty}\sum_{j\in I^*}\frac{\gamma^*_{j}c_{j}}{(D_{j})^{-\alpha}}\mathbb{E}\Big[\Big(1-e^{-\sum_{\mathbf{t}\in\mathcal{E}_{j}}g( \frac{y}{D_{j}}\mathbf{\Theta}_{\mathcal{E}_{j},\mathbf{t}})} \Big)\exp\Big(-\sum_{\mathbf{t}\in\tilde{\mathcal{D}}_{j}}g( \frac{y}{D_{j}}\mathbf{\Theta}_{\mathcal{E}_{j},\mathbf{t}}) \Big)\Big]d(-y^{-\alpha})
\end{equation*}
\begin{equation*}
\stackrel{\textnormal{Tonelli's theorem}}{=}\sum_{j\in I^*}\int_{0}^{\infty}\frac{\gamma^*_{j}c_{j}}{(D_{j})^{-\alpha}}\mathbb{E}\Big[\Big(1-e^{-\sum_{\mathbf{t}\in\mathcal{E}_{j}}g( \frac{y}{D_{j}}\mathbf{\Theta}_{\mathcal{E}_{j},\mathbf{t}})} \Big)\exp\Big(-\sum_{\mathbf{t}\in\tilde{\mathcal{D}}_{j}}g( \frac{y}{D_{j}}\mathbf{\Theta}_{\mathcal{E}_{j},\mathbf{t}}) \Big)\Big]d(-y^{-\alpha})
\end{equation*}
\begin{equation*}
\stackrel{(\textnormal{$\forall j\in I^*$ let }z=\frac{y}{D_{j}})}{=}\sum_{j\in I^*}\int_{0}^{\infty}\gamma^*_{j}c_{j}\mathbb{E}\Big[\Big(1-e^{-\sum_{\mathbf{t}\in\mathcal{E}_{j}}g( z\mathbf{\Theta}_{\mathcal{E}_{j},\mathbf{t}})}\Big)\exp\Big(-\sum_{\mathbf{t}\in\tilde{\mathcal{D}}_{j}}g( z\mathbf{\Theta}_{\mathcal{E}_{j},\mathbf{t}}) \Big) \Big]d(-z^{-\alpha}).
\end{equation*}

Finally, Corollary 4.14 in \cite{Kallenberg2} ensures the existence of the limiting random measure $N^{\Lambda}$ and so $N^{\Lambda}$ has the stated Laplace formulation.

\subsection{Proof of Proposition \ref{lem-bound-L-New}}
We start by showing the following useful Lemma:
\begin{lemma}\label{lem-AC1-L-New} Let $(\mathbf{Y}_{\Upsilon,\mathbf{t}}:\mathbf{t}\in\mathbb{Z}^{k})$ be an $\mathbb{R}^{d}$-valued random field such that the time change formula (\ref{timechangeY-New}) is satisfied. Let $\mathcal{L}$ be a (not necessarily full rank) lattice. Let $\mathbf{\Theta}_{\Upsilon,\mathbf{t}}=\mathbf{Y}_{\Upsilon,\mathbf{t}}/\rho_{\Upsilon}(\mathbf Y)$, $\mathbf{t}\in\mathbb{Z}^{k}$. Let $\mathcal{H}:=\bigcup_{\mathbf{s}\in\mathcal{G}}(\Upsilon)_{\mathbf{s}}$ where $\mathcal{G}=\mathcal{L}\cap\{\mathbf{t}\in\mathbb{Z}^{k}:\mathbf{t}\succeq\mathbf{0}\}$ and such that $(\Upsilon)_{\mathbf{s}}\cap (\Upsilon)_{\mathbf{s}'}=\emptyset$ for every $\mathbf{s},\mathbf{s}'\in\mathcal{G}$ with $\mathbf{s}\neq\mathbf{s}'$. Then $|\mathbf{\Theta}_{\Upsilon,\mathbf{t}}|\to0$ a.s.~as $|\mathbf{t}|\to\infty$ for $\mathbf{t}\in\mathcal{H}$ implies that $\sum_{\mathbf{t}\in\mathcal{L}}\rho_{(\Upsilon)_{\mathbf{t}}}(\mathbf{\Theta})^{\alpha}<\infty$ a.s., and that $\sum_{\mathbf{t}\in\bigcup_{\mathbf{s}\in\mathcal{L}}(\Upsilon)_{\mathbf{s}}}|\mathbf{\Theta}_{\Upsilon,\mathbf{t}}|^{\alpha}<\infty$ a.s.
\end{lemma}
\begin{proof}
	The proof is divided in two parts. In the first part we show that  $|\mathbf{\Theta}_{\Upsilon,\mathbf{t}}|\to0$ a.s.~as $|\mathbf{t}|\to\infty$ for $\mathbf{t}\in\bigcup_{\mathbf{s}\in\mathcal{L}}(\Upsilon)_{\mathbf{s}}$ and then that $\sum_{\mathbf{t}\in\mathcal{L}}\rho_{(\Upsilon)_{\mathbf{t}}}(\mathbf{\Theta})^{\alpha}<\infty$ a.s.
	
	From $|\mathbf{\Theta}_{\Upsilon,\mathbf{t}}|\to0$ a.s.~as $|\mathbf{t}|\to\infty$ for $\mathbf{t}\in\mathcal{H}$ by continuity we obtain that $\rho_{(\Upsilon)_{\mathbf{t}}}(\mathbf{\Theta})\to0$ a.s.~as $|\mathbf{t}|\to\infty$ for $\mathbf{t}\in\mathcal{G}$.	Let $\epsilon>0$. Observe that $\mathcal{L}=\mathcal{G}\cup-\mathcal{G}$ and that for every $0<c\leq1$
	\begin{multline*}
	\mathbb{P}\bigg(\bigcup_{\mathbf{t}\in\mathcal{G}}\Big\{\rho_{(\Upsilon)_{\mathbf{t}}}(\mathbf Y)\geq c>\sup\limits_{\mathbf{t}\prec\mathbf{z},\mathbf{z}\in\mathcal{G}}\rho_{(\Upsilon)_{\mathbf{z}}}(\mathbf Y)\Big\}\bigg)\\
	    =\sum_{\mathbf{t}\in\mathcal{G}}\mathbb{P}\Big(\rho_{(\Upsilon)_{\mathbf{t}}}(\mathbf Y)\geq c>\sup\limits_{\mathbf{t}\prec\mathbf{z},\mathbf{z}\in\mathcal{G}}\rho_{(\Upsilon)_{\mathbf{z}}}(\mathbf Y)\Big)= 1,
	\end{multline*}
	because $\mathbb{P}(\rho_{\Upsilon}(\mathbf Y)\geq1)=1$, and $\Upsilon\in\mathcal{H}$.
	
	Suppose that $\mathbb{P}(\sum_{\mathbf{h}\in-\mathcal{G}}\mathbf{1}(\rho_{(\Upsilon)_{\mathbf{h}}}(\mathbf Y)>\epsilon)=\infty)>0$. We have that
	\begin{multline*}
	\mathbb{P}(\sum_{\mathbf{h}\in-\mathcal{G}}\mathbf{1}(\rho_{(\Upsilon)_{\mathbf{h}}}(\mathbf Y)>\epsilon)=\infty)\\
	    =\sum_{\mathbf{t}\in\mathcal{G}}\mathbb{P}\Big(\sum_{\mathbf{h}\in-\mathcal{G}}\mathbf{1}(\rho_{(\Upsilon)_{\mathbf{h}}}(\mathbf Y)>\epsilon),\rho_{(\Upsilon)_{\mathbf{t}}}(\mathbf Y)\geq 1>\sup\limits_{\mathbf{t}\prec\mathbf{z},\mathbf{z}\in\mathcal{G}}\rho_{(\Upsilon)_{\mathbf{z}}}(\mathbf Y)\Big).
	\end{multline*}
	Consider any $\mathbf{t}\in\mathcal{G}$ s.t.
	\begin{equation*}
	\mathbb{P}\Big(\sum_{\mathbf{h}\in-\mathcal{G}}\mathbf{1}(\rho_{(\Upsilon)_{\mathbf{h}}}(\mathbf Y)>\epsilon),\rho_{(\Upsilon)_{\mathbf{t}}}(\mathbf Y)\geq 1>\sup\limits_{\mathbf{t}\prec\mathbf{z},\mathbf{z}\in\mathcal{G}}\rho_{(\Upsilon)_{\mathbf{z}}}(\mathbf Y)\Big)>0.
	\end{equation*}
	By the time change formula (\ref{timechangeY-New}) we get
	\begin{align*}
	\infty&=\mathbb{E}\bigg[\sum_{\mathbf{h}\in-\mathcal{G}}\mathbf{1}\Big(\rho_{(\Upsilon)_{\mathbf{h}}}(\mathbf Y)>\epsilon,\rho_{(\Upsilon)_{\mathbf{t}}}(\mathbf Y)\geq 1>\sup\limits_{\mathbf{t}\prec\mathbf{z},\mathbf{z}\in\mathcal{G}}\rho_{(\Upsilon)_{\mathbf{z}}}(\mathbf Y)\Big)\bigg]\\
	&	=\sum_{\mathbf{h}\in-\mathcal{G}}\mathbb{P}\bigg(\rho_{(\Upsilon)_{\mathbf{h}}}(\mathbf Y)>\epsilon,\rho_{(\Upsilon)_{\mathbf{t}}}(\mathbf Y)\geq 1>\sup\limits_{\mathbf{t}\prec\mathbf{z},\mathbf{z}\in\mathcal{G}}\rho_{(\Upsilon)_{\mathbf{z}}}(\mathbf Y)\bigg)\\
	&=\sum_{\mathbf{h}\in-\mathcal{G}}\int_{\epsilon}^{\infty}\mathbb{P}\Big(r\rho_{(\Upsilon)_{-\mathbf{h}}}(\mathbf{\Theta})>1,r\rho_{(\Upsilon)_{\mathbf{t}-\mathbf{h}}}(\mathbf{\Theta})\geq 1>r\sup\limits_{\mathbf{t}-\mathbf{h}\prec\mathbf{z},\mathbf{z}\in\mathcal{G}}\rho_{(\Upsilon)_{\mathbf{z}}}(\mathbf{\Theta})\Big)d(-r^{-\alpha})\\
	&	\stackrel{(r=q\epsilon)}{=}\epsilon^{-\alpha}\sum_{\mathbf{h}\in-\mathcal{G}}\int_{1}^{\infty}\mathbb{P}\Big(q\epsilon\rho_{(\Upsilon)_{-\mathbf{h}}}(\mathbf{\Theta})>1,q\epsilon\rho_{(\Upsilon)_{\mathbf{t}-\mathbf{h}}}(\mathbf{\Theta})\geq 1>q\epsilon\sup\limits_{\mathbf{t}-\mathbf{h}\prec\mathbf{z},\mathbf{z}\in\mathcal{G}}\rho_{(\Upsilon)_{\mathbf{z}}}(\mathbf{\Theta})\Big)d(-q^{-\alpha})\\
	&
	\leq\epsilon^{-\alpha}\sum_{\mathbf{h}\in-\mathcal{G}}\int_{1}^{\infty}\mathbb{P}\Big(\rho_{(\Upsilon)_{\mathbf{t}-\mathbf{h}}}(\mathbf{\Theta})\geq \frac{1}{q\epsilon}>\sup\limits_{\mathbf{t}-\mathbf{h}\prec\mathbf{z},\mathbf{z}\in\mathcal{G}}\rho_{(\Upsilon)_{\mathbf{z}}}(\mathbf{\Theta})\Big)d(-q^{-\alpha})\\
	&\leq \epsilon^{-\alpha}\int_{1}^{\infty}d(-q^{-\alpha})=\epsilon^{-\alpha}<\infty,
	\end{align*}
	where we used that for every $\mathbf{t}\in\mathcal{G}$ and $\mathbf{h}\in-\mathcal{G}$ we have that $\mathbf{t}-\mathbf{h}\in\mathcal{G}$. Thus, we have a contradiction and so $\rho_{(\Upsilon)_{\mathbf{t}}}(\mathbf{\Theta})\to0$ a.s.~as $|\mathbf{t}|\to\infty$ for $\mathbf{t}\in-\mathcal{G}$ which by homogeneity and continuity implies that $|\mathbf{\Theta}_{\Upsilon,\mathbf{t}}|\to0$ a.s.~as $|\mathbf{t}|\to\infty$ for $\mathbf{t}\in\bigcup_{\mathbf{s}\in\mathcal{L}}(\Upsilon)_{\mathbf{s}}$.
	Assume that $\mathbb{P}(\sum_{\mathbf{t}\in\mathcal{L}}\rho_{(\Upsilon)_{\mathbf{t}}}(\mathbf{\Theta})^{\alpha}=\infty)>0$. We have that 
	\begin{multline*}
	\mathbb{P}(\sum_{\mathbf{t}\in\mathcal{L}}\rho_{(\Upsilon)_{\mathbf{t}}}(\mathbf{\Theta})^{\alpha}=\infty)=\sum_{\mathbf{i}\in\mathcal{L}}\mathbb{P}(\sum_{\mathbf{t}\in\mathcal{L}}\rho_{(\Upsilon)_{\mathbf{t}}}(\mathbf{\Theta})^{\alpha}=\infty,\mathbf T^{*}_{\Upsilon,\mathcal{L}}=\mathbf{i})\\
	=\sum_{\mathbf{i}\in H}\mathbb{P}(\sum_{\mathbf{t}\in\mathcal{L}}\rho_{(\Upsilon)_{\mathbf{t}}}(\mathbf{\Theta})^{\alpha}=\infty,\mathbf T^{*}_{\Upsilon,\mathcal{L}}=\mathbf{i}),
	\end{multline*}
	where $H$ is the subset of $\mathcal{L}$ s.t.~$\mathbb{P}(\sum_{\mathbf{t}\in\mathcal{L}}\rho_{(\Upsilon)_{\mathbf{t}}}(\mathbf{\Theta})^{\alpha}=\infty,\mathbf T^{*}_{\Upsilon,\mathcal{L}}=\mathbf{i})>0$ for every $\mathbf{i}\in H$.  Let $\mathbf{i}\in H$, then
	\begin{equation*}
	\infty=\mathbb{E}\bigg[\sum_{\mathbf{t}\in\mathcal{L}}\rho_{(\Upsilon)_{\mathbf{t}}}(\mathbf{\Theta})^{\alpha}\mathbf{1}(\mathbf T^{*}_{\Upsilon,\mathcal{L}}=\mathbf{i})\bigg]=\sum_{\mathbf{t}\in\mathcal{L}}\mathbb{E}\bigg[\rho_{(\Upsilon)_{\mathbf{t}}}(\mathbf{\Theta})^{\alpha}\mathbf{1}(\mathbf T^{*}_{\Upsilon,\mathcal{L}}=\mathbf{i})\bigg].
	\end{equation*}
	Now, we generalise the arguments adopted in the proof of Lemma 3.3 in \cite{SW}. For each $\mathbf{i}\in\mathcal{L}$ define a function $g_{\mathbf{i}}:(\bar{\mathbb{R}}^{d})^{\mathbb{Z}^{k}}\to\mathbb{R}$ as follows. If $(\mathbf{\Theta}_{\Upsilon,\mathbf{s}}, \mathbf{s}\in\mathbb{Z}^{k})$ is such that
	\begin{equation*}
	|\mathbf{\Theta}_{\Upsilon,\mathbf{j}}|<|\mathbf{\Theta}_{\Upsilon,\mathbf{i}}|\quad\textnormal{for $\mathbf{j}\prec\mathbf{i}$ and $\mathbf{j}\in\mathcal{L}$},\quad |\mathbf{\Theta}_{\Upsilon,\mathbf{j}}|\leq|\mathbf{\Theta}_{\Upsilon,\mathbf{i}}|\quad\textnormal{for $\mathbf{j}\succeq\mathbf{i}$ and $\mathbf{j}\in\mathcal{L}$},
	\end{equation*}
	then set $g_{\mathbf{i}}(\mathbf{\Theta}_{\Upsilon,\mathbf{z}}, \mathbf{z}\in\mathbb{Z}^{k})=1$. Otherwise set $g_{\mathbf{i}}(\mathbf{\Theta}_{\Upsilon,\mathbf{z}}, \mathbf{z}\in\mathbb{Z}^{k})=0$. Then, by time change formula we have
	\begin{align*}
	\infty&=\sum_{\mathbf{t}\in\mathcal{L}}\mathbb{E}\bigg[\rho_{(\Upsilon)_{\mathbf{t}}}(\mathbf{\Theta})^{\alpha}\mathbf{1}(\mathbf T^{*}_{\Upsilon,\mathcal{L}}=\mathbf{i})\bigg]=\sum_{\mathbf{t}\in\mathcal{L}}\mathbb{E}\bigg[\rho_{(\Upsilon)_{\mathbf{t}}}(\mathbf{\Theta})^{\alpha}g_{\mathbf{i}}(\mathbf{\Theta}_{\Upsilon,\mathbf{s}}, \mathbf{z}\in\mathbb{Z}^{k})\bigg]\\
	& =\sum_{\mathbf{t}\in\mathcal{L}}\mathbb{E}\bigg[\rho_{(\Upsilon)_{\mathbf{t}}}(\mathbf{\Theta})^{\alpha}g_{\mathbf{i}}\bigg(\frac{\mathbf{\Theta}_{\Upsilon,\mathbf{z}}}{\rho_{(\Upsilon)_{\mathbf{t}}}(\mathbf{\Theta})}, \mathbf{z}\in\mathbb{Z}^{k}\bigg)\bigg]\\
	   &=\sum_{\mathbf{t}\in\mathcal{L}}\mathbb{E}\bigg[g_{\mathbf{i}}(\mathbf{\Theta}_{\Upsilon,\mathbf{z}-\mathbf{t}}, \mathbf{z}\in\mathbb{Z}^{k})\mathbf{1}(\rho_{(\Upsilon)_{-\mathbf{t}}}(\mathbf{\Theta})\neq\mathbf{0})\bigg]\\
	   &
	\leq \sum_{\mathbf{t}\in\mathcal{L}}\mathbb{E}\bigg[g_{\mathbf{i}}(\mathbf{\Theta}_{\Upsilon,\mathbf{z}-\mathbf{t}}, \mathbf{z}\in\mathbb{Z}^{k})\bigg]
	=\sum_{\mathbf{t}\in\mathcal{L}}\mathbb{E}\bigg[\mathbf{1}(\mathbf T_{\Upsilon,\mathcal{L}}^{*}=\mathbf{i}-\mathbf{t})\bigg]=1,
	\end{align*}
	which is a contradiction. Notice that we used the fact that by construction, for every $\mathbf{i},\mathbf{t}\in\mathcal{L}$, we have $\mathbf{i}-\mathbf{t}\in\mathcal{L}$.
	
	Thus, we have $\sum_{\mathbf{t}\in\mathcal{L}}\rho_{(\Upsilon)_{\mathbf{t}}}(\mathbf{\Theta})^{\alpha}<\infty$ a.s., and by homogeneity and continuity we have that $\sum_{\mathbf{t}\in\mathcal{L}}\max_{\mathbf{s}\in\Upsilon}|\mathbf{\Theta}_{\Upsilon,\mathbf{t}+\mathbf{s}}|^{\alpha}<\infty$ a.s., and since $\Upsilon$ is finite we obtain that $\sum_{\mathbf{t}\in\bigcup_{\mathbf{s}\in\mathcal{L}}(\Upsilon)_{\mathbf{s}}}|\mathbf{\Theta}_{\Upsilon,\mathbf{t}}|^{\alpha}<\infty$ a.s.
\end{proof}

	Let $j\in\mathbb{N}$. From similar arguments as the ones used in the proof of Proposition \ref{lem-bound-L}, we have that for every $\mathbf{t}\in\mathcal{D}_{j}$ (and so $\mathbf{t}\in\mathcal{D}_{j}\cap K_{p}$ for some $j,p\in\mathbb{N}$) there exists a $n_{\mathbf{t}}\in\mathbb{N}$ s.t.~$\mathbf{t}\in R^{(j)}_{0,\Lambda_{m}}$ for every $m>n_{\mathbf{t}}$. Further, choose $(d_{n})_{n\in\mathbb{N}}$ such that $d_{n}$ is the highest integer s.t.~$\max_{|\mathbf{t}|\leq d_{n},\mathbf{t}\in\mathcal{D}_{j}}n_{\mathbf{t}}< r_{n}$. It is possible to see that $d_{n}\to\infty$ as $n\to\infty$ and that for every $2l<r_{n}$
	\begin{align*}
	\lefteqn{\mathbb{P}\Big(\max_{2l\leq |\mathbf{t}|\leq d_{n}\,|\,\mathbf{t}\in \mathcal{D}_{j} }|\mathbf{X}_{\mathbf{t}}|>a_{n}^\Lambda x\,\big|\max\limits_{\mathbf{t}\in\mathcal{E}_j}|\mathbf{X}_{\mathbf{t}}|>a_{n}^\Lambda x\Big)}\\
	    &=\mathbb{P}\Big(\max_{l\leq |\mathbf{t}|\leq d_{n}\,|\,\mathbf{t}\in \mathcal{D}_{j} \,|\,\mathbf{t}\in R^{(j)}_{2l,\Lambda_{r_{n}}} }|\mathbf{X}_{\mathbf{t}}|>a_{n}^\Lambda x\,\big|\max\limits_{\mathbf{t}\in\mathcal{E}_j}|\mathbf{X}_{\mathbf{t}}|>a_{n}^\Lambda x\Big)\\
	&\leq\mathbb{P}\Big(\max_{\mathbf{t}\in R^{(j)}_{2l,\Lambda_{r_{n}}} }|\mathbf{X}_{\mathbf{t}}|>a_{n}^\Lambda x\,\big|\max\limits_{\mathbf{t}\in\mathcal{E}_j}|\mathbf{X}_{\mathbf{t}}|>a_{n}^\Lambda x\Big).
	\end{align*}	
	Since
	\begin{align*}
	\lefteqn{\mathbb{P}\Big(\max_{2l\leq |\mathbf{t}|\leq d_{n}\,|\,\mathbf{t}\in\mathcal{D}_{j} }|\mathbf{X}_{\mathbf{t}}|>C_{j}a_{n}^\Lambda x\,\big|\,\rho_{\mathcal{E}_{j}}(\mathbf{X})>a_{n}^\Lambda x\Big)}\\
	    =&\frac{\mathbb{P}\Big(\max_{2l\leq |\mathbf{t}|\leq d_{n}\,|\,\mathbf{t}\in\mathcal{D}_{j} }|\mathbf{X}_{\mathbf{t}}|>C_{j}a_{n}^\Lambda x,\rho_{\mathcal{E}_{j}}(\mathbf{X})>a_{n}^\Lambda x\Big)}{\mathbb{P}(\rho_{\mathcal{E}_{j}}(\mathbf{X})>a_{n}^\Lambda x)}\\
	\leq&\frac{\mathbb{P}\Big(\max_{2l\leq |\mathbf{t}|\leq d_{n}\,|\,\mathbf{t}\in\mathcal{D}_{j} }|\mathbf{X}_{\mathbf{t}}|>C_{j}a_{n}^\Lambda x,\max_{\mathbf{s}\in\mathcal{E}_{j}}|\mathbf{X}_{\mathbf{s}}|>C_{j}a_{n}^\Lambda x\Big)}{\mathbb{P}(\rho_{\mathcal{E}_{j}}(\mathbf{X})>a_{n}^\Lambda x)}\\
=&\mathbb{P}\Big(\max_{2l\leq |\mathbf{t}|\leq d_{n}\,|\,\mathbf{t}\in\mathcal{D}_{j} }|\mathbf{X}_{\mathbf{t}}|>C_{j}a_{n}^\Lambda x|\max_{\mathbf{s}\in\mathcal{E}_{j}}|\mathbf{X}_{\mathbf{s}}|>C_{j}a_{n}^\Lambda x\Big)\\
&\frac{\mathbb{P}(\max_{\mathbf{s}\in\mathcal{E}_{j}}|\mathbf{X}_{\mathbf{s}}|>C_{j}a_{n}^\Lambda x)}{\mathbb{P}(\max_{\mathbf{s}\in\mathcal{E}_{j}}|\mathbf{X}_{\mathbf{s}}|>a_{n}^\Lambda x)}\frac{\mathbb{P}(\max_{\mathbf{s}\in\mathcal{E}_{j}}|\mathbf{X}_{\mathbf{s}}|>a_{n}^\Lambda x)}{\mathbb{P}(\rho_{\mathcal{E}_{j}}(\mathbf{X})>a_{n}^\Lambda x)},
	\end{align*}
	by condition (AC$^{\Lambda}_{\succeq,I^*}$) we obtain the following anti-clustering condition:
	\begin{equation}\label{AC-UionionD-New}
	\lim\limits_{l\to\infty}\limsup_{n\to\infty}\mathbb{P}\Big(\max_{2l\leq |\mathbf{t}|\leq d_{n}\,|\,\mathbf{t}\in\mathcal{D}_{j} }|\mathbf{X}_{\mathbf{t}}|>C_{j}a_{n}^\Lambda x\,\big|\rho_{\mathcal{E}_{j}}(\mathbf{X})>a_{n}^\Lambda x\Big)=0.
	\end{equation}
	
	Now, for any $z>0$, by the regular variation of $\rho_{\mathcal{E}_{j}}(\mathbf{X})$ (namely of $\max_{\mathbf{s}\in\mathcal{E}_{j}}|\mathbf{X}_{\mathbf{s}}|$) and by (\ref{AC-UionionD-New}) we have that 
	\begin{equation*}
	\lim\limits_{l\to\infty}\limsup_{n\to\infty}\mathbb{P}\Big(\max_{2l\leq |\mathbf{t}|\leq d_{n}\,|\,\mathbf{t}\in\mathcal{D}_{j} }|\mathbf{X}_{\mathbf{t}}|>za_{n}^\Lambda x\,\big|\rho_{\mathcal{E}_{j}}(\mathbf{X})>a_{n}^\Lambda x\Big)=0.
	\end{equation*}
	In other words, for any $\epsilon>0$ and $z>0$, there exists $l>0$ such that for all $w>l$
	\begin{equation*}
	\mathbb{P}\bigg(\max_{l\leq |\mathbf{t}|\leq w\,|\,\mathbf{t}\in\mathcal{D}_{j}  }|\mathbf{Y}_{\mathcal{E}_{j}}(\mathbf{t})|>z\bigg)\leq \epsilon.
	\end{equation*}
	This, implies that $\mathbb{P}(\lim\limits_{|\mathbf{t}|\to\infty}|\mathbf{Y}_{\mathcal{E}_{j}}(\mathbf{t})|=0)=1$ and so $\mathbb{P}(\lim\limits_{|\mathbf{t}|\to\infty}|\mathbf{\Theta}_{\mathcal{E}_{j}}(\mathbf{t})|=0)=1$ for $\mathbf{t}\in\mathcal{D}_{j}$. The argument holds for every $j\in\mathbb{N}$. Since $\tilde{\mathcal{D}}_{j}\subset\mathcal{D}_{j}$, from Lemma \ref{lem-AC1-L-New} we obtain the statement.

\section{Proofs in Section \ref{sec:appl}}\label{sec:proof3}

\subsection{Proofs in Section \ref{Sec:Extremal-index}}

Since in Section \ref{Sec:Extremal-index} the $\mathbb{R}^d$-valued stationary random field $(\mathbf{X}_\mathbf{t})_{\mathbf{t}\in\mathbb{Z}^{k}}$ is always considered in modulus and since $(|\mathbf{X}_\mathbf{t}|)_{\mathbf{t}\in\mathbb{Z}^{k}}$ is stationary and regularly varying, it is sufficient to prove the results for a non-negative valued stationary random field $(X_\mathbf{t})_{\mathbf{t}\in\mathbb{Z}^{k}}$, as we do for the remaining proofs.

\begin{theorem}\label{thm-u}
	Consider the following conditions:
	\\\textnormal{(I)} $(X_\mathbf{t})_{\mathbf{t}\in\mathbb{Z}^{k}}$ is a real valued stationary random field whose marginal distribution $F$ does not have an atom at the right endpoint $x_{F}$.
	\\\textnormal{(II)} For a sequence $u_{n}\uparrow x_{F}$ and an integer sequence $r_{n}\to\infty$ s.t.~$k_{n}=[|\Lambda_n|/|\Lambda_{r_{n}}|]\to\infty$ the following anti-clustering condition is satisfied:
	\begin{equation}\label{AC-u}
	\lim\limits_{l\to\infty}\limsup\limits_{n\to\infty}\mathbb{P}(\hat{M}^{\Lambda,X}_{l,r_{n}}>u_{n}\,|\, X_\mathbf{0}>u_{n})=0.
	\end{equation}
	\noindent\textnormal{(III)} A mixing condition holds:
	\begin{equation}\label{mixing-u}
	\mathbb{P}\big(\max_{\mathbf{t}\in\Lambda_{n} }X_\mathbf{t}\leq u_{n}\big)-(\mathbb{P}\big(\max_{\mathbf{t}\in\Lambda_{r_{n}} }X_\mathbf{t}\leq u_{n})^{k_{n}}\big)\to0,\quad n\to\infty,
	\end{equation}
	where $(u_{n})$, $(k_{n})$ and $(r_{n})$ are as in \textnormal{(II)}.
	\\\textnormal{(IV)} For any $\tau\geq0$ there exists a sequence $(u_{n})=(u_{n}(\tau))$ s.t.~$\lim\limits_{n\to\infty}|\Lambda_n|\mathbb{P}(X_\mathbf{0}>u_{n}(\tau))=\tau$ and \textnormal{(II)} and \textnormal{(III)} are satisfied for these sequences $(u_{n})$.
	
	Then, the following statements hold:
	\\\textnormal{(a)} If \textnormal{(I)} and \textnormal{(II)} are satisfied then
	\begin{equation}\label{I}
	\lim\limits_{l\to\infty}\limsup\limits_{n\to\infty}\bigg|\theta^{\Lambda}_{n}-\sum_{j=1}^{\infty}\lambda_{j}\mathbb{P}(\max_{\mathbf{t}\in \mathcal{D}_{j}\cap K_{l} }X_\mathbf{t}\leq u_{n}\,|\, X_\mathbf{0}>u_{n})\bigg|=0,
	\end{equation}
	and $\liminf\limits_{n\to\infty}\theta^{\Lambda}_{n}>0$.
	\\\textnormal{(b)} If \textnormal{(I)} and \textnormal{(IV)} are satisfied and $\theta^{\Lambda}_{b}=\lim\limits_{n\to\infty}\theta^{\Lambda}_{n}$ exists, then $\theta^{\Lambda}_{X}\in(0,1]$ exists and $\theta^{\Lambda}_{X}=\theta^{\Lambda}_{b}$.
\end{theorem}
\begin{remark}
	Notice that when $u_{n}=a_{n}x$ then (\ref{AC-u}) is the (AC$^{\Lambda}_{\succ}$) condition.
\end{remark}	
\begin{proof}
	Let us first focus on (\ref{I}). Denote by $t_{|\Lambda_{r_{n}}|}$ the highest element of $\Lambda_{r_{n}}$ according to $\prec$, by $t_{|\Lambda_{r_{n}}|-1}$ the second highest one, ..., by $t_{1}$ the lowest one. Further, for $m=1,...,|\Lambda_{r_{n}}|$ let $\mathcal{M}_{m}:=\max_{j=m,...,|\Lambda_{r_{n}}|}X_{t_{j}}$ and for $m=|\Lambda_{r_{n}}|+1$ let $\mathcal{M}_{m}:=0$. Thus, we have $\mathbb{P}(\mathcal{M}_{|\Lambda_{r_{n}}|+1}>u_{n})=0$ and
	\begin{equation*}
	\mathbb{P}(\max_{\mathbf{t}\in\Lambda_{r_{n}} }X_\mathbf{t}> u_{n})=\sum_{m=2}^{|\Lambda_{r_{n}}|+1}-\mathbb{P}(\mathcal{M}_{m}>u_{n})+\mathbb{P}(\mathcal{M}_{m-1}>u_{n}).
	\end{equation*}
	Consider $l\in\mathbb{N}$ with $l\ll|\Lambda_{r_{n}}|$ and for $m=2,...,|\Lambda_{r_{n}}|$ let 
	\begin{multline}
	    \mathcal{M}^{\circ}_{m}:=\max_{\mathbf{t}\in\{t_{m}-t_{m-1},...,t_{|\Lambda_{r_{n}}|}-t_{m-1}\}}X_\mathbf{t} \text{ and }\\
	    \label{Mm}\mathcal{M}^{\circ}_{m\setminus l}:=\max_{\mathbf{t}\in\{t_{m}-t_{m-1},...,t_{|\Lambda_{r_{n}}|}-t_{m-1}\}\setminus K_{l}}X_\mathbf{t}.
	\end{multline}
	We have
	\begin{multline}\nonumber
	-\mathbb{P}(\mathcal{M}_{m}>u_{n})+\mathbb{P}(\mathcal{M}_{m-1}>u_{n})=\\\label{Mm}-\mathbb{P}(\mathcal{M}^{\circ}_{m}>u_{n})+\mathbb{P}(\mathcal{M}^{\circ}_{m}\vee X_{\mathbf{0}}>u_{n})=\mathbb{P}(\mathcal{M}^{\circ}_{m}\leq u_{n},X_{\mathbf{0}}>u_{n})
	\end{multline}
	Notice that for $m=|\Lambda_{r_{n}}|+1$ we have that
	\begin{equation*}
	-\mathbb{P}(\mathcal{M}_{m}>u_{n})+\mathbb{P}(\mathcal{M}_{m-1}>u_{n})=\mathbb{P}(X_{\mathbf{0}}>u_{n})
	\end{equation*}
	and so when divided by $|\Lambda_{r_n}|\mathbb{P}(X>u_{n})$ is asymptotically negligible.
	
	Now, for each $j,n\in\mathbb{N}$ consider the points $t_{m}$, for $m=2,...,|\Lambda_{r_{n}}|$, such that $\{t_{m}-t_{m-1},...,t_{|\Lambda_{r_{n}}|}-t_{m-1}\}\cap K_{l}= \mathcal{D}_{j}\cap K_{l}$. For such points we have that (\ref{Mm}) is equal to
	\begin{equation*}
	\mathbb{P}(\mathcal{M}^{\circ}_{m}\leq u_{n},X_{\mathbf{0}}>u_{n},\max_{\mathbf{t}\in \mathcal{D}_{j}\cap K_{l} }X_\mathbf{t}\leq u_{n})+\mathbb{P}(\mathcal{M}^{\circ}_{m}\leq u_{n},X_{\mathbf{0}}>u_{n},\max_{\mathbf{t}\in \mathcal{D}_{j}\cap K_{l} }X_\mathbf{t}> u_{n})
	\end{equation*}
	\begin{equation*}
	=\mathbb{P}(\mathcal{M}^{\circ}_{m}\leq u_{n},X_{\mathbf{0}}>u_{n},\max_{\mathbf{t}\in \mathcal{D}_{j}\cap K_{l} }X_\mathbf{t}\leq u_{n})=\mathbb{P}(\mathcal{M}^{\circ}_{m\setminus l}\leq u_{n},X_{\mathbf{0}}>u_{n},\max_{\mathbf{t}\in \mathcal{D}_{j}\cap K_{l} }X_\mathbf{t}\leq u_{n})
	\end{equation*}
	\begin{equation*}
	=\mathbb{P}(X_{\mathbf{0}}>u_{n},\max_{\mathbf{t}\in \mathcal{D}_{j}\cap K_{l} }X_\mathbf{t}\leq u_{n})-\mathbb{P}(\mathcal{M}^{\circ}_{m\setminus l}> u_{n},X_{\mathbf{0}}>u_{n},\max_{\mathbf{t}\in \mathcal{D}_{j}\cap K_{l} }X_\mathbf{t}\leq u_{n}).
	\end{equation*}
	and that 
	\begin{multline*}
	\mathbb{P}(\mathcal{M}^{\circ}_{m\setminus l}> u_{n},X_{\mathbf{0}}>u_{n},\max_{\mathbf{t}\in \mathcal{D}_{j}\cap K_{l} }X_\mathbf{t}\leq u_{n})\leq \mathbb{P}(\mathcal{M}^{\circ}_{m\setminus l}> u_{n},X_{\mathbf{0}}>u_{n})\\
	    \leq \mathbb{P}(\hat{M}_{l,r_{n}}^{X}> u_{n},X_{\mathbf{0}}>u_{n}).
	\end{multline*}
	
	For the points $t_{m}$, for $m=2,...,|\Lambda_{r_{n}}|$, such that $\{t_{m}-t_{m-1},...,t_{|\Lambda_{r_{n}}|}-t_{m-1}\}\cap K_{l}\neq \mathcal{D}_{j}\cap K_{l}$ for every $j\in\mathbb{N}$, we will use that
	\begin{equation*}
	\frac{\mathbb{P}(\mathcal{M}^{\circ}_{m}\leq u_{n},X_{\mathbf{0}}>u_{n})}{|\Lambda_{r_{n}}|\mathbb{P}(X>u_{n})}\leq \frac{1}{|\Lambda_{r_{n}}|}
	\end{equation*}
	
	Observe that there are finitely many different subsets of $K_{l}$ and, following the notation of the proof of Theorem \ref{t1-L}, we denote their total number by $\tau_{l}$ and denote them by $\Xi_{l}^{(1)},....,\Xi_{l}^{(\tau_{l})}$. Further, for $z=1,...,\tau_{l}$, we let $\mu_{r_{n}}^{(z)}$ be the number of points $t_{m}$, $m=2,...,|\Lambda_{r_{n}}|$, such that $\{t_{m}-t_{m-1},...,t_{|\Lambda_{r_{n}}|}-t_{m-1}\}\cap K_{l}=\Xi_{l}^{(z)}$, that is $\mu_{r_{n}}^{(z)}=|\{t_{m},m=2,...,|\Lambda_{r_{n}}|:\{t_{m}-t_{m-1},...,t_{|\Lambda_{r_{n}}|}-t_{m-1}\}\cap K_{l}=\Xi_{l}^{(z)}\}|$. Recall that $\mathfrak{I}_{l}^{(z)}=\{i\in\mathbb{N}:\Xi_{l}^{(z)}=\mathcal{D}_{i}\cap K_{l}\}$. Then, by (i) and (ii) and in particular by point (I) in Proposition \ref{lem-AC1-L-2} we have
	\begin{equation*}
	\frac{\mu_{r_{n}}^{(z)}}{|\Lambda_{r_{n}}|}\to\sum_{i\in \mathfrak{I}_{l}^{(z)}}\lambda_{i},\quad\quad\textnormal{as $n\to\infty$}.
	\end{equation*}
	Notice that if $\Xi_{l}^{(z)}\neq \mathcal{D}_{i}\cap K_{l} $, for every $i\in\mathbb{N}$, then $\mathfrak{I}_{l}^{(z)}$ is empty and so $\frac{\mu_{r_{n}}^{(z)}}{|\Lambda_{r_n}|}\to0$, as $n\to\infty$, and we let $Z_{l}$ the subset of $\{1,...,\tau_{l}\}$ of such $z$s.
	Further, for $z\in\{1,...,\tau_{l}\}\setminus Z_{l}$ we let $\mathcal{D}_{j(z)}$ indicate the (or one of the) $\mathcal{D}_{i}$ such that $\Xi_{l}^{(z)}= \mathcal{D}_{i}\cap K_{l} $. Thus, 
	\begin{equation*}
	\sum_{z\in\{1,...,\tau_{l}\}\setminus Z_{l}}\frac{\mu_{r_{n}}^{(z)}}{|\Lambda_{r_n}|}\stackrel{n\to\infty}{\to}\sum_{z\in\{1,...,\tau_{l}\}\setminus Z_{l}}\sum_{i\in \mathfrak{I}_{l}^{(z)}}\lambda_{i}= \sum_{i=1}^{\infty}\lambda_{i}=1.
	\end{equation*}
	Hence, we have that
	\begin{equation*}
	\theta^{\Lambda}_{n}=\frac{\mathbb{P}(\max_{\mathbf{t}\in\Lambda_{r_{n}} }X_\mathbf{t}>u_{n})}{|\Lambda_{r_{n}}|\mathbb{P}(X>u_{n})}= \sum_{z\in\{1,...,\tau_{l}\}\setminus Z_{l}}\frac{\mu_{r_{n}}^{(z)}}{|\Lambda_{r_{n}}|}\mathbb{P}(\max_{\mathbf{t}\in \mathcal{D}_{j(z)}\cap K_{l} }X_\mathbf{t}\leq u_{n}|X_{\mathbf{0}}>u_{n})+A_{l,n}
	\end{equation*}
	where $A_{l,n}$ is such that
	\begin{equation*}
	|A_{l,n}|\leq \sum_{z\in\{1,...,\tau_{l}\}\setminus Z_{l}}\frac{\mu_{r_{n}}^{(z)}}{|\Lambda_{r_{n}}|}\mathbb{P}(\hat{M}_{l,r_{n}}^{X}> u_{n}|X_{\mathbf{0}}>u_{n}) + \sum_{z\in Z_{l}}\frac{\mu_{r_{n}}^{(z)}}{|\Lambda_{r_{n}}|}.
	\end{equation*}
	Therefore, applying (\ref{AC-u}) we obtain that (\ref{I}).
	
	To show that $\liminf\limits_{n\to\infty}\theta^{\Lambda}_{n}>0$ we proceed as follows. Consider the a set of points composed by points $\{t_{1},...,t_{|\Lambda_{r_{n}}|}\}$ which have a supremum distance of at least $2l$, and denote it $W_{n}$ and its points (in increasing order according to $\succ$) by $w_{1},...,w_{p_{n}}$ for some $p_{n}\in\mathbb{N}$. Observe that the sets 
	\begin{equation*}
\left\{X_{w_{m}}>u_{n},\max_{\mathbf{t}\succeq w_{m},\mathbf{t}\in\{t_{1},...,t_{|\Lambda_{r_{n}}|}\}\setminus K_{l-1}(w_{m})}X_{\mathbf{t}}\leq u_{n} \right\}\,,
	\end{equation*}for $m=1,...,p_{n}$, are disjoint and their union is a subset of $\{\max_{\mathbf{t}\in\Lambda_{r_{n}} }X_{\mathbf{t}}> u_{n}\}$. Observe also that $|\Lambda_{r_{n}}|\geq|W_{n}|\geq\lfloor|\Lambda_{r_{n}}|/(2l)^{k}\rfloor$. Hence, we have
	\begin{align*}
	\theta^{\Lambda}_{n}&=\frac{\mathbb{P}(\max_{\mathbf{t}\in\Lambda_{r_{n}} }X_{\mathbf{t}}>u_{n})}{|\Lambda_{r_{n}}|\mathbb{P}(X>u_{n})}\\
	&\geq \sum_{m=1}^{p_{n}}\frac{\mathbb{P}(X_{w_{m}}>u_{n},\max_{\mathbf{t}\succeq w_{m},\mathbf{t}\in\{t_{1},...,t_{|\Lambda_{r_{n}}|}\}\setminus K_{l-1}(w_{m})}X_\mathbf{t}\leq u_{n})}{|\Lambda_{r_{n}}|\mathbb{P}(X>u_{n})}\\
&	=\sum_{m=1}^{p_{n}}\frac{\mathbb{P}(X_\mathbf{0}>u_{n},\max_{\mathbf{t}\succeq \mathbf{0},\mathbf{t}\in\{t_{1}-w_{m},...,t_{|\Lambda_{r_{n}}|}-w_{m}\}\setminus K_{l-1}}X_\mathbf{t}\leq u_{n})}{|\Lambda_{r_{n}}|\mathbb{P}(X>u_{n})}\\
	&\geq \sum_{m=1}^{p_{n}}\frac{\mathbb{P}(X_{\mathbf{0}}>u_{n},\hat{M}_{l-1,r_{n}}\leq u_{n})}{|\Lambda_{r_{n}}|\mathbb{P}(X>u_{n})}\\
&=\frac{p_{n}}{|\Lambda_{r_{n}}|}[1-\mathbb{P}(\hat{M}_{l-1,r_{n}}> u_{n}|X_{\mathbf{0}}>u_{n})]\\
&\geq\frac{\lfloor|\Lambda_{r_{n}}|/(2l)^{k}\rfloor}{|\Lambda_{r_{n}}|}[1-\mathbb{P}(\hat{M}_{l-1,r_{n}}> u_{n}|X_{\mathbf{0}}>u_{n})].
	\end{align*}
	Then
	\begin{equation*}
	\liminf\limits_{n\to\infty}\theta_{n}=\frac{1}{(2l)^{k}}[1-\limsup\limits_{n\to\infty}\mathbb{P}(\hat{M}_{l-1,r_{n}}> u_{n}|X_{\mathbf{0}}>u_{n})].
	\end{equation*}
	This proves point (a).
	
	Now, assume that $\theta^{\Lambda}_{b}=\lim\limits_{n\to\infty}\theta^{\Lambda}_{n}$ exists. We need to show that $\theta^{\Lambda}_{X}=\theta^{\Lambda}_{b}$. By Taylor expansion we have
	\begin{align*}
	(\mathbb{P}(\max_{\mathbf{t}\in\Lambda_{r_{n}} }X_{\mathbf{t}}\leq u_{n}))^{k_{n}}&=\exp(k_{n}\log(1-\mathbb{P}(\max_{\mathbf{t}\in\Lambda_{r_{n}} }X_{\mathbf{t}}> u_{n})))\\
&	=\exp\Big(-\frac{|\Lambda_{n}|}{|\Lambda_{r_{n}}|}\mathbb{P}(\max_{\mathbf{t}\in\Lambda_{r_{n}} }X_{\mathbf{t}}> u_{n})(1+o(1))\Big)\\
&=\exp\Big(-\frac{\tau}{|\Lambda_{r_{n}}|}\frac{\mathbb{P}(\max_{\mathbf{t}\in\Lambda_{r_{n}} }X_{\mathbf{t}}> u_{n})}{\mathbb{P}(X>u_{n})}(1+o(1))\Big)\\
	&=\exp\Big(-\tau\theta^\Lambda_{n}(1+o(1))\Big)\to e^{-\theta_{b}\tau},\quad\text{as $n\to\infty$}.
	\end{align*}
	Hence, by the mixing condition (\ref{mixing-u}) we have
	\begin{multline*}
	\mathbb{P}(\max_{\mathbf{t}\in\Lambda_{n} }X_{\mathbf{t}}\leq u_{n})=[\mathbb{P}(\max_{\mathbf{t}\in\Lambda_{n} }X_{\mathbf{t}}\leq u_{n})-(\mathbb{P}(\max_{\mathbf{t}\in\Lambda_{r_{n}} }X_{\mathbf{t}}\leq u_{n}))^{k_{n}}]\\
	+(\mathbb{P}(\max_{\mathbf{t}\in\Lambda_{r_{n}} }X_{\mathbf{t}}\leq u_{n}))^{k_{n}}\to e^{-\theta_{b}\tau},\quad\text{as $n\to\infty$}.
	\end{multline*}
	Since this holds for any $\tau>0$, we conclude that $\theta^{\Lambda}_{X}=\theta^{\Lambda}_{b}$.
\end{proof}

\begin{proof}[Proof of Theorem \ref{co-u}]
	Recall (\ref{I}) and let $\theta^{(l)}:=\lim\limits_{n\to\infty}\sum_{j=1}^{\infty}\lambda_{j}\mathbb{P}(\max_{\mathbf{t}\in \mathcal{D}_{j}\cap K_{l} }X_\mathbf{t}\leq u_{n}\,|\, X_{\mathbf{0}}>u_{n})$, for $l\in\mathbb{N}$. By the continuous mapping theorem (and noticing that the sum is actually a finite sum since there are finitely many different combination of points inside $K_{l}$ for given $l$) we have $\theta^{(l)}=\sum_{j=1}^{\infty}\lambda_{j}\mathbb{P}(\max_{\mathbf{t}\in \mathcal{D}_{j}\cap K_{l} }Y_{\mathbf{t}}\leq 1)$ and by monotonicity of the probability measure $\theta^{(l)}\downarrow\sum_{j=1}^{\infty}\lambda_{j}\mathbb{P}(\sup_{\mathbf{t}\in \mathcal{D}_{j} }Y_{\mathbf{t}}\leq 1)$ as $l\to\infty$. Given that $\lim\limits_{l\to\infty}\theta^{(l)}=\theta^{\Lambda}_{b}$ exists, Theorem \ref{thm-u} point (a) ensures that for $(u_{n})=(u_{n}(\tau))$ and some $\tau>0$ we have that $\theta^{\Lambda}_{b}=\lim\limits_{n\to\infty}\theta^{\Lambda}_{n}$ exists and is positive. Then, from Theorem \ref{thm-u} point (b) for $(u_{n})=(u_{n}(\tau))$ and arbitrary $\tau>0$ we obtain that $\theta^{\Lambda}_{X}$ exists, is positive and it is equal to $\theta^{\Lambda}_{b}$, hence we obtain point (2). Moreover, from these arguments we immediately obtain the first equality in (\ref{u-index}), while for the others, using $\Theta_{\mathbf{0}}\stackrel{a.s.}{=}1$, we have that
	\begin{align*}
	\theta^{\Lambda}_{b}&=\sum_{j=1}^{\infty}\lambda_{j}\mathbb{P}(Y\sup_{\mathbf{t}\in \mathcal{D}_{j} }|\mathbf\Theta_\mathbf{t}|\leq 1)=\sum_{j=1}^{\infty}\lambda_{j}\bigg(1-\int_{1}^{\infty}\mathbb{P}(y\sup_{\mathbf{t}\in\mathcal{D}_{j}}\Theta_{\mathbf{t}}> 1)d(-y^{-\alpha})\bigg)\\
	&=\sum_{j=1}^{\infty}\lambda_{j}\bigg(1-\int_{0}^{1}\mathbb{P}(\sup_{\mathbf{t}\in\mathcal{D}_{j}}\Theta_{\mathbf{t}}^{\alpha}> u)du\bigg)=\sum_{j=1}^{\infty}\lambda_{j}\bigg(1-\mathbb{E}\Big[\sup_{\mathbf{t}\in\mathcal{D}_{j}}\Theta_{\mathbf{t}}^{\alpha}\wedge 1 \Big]\bigg)\\
	&=\sum_{j=1}^{\infty}\lambda_{j}\bigg(\mathbb{E}\Big[\Big(1-\sup_{\mathbf{t}\in\mathcal{D}_{j}}\Theta_{\mathbf{t}}^{\alpha}\Big)_{+} \Big]\bigg)=\sum_{j=1}^{\infty}\lambda_{j}\bigg(\mathbb{E}\Big[\sup_{\mathbf{t}\in\mathcal{D}_{j}\cup\{\mathbf{0}\}}\Theta_{\mathbf{t}}^{\alpha}-\sup_{\mathbf{t}\in\mathcal{D}_{j}}\Theta_{\mathbf{t}}^{\alpha} \Big]\bigg).
	\end{align*}
\end{proof}
\begin{theorem}\label{thm-u-I}
	Consider the following conditions:
	\\\textnormal{(I)} $(X_\mathbf{t})_{\mathbf{t}\in\mathbb{Z}^{k}}$ is a real valued stationary random field whose marginal distribution $F$ does not have an atom at the right endpoint $x_{F}$.
	\\\textnormal{(II)} For a sequence $u_{n}\uparrow x_{F}$ and an integer sequence $r_{n}\to\infty$ s.t.~$k_{n}=[|\Lambda_n|/|\Lambda_{r_{n}}|]\to\infty$ the following anti-clustering condition is satisfied: for every $j\in I^*$
	\begin{equation}\label{AC-u-I}
	\lim\limits_{l\to\infty}\limsup\limits_{n\to\infty}\mathbb{P}(\hat{M}^{\Lambda,X,(j)}_{2l,r_{n}}>u_{n}\,|\, \max\limits_{\mathbf{t}\in\mathcal{E}_j}X_\mathbf{t}>u_{n})=0.
	\end{equation}
	\noindent\textnormal{(III)} A mixing condition holds:
	\begin{equation}\label{mixing-u-I}
	\mathbb{P}(\max_{\mathbf{t}\in\Lambda_{n} }X_\mathbf{t}\leq u_{n})-(\mathbb{P}(\max_{\mathbf{t}\in\Lambda_{r_{n}} }X_\mathbf{t}\leq u_{n}))^{k_{n}}\to0,\quad n\to\infty,
	\end{equation}
	where $(u_{n})$, $(k_{n})$ and $(r_{n})$ are as in \textnormal{(II)}.
	\\\textnormal{(IV)} For any $\tau\geq0$ there exists a sequence $(u_{n})=(u_{n}(\tau))$ s.t.~$\lim\limits_{n\to\infty}|\Lambda_n|\mathbb{P}(X_\mathbf{0}>u_{n}(\tau))=\tau$ and \textnormal{(II)} and \textnormal{(III)} are satisfied for these sequences $(u_{n})$.
	
	Then, the following statements hold:
	\\\textnormal{(a)} If \textnormal{(I)} and \textnormal{(II)} are satisfied then
	\begin{equation}\label{I-new}
	\lim\limits_{l\to\infty}\limsup\limits_{n\to\infty}\bigg|\theta^{\Lambda}_{n}-\sum_{h\in I^*,h<m_{4l}}\mathbb{P}(\max_{\mathbf{t}\in \tilde{\mathcal{D}}_{h}\cap K_{2l} }X_\mathbf{t}\leq u_{n}|\max_{\mathbf{t}\in\mathcal{E}_{h}}X_\mathbf{t}>u_{n})\frac{|S'_{h,4l}|}{|\Lambda_{r_n}|}\frac{\mathbb{P}(\max_{\mathbf{t}\in\mathcal{E}_{h}}X_\mathbf{t}>u_{n})}{\mathbb{P}(X>u_{n})}\bigg|=0,
	\end{equation}
	and $\liminf\limits_{n\to\infty}\theta^{\Lambda}_{n}>0$.
	\\\textnormal{(b)} If \textnormal{(I)} and \textnormal{(IV)} are satisfied and $\theta^{\Lambda}_{b}=\lim\limits_{n\to\infty}\theta^{\Lambda}_{n}$ exists, then $\theta^{\Lambda}_{X}\in(0,1]$ exists and $\theta^{\Lambda}_{X}=\theta^{\Lambda}_{b}$.
\end{theorem}
\begin{remark}
	Notice that when $u_{n}=a_{n}x$ then (\ref{AC-u-I}) is the (AC$^{\Lambda}_{\succ,I^*}$) condition.
\end{remark}
\begin{proof}
	Denote by $t_{|\Lambda_{r_{n}}|}$ the highest element of $\Lambda_{r_{n}}$ according to $\prec$, by $t_{|\Lambda_{r_{n}}|-1}$ the second highest one, ..., by $t_{1}$ the lowest one. Consider the $s$, $\tilde{s}$, and $\hat{s}$ introduced in the proof of Theorem \ref{t1-L-E-Pro}. Let $l\in\mathbb{N}$. For $m=1,...,\hat{u}$ let $\mathfrak{M}_{m}:=\max_{\mathbf{t}\in \cup_{i=m}^{\hat{u}}(\hat{\mathcal{E}}_{i})_{\hat{s}_{i}}}X_\mathbf{t}$ and for $m=|\Lambda_{r_{n}}|+1$ let $\mathfrak{M}_{m}:=0$. Thus, we have $\mathbb{P}(\mathfrak{M}_{|\Lambda_{r_{n}}|+1}>u_{n})=0$ and
	\begin{equation*}
	\mathbb{P}(\max_{\mathbf{t}\in\Lambda_{r_{n}} }X_\mathbf{t}> u_{n})=\sum_{m=2}^{\hat{u}+1}-\mathbb{P}(\mathfrak{M}_{m}>u_{n})+\mathbb{P}(\mathfrak{M}_{m-1}>u_{n}).
	\end{equation*}
	For $m=2,...,\hat{u}$ let $\mathfrak{M}^{\circ}_{m}:=\max_{\mathbf{t}\in \cup_{i=m}^{\hat{u}}(\hat{\mathcal{E}}_{i})_{\hat{s}_{i}-\hat{s}_{m-1}}}X_\mathbf{t}$ and $\mathfrak{M}^{\circ}_{m\setminus l}:=\max_{\mathbf{t}\in \cup_{i=m}^{\hat{u}}(\hat{\mathcal{E}}_{i})_{\hat{s}_{i}-\hat{s}_{m-1}}\setminus K_{l}}X_\mathbf{t}$.
	We have
	\begin{multline}
	-\mathbb{P}(\mathfrak{M}_{m}>u_{n})+\mathbb{P}(\mathfrak{M}_{m-1}>u_{n})=-\mathbb{P}(\mathfrak{M}^{\circ}_{m}>u_{n})+\mathbb{P}(\mathfrak{M}^{\circ}_{m}\vee \max_{\mathbf{t}\in\hat{\mathcal{E}}_{m-1}}X_\mathbf{t}>u_{n})\label{Mm-new}\\
	=\mathbb{P}(\mathfrak{M}^{\circ}_{m}\leq u_{n},\max_{\mathbf{t}\in\hat{\mathcal{E}}_{m-1}}X_\mathbf{t}>u_{n})
	\end{multline}
	Notice that for $m=|\Lambda_{r_{n}}|+1$ we have that
	\begin{equation*}
	-\mathbb{P}(\mathfrak{M}_{m}>u_{n})+\mathbb{P}(\mathfrak{M}_{m-1}>u_{n})=\mathbb{P}(\max_{\mathbf{t}\in\hat{\mathcal{E}}_{\hat{u}}}X_\mathbf{t}>u_{n})\leq \mathbb{P}(\max_{\mathbf{t}\in K_{l},\mathbf{t}\succeq\mathbf{0}}X_\mathbf{t}>u_{n})
	\end{equation*}
	and so when divided by $|\Lambda_{r_n}|\mathbb{P}(X>u_{n})$ is asymptotically negligible, for every fixed $l\in\mathbb{N}$.
	
	Moreover, we have that (\ref{Mm-new}) is equal to
	\begin{align*}
	\lefteqn{\mathbb{P}(\mathfrak{M}^{\circ}_{m}\leq u_{n},\max_{\mathbf{t}\in\hat{\mathcal{E}}_{m-1}}X_\mathbf{t}>u_{n},\max_{\mathbf{t}\in \cup_{i=m}^{\hat{u}}(\hat{\mathcal{E}}_{i})_{\hat{s}_{i}-\hat{s}_{m-1}}\cap K_{2l}}X_\mathbf{t}\leq u_{n})}\\
	&+\mathbb{P}(\mathfrak{M}^{\circ}_{m}\leq u_{n},\max_{\mathbf{t}\in\hat{\mathcal{E}}_{m-1}}X_\mathbf{t}>u_{n},\max_{\mathbf{t}\in \cup_{i=m}^{\hat{u}}(\hat{\mathcal{E}}_{i})_{\hat{s}_{i}-\hat{s}_{m-1}}\cap K_{2l}}X_\mathbf{t}> u_{n})\\
	=&\mathbb{P}(\mathfrak{M}^{\circ}_{m}\leq u_{n},\max_{\mathbf{t}\in\hat{\mathcal{E}}_{m-1}}X_\mathbf{t}>u_{n},\max_{\mathbf{t}\in \cup_{i=m}^{\hat{u}}(\hat{\mathcal{E}}_{i})_{\hat{s}_{i}-\hat{s}_{m-1}}\cap K_{2l}}X_\mathbf{t}\leq u_{n})\\
	=&\mathbb{P}(\mathfrak{M}^{\circ}_{m\setminus 2l}\leq u_{n},\max_{\mathbf{t}\in\hat{\mathcal{E}}_{m-1}}X_\mathbf{t}>u_{n},\max_{\mathbf{t}\in \cup_{i=m}^{\hat{u}}(\hat{\mathcal{E}}_{i})_{\hat{s}_{i}-\hat{s}_{m-1}}\cap K_{2l}}X_\mathbf{t}\leq u_{n})\\
	=&\mathbb{P}(\max_{\mathbf{t}\in\hat{\mathcal{E}}_{m-1}}X_\mathbf{t}>u_{n},\max_{\mathbf{t}\in \cup_{i=m}^{\hat{u}}(\hat{\mathcal{E}}_{i})_{\hat{s}_{i}-\hat{s}_{m-1}}\cap K_{2l}}X_\mathbf{t}\leq u_{n})\\
	&-\mathbb{P}(\mathfrak{M}^{\circ}_{m\setminus 2l}> u_{n},\max_{\mathbf{t}\in\hat{\mathcal{E}}_{m-1}}X_\mathbf{t}>u_{n},\max_{\mathbf{t}\in \cup_{i=m}^{\hat{u}}(\hat{\mathcal{E}}_{i})_{\hat{s}_{i}-\hat{s}_{m-1}}\cap K_{2l}}X_\mathbf{t}\leq u_{n})
	\end{align*}
	where
	\begin{multline*}
	\mathbb{P}(\mathfrak{M}^{\circ}_{m\setminus 2l}> u_{n},\max_{\mathbf{t}\in\hat{\mathcal{E}}_{m-1}}X_\mathbf{t}>u_{n},\max_{\mathbf{t}\in \cup_{i=m}^{\hat{u}}(\hat{\mathcal{E}}_{i})_{\hat{s}_{i}-\hat{s}_{m-1}}\cap K_{2l}}X_\mathbf{t}\leq u_{n})\\
	\leq \mathbb{P}(\mathfrak{M}^{\circ}_{m\setminus 2l}> u_{n},\max_{\mathbf{t}\in\hat{\mathcal{E}}_{m-1}}X_\mathbf{t}>u_{n})\leq \mathbb{P}(\hat{M}_{2l,r_{n}}^{\Lambda,X,(j)}> u_{n},\max_{\mathbf{t}\in\mathcal{E}_{j}}X_\mathbf{t}>u_{n}),
	\end{multline*}
	for some $j\in I^*$ with $j<m_{4l}$. Hence, we have that
	\begin{multline*}
	\theta^{\Lambda}_{n}=\frac{\mathbb{P}(\max_{\mathbf{t}\in\Lambda_{r_{n}} }X_\mathbf{t}>u_{n})}{|\Lambda_{r_n}|\mathbb{P}(X>u_{n})}\\
	= \sum_{i\in u}\mathbb{P}(\max_{\mathbf{t}\in \tilde{\mathcal{D}}_{j_i}\cap K_{2l} }X_\mathbf{t}\leq u_{n}|\max_{\mathbf{t}\in\mathcal{E}_{j_i}}X_\mathbf{t}>u_{n})\frac{\mathbb{P}(\max_{\mathbf{t}\in\mathcal{E}_{j_i}}X_\mathbf{t}>u_{n})}{|\Lambda_{r_n}|\mathbb{P}(X>u_{n})}+B_{l,n}
	\end{multline*}
	where the absolute value of $B_{l,n}$ is such that
	\begin{multline*}
	|B_{l,n}|\leq \frac{\sum_{h\in I^*,h<m_{4l}}|\bar{S}_{h,4l}||\mathcal{E}_{h}|+\tilde{u}}{|\Lambda_{r_n}|}\\
	+ \sum_{i\in I^*,i<m_{4l}}\mathbb{P}(\hat{M}_{l,r_{n}}^{\Lambda,X,(j_i)}> u_{n}|\max_{\mathbf{t}\in\mathcal{E}_{j_i}}X_\mathbf{t}>u_{n})\frac{|S'_{j_i,4l}|\mathbb{P}(\max_{\mathbf{t}\in\mathcal{E}_{j_i}}X_\mathbf{t}>u_{n})}{|\Lambda_{r_n}|\mathbb{P}(X>u_{n})}.
	\end{multline*}	
	
	By (\ref{AC-u-I}) and by the same arguments as the ones used in the proof of Theorem \ref{t1-L-E-Pro}, see in particular (\ref{bound-J}), (\ref{m-1}), and (\ref{u-tilde}), we obtain that $\lim\limits_{l\to\infty}\limsup\limits_{n\to\infty}|B_{l,n}|=0$. Moreover, since
	\begin{multline*}
	\sum_{i\in u}\mathbb{P}(\max_{\mathbf{t}\in \tilde{\mathcal{D}}_{j_i}\cap K_{2l} }X_\mathbf{t}\leq u_{n}|\max_{\mathbf{t}\in\mathcal{E}_{j_i}}X_\mathbf{t}>u_{n})\frac{\mathbb{P}(\max_{\mathbf{t}\in\mathcal{E}_{j_i}}X_\mathbf{t}>u_{n})}{|\Lambda_{r_n}|\mathbb{P}(X>u_{n})}\\
	=\sum_{h\in I^*,h<m_{4l}}|S'_{h,4l}|\mathbb{P}(\max_{\mathbf{t}\in \tilde{\mathcal{D}}_{h}\cap K_{2l} }X_\mathbf{t}\leq u_{n}|\max_{\mathbf{t}\in\mathcal{E}_{h}}X_\mathbf{t}>u_{n})\frac{\mathbb{P}(\max_{\mathbf{t}\in\mathcal{E}_{h}}X_\mathbf{t}>u_{n})}{|\Lambda_{r_n}|\mathbb{P}(X>u_{n})}.
	\end{multline*}
	we obtain (\ref{I-new}).
	
	To show that $\liminf\limits_{n\to\infty}\theta^{\Lambda}_{n}>0$ we proceed as follows. Consider $S'_{h,4l}$, for some $h\in I^*$ with $h<m_{4l}$. Let $W^{(h)}_n$ be the set of points in $S'_{h,4l}$ that have supremum distance of $4l$ from each other. Observe that the sets \begin{equation*}
	    \left\{\max_{\mathbf{t}\in\mathcal{E}_{h}}X_{\mathbf{t}+s_{m}}>u_{n},\max_{\mathbf{t}\succ s_m,\mathbf{t}\in \cup_{i=1}^{\hat{u}}(\hat{\mathcal{E}}_{i})_{\hat{s}_{i}}\setminus K_{2l}}X_\mathbf{t}\leq u_{n} \right\},\quad s_m\in W^{(h)}_n,
	\end{equation*}
	are disjoint and their union is a subset of $\{\max_{\mathbf{t}\in\Lambda_{r_{n}} }X_{\mathbf{t}}> u_{n}\}$. This is because 
	\begin{equation*}
	\bigg\{\max_{\mathbf{t}\in K_{2l}\cap\Lambda_{r_{n}}}X_{\mathbf{t}+s_{m}}>u_{n},\max_{\mathbf{t}\succ s_m,\mathbf{t}\in \cup_{i=1}^{\hat{u}}(\hat{\mathcal{E}}_{i})_{\hat{s}_{i}}\setminus K_{2l}}X_\mathbf{t}\leq u_{n} \bigg\},\quad s_m\in W^{(h)}_n,
	\end{equation*}
	are disjoint and their union is a subset of $\{\max_{\mathbf{t}\in\Lambda_{r_{n}} }X_{\mathbf{t}}> u_{n}\}$ and because $\left\{\max_{\mathbf{t}\in K_{2l}\cap\Lambda_{r_{n}}}X_{\mathbf{t}+s_{m}}>u_{n}\right\}\supset\left\{ \max_{\mathbf{t}\in\mathcal{E}_{h}}X_{\mathbf{t}+s_{m}}>u_{n}\right\}$. Observe that $|W^{(h)}_n|\geq \frac{|S'_{h,4l}|}{(4l)^{k}}$. Then, we have
	\begin{align*}
	\theta^{\Lambda}_{n}&=\frac{\mathbb{P}(\max_{\mathbf{t}\in\Lambda_{r_{n}} }X_{\mathbf{t}}>u_{n})}{|\Lambda_{r_{n}}|\mathbb{P}(X>u_{n})}\geq \sum_{s_m\in W^{(h)}_n}\frac{\mathbb{P}(\max_{\mathbf{t}\in\mathcal{E}_{h}}X_{\mathbf{t}+s_{m}}>u_{n},\max_{\mathbf{t}\succ s_m,\mathbf{t}\in \cup_{i=1}^{\hat{u}}(\hat{\mathcal{E}}_{i})_{\hat{s}_{i}}\setminus K_{2l}}X_\mathbf{t}\leq u_{n})}{|\Lambda_{r_{n}}|\mathbb{P}(X>u_{n})}\\
&	=\sum_{s_m\in W^{(h)}_n}\frac{\mathbb{P}(\max_{\mathbf{t}\in\mathcal{E}_{h}}X_\mathbf{t}>u_{n},\max_{\mathbf{t}\succ 0,\mathbf{t}\in \cup_{i=1}^{\hat{u}}(\hat{\mathcal{E}}_{i})_{\hat{s}_{i}-s_m}\setminus K_{2l}}X_\mathbf{t}\leq u_{n})}{|\Lambda_{r_{n}}|\mathbb{P}(X>u_{n})}\\
	&\geq |W^{(h)}_n|\frac{\mathbb{P}(\max_{\mathbf{t}\in\mathcal{E}_{h}}X_\mathbf{t}>u_{n},\hat{M}^{\Lambda,X,(h)}_{2l,r_{n}}\leq u_{n})}{|\Lambda_{r_{n}}|\mathbb{P}(X>u_{n})}\\
&	=\Big(1-\mathbb{P}(\max_{\mathbf{t}\in\mathcal{E}_{h}}X_\mathbf{t}>u_{n}|\hat{M}^{\Lambda,X,(h)}_{2l,r_{n}}\leq u_{n})\Big)\frac{|W^{(h)}_n|}{|\Lambda_{r_{n}}|}\frac{\mathbb{P}(\max_{\mathbf{t}\in\mathcal{E}_{h}}X_\mathbf{t}>u_{n})}{\mathbb{P}(X>u_{n})}.
	\end{align*} 
	Then
	\begin{equation*}
	\liminf\limits_{n\to\infty}\theta_{n}\geq\frac{\gamma^*_h c_h}{(4l)^{k}}[1-\limsup\limits_{n\to\infty}\mathbb{P}(\hat{M}^{\Lambda,X,(h)}_{2l,r_{n}}> u_{n}|X_{\mathbf{0}}>u_{n})]>0,
	\end{equation*}
	for some $l$ large enough. This proves point (a). The proof of point (b) follows from the same arguments as the ones used for the proof of Theorem \ref{thm-u} point (b).
\end{proof}	

\begin{proof}[Proof of Theorem \ref{co-u-I}]
	Recall (\ref{I-new}) and let 
	\begin{equation*}
	\theta^{(l)}:=\lim\limits_{n\to\infty}\sum_{h\in I^*,h<m_{4l}}\mathbb{P}(\max_{\mathbf{t}\in \tilde{\mathcal{D}}_{h}\cap K_{2l} }X_\mathbf{t}\leq u_{n}|\max_{\mathbf{t}\in\mathcal{E}_{h}}X_\mathbf{t}>u_{n})\frac{|S'_{h,4l}|}{|\Lambda_{r_n}|}\frac{\mathbb{P}(\max_{\mathbf{t}\in\mathcal{E}_{h}}X_\mathbf{t}>u_{n})}{\mathbb{P}(X>u_{n})},
	\end{equation*}
	for $l\in\mathbb{N}$. Using the arguments in the proof of Theorem \ref{t1-L-E-Pro}, we have that
	\begin{align*}
	\theta^{(l)}&=\lim\limits_{n\to\infty}\sum_{h\in I^*,h<m_{4l}}\gamma^*_h\frac{\mathbb{P}(\max_{\mathbf{t}\in \tilde{\mathcal{D}}_{h}\cap K_{2l} }X_\mathbf{t}\leq u_{n},\max_{\mathbf{t}\in\mathcal{E}_{h}}X_\mathbf{t}>u_{n},\rho_{\mathcal{E}_{j_i}}(X)>\frac{u_{n}}{D_{j_i}})}{\mathbb{P}(X>u_{n})}\\
&	=\sum_{h\in I^*,h<m_{4l}}\gamma^*_{h}c_{h}D_{h}^{\alpha}\mathbb{P}\Big(\max_{\mathbf{t}\in\mathcal{E}_{h}}Y_{\mathcal{E}_{h},\mathbf{t}}>D_{h},\max_{\mathbf{t}\in \tilde{\mathcal{D}}_{h}\cap K_{2l} }Y_{\mathcal{E}_{h},\mathbf{t}}\leq D_{h}\Big)\\
&	=\sum_{h\in I^*,h<m_{4l}}\int_{1}^{\infty}\gamma^*_{h}c_{h}D_{h}^{\alpha}\mathbb{P}\Big(y\max_{\mathbf{t}\in\mathcal{E}_{h}}\Theta_{\mathcal{E}_{h},\mathbf{t}}>D_{h},y\max_{\mathbf{t}\in \tilde{\mathcal{D}}_{h}\cap K_{2l} }\Theta_{\mathcal{E}_{h},\mathbf{t}}\leq D_{h}\Big)d(-y^{-\alpha}),
	\end{align*}	
	and
	\begin{align*}
	\lefteqn{\theta^{\Lambda}_{b}=\lim\limits_{l\to\infty}\theta^{(l)}=\lim\limits_{l\to\infty}\sum_{h\in I^*,h<m_{4l}}\int_{1}^{\infty}\gamma^*_{h}c_{h}D_{h}^{\alpha}\mathbb{P}\Big(y\max_{\mathbf{t}\in\mathcal{E}_{h}}\Theta_{\mathcal{E}_{h},\mathbf{t}}>D_{h},y\max_{\mathbf{t}\in \tilde{\mathcal{D}}_{h}\cap K_{2l} }\Theta_{\mathcal{E}_{h},\mathbf{t}}\leq D_{h}\Big)d(-y^{-\alpha})}\\
&	=\sum_{h\in I^*}\int_{1}^{\infty}\gamma^*_{h}c_{h}D_{h}^{\alpha}\mathbb{P}\Big(y\max_{\mathbf{t}\in\mathcal{E}_{h}}\Theta_{\mathcal{E}_{h},\mathbf{t}}>D_{h},y\sup_{\mathbf{t}\in \tilde{\mathcal{D}}_{h} }\Theta_{\mathcal{E}_{h},\mathbf{t}}\leq D_{h}\Big)d(-y^{-\alpha})\\
&	=\sum_{h\in I^*}\int_{0}^{\infty}\gamma^*_{h}c_{h}D_{h}^{\alpha}\mathbb{P}\Big(y\max_{\mathbf{t}\in\mathcal{E}_{h}}\Theta_{\mathcal{E}_{h},\mathbf{t}}>D_{h},y\sup_{\mathbf{t}\in \tilde{\mathcal{D}}_{h} }\Theta_{\mathcal{E}_{h},\mathbf{t}}\leq D_{h}\Big)d(-y^{-\alpha}).
	\end{align*}
	Given that $\lim\limits_{l\to\infty}\theta^{(l)}=\theta^{\Lambda}_{b}$ exists, Theorem \ref{thm-u} point (a) ensures that for $(u_{n})=(u_{n}(\tau))$ and some $\tau>0$ we have that $\theta^{\Lambda}_{b}=\lim\limits_{n\to\infty}\theta^{\Lambda}_{n}$ exists and is positive. Thus, from Theorem \ref{thm-u} point (b) for $(u_{n})=(u_{n}(\tau))$ and arbitrary $\tau>0$ we obtain that $\theta^{\Lambda}_{X}$ exists, is positive and it is equal to $\theta^{\Lambda}_{b}$. From these arguments we obtain the first equality in (\ref{index4}) while for the others we have that we have that
	\begin{align*}
	\theta^{\Lambda}_{b}&=\sum_{i\in I^*}\gamma^*_{i}c_{i}D_{i}^{\alpha}\mathbb{P}\Big(Y\max_{\mathbf{t}\in\mathcal{E}_{i}}\Theta_{\mathcal{E}_{i},\mathbf{t}}>D_{i},Y\sup_{\mathbf{t}\in \tilde{\mathcal{D}}_{i} }\Theta_{\mathcal{E}_{i},\mathbf{t}}\leq D_{i}\Big)\\
	&=\sum_{i\in I^*}\gamma^*_{i}c_{i}D_{i}^{\alpha}\bigg(\mathbb{P}\Big(Y\max_{\mathbf{t}\in\mathcal{E}_{i}}\Theta_{\mathcal{E}_{i},\mathbf{t}}>D_{i}\Big)-\mathbb{P}\Big(Y\max_{\mathbf{t}\in\mathcal{E}_{i}}\Theta_{\mathcal{E}_{i},\mathbf{t}}>D_{i},Y\sup_{\mathbf{t}\in \tilde{\mathcal{D}}_{i} }\Theta_{\mathcal{E}_{i},\mathbf{t}}> D_{i}\Big)\bigg)\\
	&=\sum_{i\in I^*}\gamma^*_{i}c_{i}D_{i}^{\alpha}\int_{1}^{\infty}\mathbb{P}\Big(y\max_{\mathbf{t}\in\mathcal{E}_{i}}\Theta_{\mathcal{E}_{i},\mathbf{t}}>D_{i}\Big)-\mathbb{P}\Big(y\max_{\mathbf{t}\in\mathcal{E}_{i}}\Theta_{\mathcal{E}_{i},\mathbf{t}}>D_{i},y\sup_{\mathbf{t}\in \tilde{\mathcal{D}}_{i} }\Theta_{\mathcal{E}_{i},\mathbf{t}}> D_{i}\Big)d(-y^{-\alpha})\\
	&\stackrel{u=y^{-\alpha}}{=}\sum_{i\in I^*}\gamma^*_{i}c_{i}D_{i}^{\alpha}\int_{0}^{1}\mathbb{P}\Big(\max_{\mathbf{t}\in\mathcal{E}_{i}}\Theta^{\alpha}_{\mathcal{E}_{i},\mathbf{t}}>uD^{\alpha}_{i}\Big)-\mathbb{P}\Big(\max_{\mathbf{t}\in\mathcal{E}_{i}}\Theta^{\alpha}_{\mathcal{E}_{i},\mathbf{t}}>uD^{\alpha}_{i},\sup_{\mathbf{t}\in \tilde{\mathcal{D}}_{i} }\Theta^{\alpha}_{\mathcal{E}_{i},\mathbf{t}}> uD_{i}^{\alpha}\Big)du\\
	&=\sum_{i\in I^*}\gamma^*_{i}c_{i}D_{i}^{\alpha}\bigg(\mathbb{E}\Big[D^{-\alpha}_{i}\max_{\mathbf{t}\in\mathcal{E}_{i}}\Theta^{\alpha}_{\mathcal{E}_{i},\mathbf{t}}\wedge 1\Big]-\mathbb{E}\Big[D_{i}^{-\alpha}\max_{\mathbf{t}\in\mathcal{E}_{i}}\Theta^{\alpha}_{\mathcal{E}_{i},\mathbf{t}}\wedge D_{i}^{-\alpha}\sup_{\mathbf{t}\in \tilde{\mathcal{D}}_{i} }\Theta^{\alpha}_{\mathcal{E}_{i},\mathbf{t}}\wedge 1\Big]\bigg)\\
	&=\sum_{i\in I^*}\gamma^*_{i}c_{i}\bigg(\mathbb{E}\Big[\max_{\mathbf{t}\in\mathcal{E}_{i}}\Theta^{\alpha}_{\mathcal{E}_{i},\mathbf{t}}\wedge D_{i}^{\alpha}\Big]-\mathbb{E}\Big[\max_{\mathbf{t}\in\mathcal{E}_{i}}\Theta^{\alpha}_{\mathcal{E}_{i},\mathbf{t}}\wedge \sup_{\mathbf{t}\in \tilde{\mathcal{D}}_{i} }\Theta^{\alpha}_{\mathcal{E}_{i},\mathbf{t}}\wedge D_{i}^{\alpha}\Big]\bigg)\\
	&=\sum_{i\in I^*}\gamma^*_{i}c_{i}\bigg(\mathbb{E}\Big[\max_{\mathbf{t}\in\mathcal{E}_{i}}\Theta^{\alpha}_{\mathcal{E}_{i},\mathbf{t}}\Big]-\mathbb{E}\Big[\max_{\mathbf{t}\in\mathcal{E}_{i}}\Theta^{\alpha}_{\mathcal{E}_{i},\mathbf{t}}\wedge \sup_{\mathbf{t}\in \tilde{\mathcal{D}}_{i} }\Theta^{\alpha}_{\mathcal{E}_{i},\mathbf{t}}\Big]\bigg)\\
&	=\sum_{i\in I^*}\gamma^*_{i}c_{i}\mathbb{E}\Big[\Big(\max_{\mathbf{t}\in\mathcal{E}_{i}}\Theta^{\alpha}_{\mathcal{E}_{i},\mathbf{t}}- \sup_{\mathbf{t}\in \tilde{\mathcal{D}}_{i} }\Theta^{\alpha}_{\mathcal{E}_{i},\mathbf{t}}\Big)_{+}\Big]=\sum_{i\in I^*}\gamma^*_{i}c_{i}\mathbb{E}\Big[\sup_{\mathbf{t}\in\mathcal{E}_{i}\cup\tilde{\mathcal{D}}_{i}}\Theta^{\alpha}_{\mathcal{E}_{i},\mathbf{t}}-\sup_{\mathbf{t}\in \tilde{\mathcal{D}}_{i} }\Theta^{\alpha}_{\mathcal{E}_{i},\mathbf{t}}\Big],
	\end{align*}
	where we used the fact that since $\rho_{\mathcal{E}_{i}}(\Theta)=1$ a.s., then $\max_{\mathbf{s}\in\mathcal{E}_{i}}\Theta(\mathbf{s}\leq D_{i}$ a.s.~by Corollary \ref{co-D-C-norm}. Moreover, applying a change of variable and the time change formula we get the representation
	\begin{equation*}
	\sum_{j\in I^*}\gamma^*_{j}c_{j}\mathbb{E}\bigg[\frac{\sup_{\mathbf{t}\in\mathcal{H}_{j}}\Theta_{\mathcal{E}_{j},\mathbf{t}}^{\alpha}}{\sum_{\mathbf{s}\in\mathcal{H}_{j}}|\Theta_{\mathcal{E}_{j},\mathbf{s}}|^{\alpha}} \bigg]=\sum_{j\in I^*}\gamma^*_{j}c_{j}\mathbb{E}\bigg[\sup\limits_{\mathbf{t}\in\mathcal{H}_{j}}{Q}_{\mathcal{E}_{j},\mathcal{L}_{j},\mathbf{t}}^{\alpha}\bigg].
	\end{equation*}
	Finally, by the time-change formula applied to 
	\begin{equation*}
	    f_{t}((\Theta_{\mathcal{E}_{j},\mathbf{s}})_{\mathbf{s}\in\mathcal{H}_{j}})=\frac{\max_{\mathbf{z}\in(\mathcal{E}_{j})_{\mathbf{t}}}\Theta_{\mathcal{E}_{j},\mathbf{z}}^{\alpha}}{\sum_{\mathbf{s}\in\mathcal{H}_{j}}\Theta_{\mathcal{E}_{j},\mathbf{s}}^{\alpha}}\mathbf{1}(\mathbf T^{*}_{j}=\mathbf{t})
	\end{equation*} 
	and shifting $\mathbf{t}$ to $\mathbf{0}$, for every $\mathbf{t}\in\mathcal{L}_{j}$, we have
	\begin{multline*}
	\mathbb{E}\bigg[\sup\limits_{\mathbf{t}\in(\mathcal{E}_{j})_{T^{*}_{j}}} {Q}_{\mathcal{E}_{j},\mathcal{L}_{j},\mathbf{t}}^{\alpha}\bigg]=\sum_{\mathbf{t}\in\mathcal{L}_{j}}\mathbb{E}\bigg[\frac{\max_{\mathbf{z}\in(\mathcal{E}_{j})_{\mathbf{t}}}\Theta_{\mathcal{E}_{j},\mathbf{z}}^{\alpha}}{\sum_{\mathbf{s}\in\mathcal{H}_{j}}\Theta_{\mathcal{E}_{j},\mathbf{s}}^{\alpha}} \mathbf{1}(\mathbf T^{*}_{j}=\mathbf{t})\bigg]\\
	=\sum_{\mathbf{t}\in\mathcal{L}_{j}}\mathbb{E}\bigg[\frac{\max_{\mathbf{z}\in\mathcal{E}_{j}}\Theta_{\mathcal{E}_{j},\mathbf{z}}^{\alpha}}{\sum_{\mathbf{s}\in\mathcal{H}_{j}}\Theta_{\mathcal{E}_{j},\mathbf{s}}^{\alpha}} \mathbf{1}(\mathbf T^{*}_{j}=\mathbf{0})\sum_{\mathbf{i}\in(\mathcal{E}_{j})_{-\mathbf{t}}}\Theta_{\mathcal{E}_{j},\mathbf{i}}^{\alpha}\bigg]=\mathbb{E}\bigg[\max_{\mathbf{z}\in\mathcal{E}_{j}}\Theta_{\mathcal{E}_{j},\mathbf{z}}^{\alpha} \mathbf{1}(\mathbf T^{*}_{j}=\mathbf{0})\bigg].
	\end{multline*}
\end{proof}

\subsection{Proof of Proposition \ref{pro-BR-1}}
 For every $j\ge 1$ we also denote  $R^{(j)}_{l,\Lambda_{n}}:=\big(\bigcup_{\mathbf{t}\in\{\mathbf{s}\in\Lambda_{n}:(\Lambda_{n})_{-\mathbf{s}}\supset\mathcal{E}_j\}}((\Lambda_{n})_{-\mathbf{t}}\cap\{\mathbf{s}\in\mathbb{Z}^{k}:\mathbf{s}\succ\mathbf{0}\})\big)\setminus K_{l}$ and   $\hat{M}^{\Lambda,X,(j)}_{l,n}:=\max_{\mathbf{i}\in R^{(j)}_{l,\Lambda_{n}}}{X}_{\mathbf{i}}$ so that Condition (AC$^{\Lambda}_{\succeq,I^*}$) is satisfied  if for every $j\in I^*$
\begin{equation*}
\lim\limits_{l\to\infty}\limsup_{n\to\infty}\mathbb{P}\Big(\hat{M}^{\Lambda,X,(j)}_{2l,r_{n}}>a_{n}^\Lambda x\,\big|\max\limits_{\mathbf{t}\in\mathcal{E}_j} {X}_{\mathbf{t}}>a_{n}^\Lambda x\Big)=0.
\end{equation*}
We start by showing the following Lemma
\begin{lemma}
	Let $(X_\textbf{t})_{\textbf{t}\in\mathbb{Z}^{k}}$ be a stationary  max-stable random field. Then $(X_{\textbf{t}})_{\textbf{t}\in\mathbb{Z}^{k}}$ is jointly regularly varying and the finite-dimensional distributions of its tail field $(Y_{\textbf{t}})_{\textbf{t}\in\mathbb{Z}^{k}}$ is given by
	\begin{multline}
	\mathbb{P}(Y_{\textbf{t}_{1}}<y_{1},...,Y_{\textbf{t}_{n}}<y_{n})\label{BR-distribution}\\
	=\frac{1}{\mathbb{E}[V_{\textbf{0}}]}\left\{\mathbb{E}\left[\max\left(\max_{i=1,...,n}\frac{1}{y_{i}}V_{\textbf{t}_{i}},V_{\textbf{0}}\right)\right]-\mathbb{E}\left[\max_{i=1,...,n}\frac{1}{y_{i}}V_{\textbf{t}_{i}}\right]\right\}
	\end{multline}
	for $\textbf{t}_{1},...,\textbf{t}_{n}\in\mathbb{Z}^{k}$ and $y_{1},...,y_{n}\in(0,\infty)$.
\end{lemma}
\begin{proof}
	It follows from similar computations as the ones in the proof of Proposition 6.1 in \cite{SW}. For any $x_{1},...,x_{n}\in(0,\infty)$ we have (see Examples 1.5.4 and 4.4.3 in \cite{MW})
	\begin{equation*}
	\mathbb{P}(X_{\textbf{t}_{1}}\leq x_{1},...,X_{\textbf{t}_{n}}\leq x_{n})=\exp\left\{-\mathbb{E}\left[\max_{i=1,...,n}\frac{V_{\textbf{t}_{i}}}{x_{i}}\right] \right\}.
	\end{equation*}
	Thus, for any $x>0$
	\begin{align*}
\lefteqn{	\mathbb{P}(x^{-1}X_{\textbf{t}_{1}}\leq y_{1},...,x^{-1}X_{\textbf{t}_{n}}\leq y_{n}\,\,|\,\,X_{\textbf{0}}>x)}\\
&	=\frac{\mathbb{P}(X_{\textbf{t}_{1}}\leq xy_{1},...,X_{\textbf{t}_{n}}\leq xy_{n})-\mathbb{P}(X_{\textbf{t}_{1}}\leq xy_{1},...,X_{\textbf{t}_{n}}\leq xy_{n},X_{\textbf{0}}\leq x)}{\mathbb{P}(X_{\textbf{0}}>x)}\\
&	=\frac{\exp\left\{-\frac{1}{x}\mathbb{E}\left[\max_{i=1,...,n}\frac{V_{\textbf{t}_{i}}}{y_{i}}\right] \right\}-\exp\left\{-\frac{1}{x}\mathbb{E}\left[\max\left(\max_{i=1,...,n}\frac{V_{\textbf{t}_{i}}}{y_{i}},V_{\textbf{0}}\right)\right] \right\}}{1-e^{-\mathbb{E}[V_{\textbf{0}}]/x}}\\
&	\sim \frac{x}{\mathbb{E}[V_{\textbf{0}}]}\left(\exp\left\{-\frac{1}{x}\mathbb{E}\left[\max_{i=1,...,n}\frac{V_{\textbf{t}_{i}}}{y_{i}}\right] \right\}-\exp\left\{-\frac{1}{x}\mathbb{E}\left[\max\left(\max_{i=1,...,n}\frac{V_{\textbf{t}_{i}}}{y_{i}},V_{\textbf{0}}\right)\right] \right\}  \right)\\
&	= \frac{x}{\mathbb{E}[V_{\textbf{0}}]}\bigg(\exp\left\{-\frac{\mathbb{E}[V_{\textbf{0}}]}{x}\frac{1}{\mathbb{E}[V_{\textbf{0}}]}\mathbb{E}\left[\max_{i=1,...,n}\frac{V_{\textbf{t}_{i}}}{y_{i}}\right] \right\}\\
&	-\exp\left\{-\frac{\mathbb{E}[V_{\textbf{0}}]}{x}\frac{1}{\mathbb{E}[V_{\textbf{0}}]}\mathbb{E}\left[\max\left(\max_{i=1,...,n}\frac{V_{\textbf{t}_{i}}}{y_{i}},V_{\textbf{0}}\right)\right] \right\}  \bigg)
	\end{align*}
	which converges to (\ref{BR-distribution}) as $x\to\infty$.
\end{proof}

We follow partially the proof of Proposition 6.2 in \cite{SW}. Fix any $j\in I^*$, observe that
	\begin{multline*}
	\mathbb{P}\Big(\hat{M}^{\Lambda,X,(j)}_{2l,r_{n}}>a_{n}^\Lambda x\,\big|\max\limits_{\mathbf{t}\in\mathcal{E}_j} X_{\mathbf{t}}>a_{n}^\Lambda x\Big)
	=1-\\
	\frac{\mathbb{P}\Big(\max\Big(\hat{M}^{\Lambda,X,(j)}_{2l,r_{n}},\max\limits_{\mathbf{t}\in\mathcal{E}_j} X_{\mathbf{t}}\Big)>a_{n}^\Lambda x\Big)-\mathbb{P}\Big(\hat{M}^{\Lambda,X,(j)}_{2l,r_{n}}>a_{n}^\Lambda x\Big)}{\mathbb{P}\Big(\max\limits_{\mathbf{t}\in\mathcal{E}_j} X_{\mathbf{t}}>a_{n}^\Lambda x\Big)}
	\end{multline*}
	and that
	\begin{equation*}
	\mathbb{P}\Big(\max\limits_{\mathbf{t}\in\mathcal{E}_j} X_{\mathbf{t}}>a_{n}^\Lambda x\Big)=\Big(1-e^{-\mathbb{E}\big[\max\limits_{\mathbf{t}\in\mathcal{E}_j} V_{\mathbf{t}}\big]/a_{n}^\Lambda x}\Big)\sim \mathbb{E}\Big[\max\limits_{\mathbf{t}\in\mathcal{E}_j} V_{\mathbf{t}}\Big]/a_{n}^\Lambda  x\,,\qquad n\to \infty\,.
	\end{equation*}
	Since
	\begin{equation*}
	0\leq -(a_{n}^\Lambda x)^{-1}\mathbb{E}\Big[\hat{M}^{\Lambda,V,(j)}_{2l,r_{n}}\Big]\leq -(a_{n}^\Lambda x)^{-1}|\Lambda_{r_n}|\to0,
	\end{equation*}
	as $n\to\infty$, we also have
		\begin{equation*}
	\mathbb{P}\Big(\hat{M}^{\Lambda,X,(j)}_{2l,r_{n}}>a_{n}^\Lambda x\Big)=\Big(1-e^{-\mathbb{E}\big[\hat{M}^{\Lambda,V,(j)}_{2l,r_{n}}\big]/a_{n}^\Lambda x}\Big)\sim \mathbb{E}\Big[\hat{M}^{\Lambda,V,(j)}_{2l,r_{n}}\Big]/a_{n}^\Lambda  x\,,\qquad n\to \infty\,.
	\end{equation*}
	Then, the (AC$_{\succ}$) condition is satisfied if and only if
	\begin{equation}\label{BR-1}
	\lim\limits_{l\to\infty}\liminf_{n\to\infty}a_{n}^\Lambda x\Big[\mathbb{P}\Big(\max\Big(\hat{M}^{\Lambda,X,(j)}_{2l,r_{n}},\max\limits_{\mathbf{t}\in\mathcal{E}_j} X_{\mathbf{t}}\Big)>a_{n}^\Lambda x\Big)-\mathbb{P}\Big(\hat{M}^{\Lambda,X,(j)}_{2l,r_{n}}>a_{n}^\Lambda x\Big)\Big]=\mathbb{E}[\max\limits_{\mathbf{t}\in\mathcal{E}_j} V_{\mathbf{t}}].
	\end{equation}
	Since
	\begin{align*}
&	\mathbb{P}\Big(\max\Big(\hat{M}^{\Lambda,X,j}_{2l,r_{n}},\max\limits_{\mathbf{t}\in\mathcal{E}_j} X_{\mathbf{t}}\Big)>a_{n}^\Lambda x\Big)-\mathbb{P}\Big(\hat{M}^{\Lambda,X,(j)}_{2l,r_{n}}>a_{n}^\Lambda x\Big)\\
&	=\exp\Big(-(a_{n}^\Lambda x)^{-1}\mathbb{E}\Big[\hat{M}^{\Lambda,V,(j)}_{2l,r_{n}}\Big]\Big)-\exp\Big(-(a_{n}^\Lambda x)^{-1}\mathbb{E}\Big[\max\Big(\hat{M}^{\Lambda,V,(j)}_{2l,r_{n}},\max\limits_{\mathbf{t}\in\mathcal{E}_j} V_{\mathbf{t}}\Big)\Big]\Big)\\
&	\sim (a_{n}^\Lambda x)^{-1}\Big(\mathbb{E}\Big[\max\Big(\hat{M}^{\Lambda,V,(j)}_{2l,r_{n}},\max\limits_{\mathbf{t}\in\mathcal{E}_j} V_{\mathbf{t}}\Big)\Big]-\mathbb{E}\Big[\hat{M}^{\Lambda,V,(j)}_{2l,r_{n}}\Big]\Big)
\end{align*}
	and then (\ref{BR-1}) holds if and only if
	\begin{equation}\label{BR-2}
	\lim\limits_{l\to\infty}\liminf_{n\to\infty}\mathbb{E}\Big[\max\Big(\hat{M}^{\Lambda,V,(j)}_{2l,r_{n}},\max\limits_{\mathbf{t}\in\mathcal{E}_j} V_{\mathbf{t}}\Big)\Big]-\mathbb{E}\Big[\hat{M}^{\Lambda,V,(j)}_{2l,r_{n}}\Big]=\mathbb{E}\Big[\max\limits_{\mathbf{t}\in\mathcal{E}_j} V_{\mathbf{t}}\Big].
	\end{equation}
	
	Then
	\begin{multline*}
	\lim\limits_{l\to\infty}\liminf_{n\to\infty}\mathbb{E}\Big[\max\Big(\hat{M}^{\Lambda,V,(j)}_{2l,r_{n}},\max\limits_{\mathbf{t}\in\mathcal{E}_j} V_{\mathbf{t}})\Big]-\mathbb{E}\Big[\hat{M}^{\Lambda,V,(j)}_{2l,r_{n}}\Big]\\
	=\lim\limits_{l\to\infty}\mathbb{E}\Big[\Big(\max_{\mathbf{t}\in R^{(j)}_{2l,\Lambda_{r_n}}\cup\mathcal{E}_j } V_{\textbf{t}} - \max\limits_{\mathbf{t}\in R^{(j)}_{2l,\Lambda_{r_n}}} V_{\textbf{t}}\Big)1\big(\max\limits_{\mathbf{t}\in\mathcal{E}_j} V_{\mathbf{t}}\ne 0\big)\Big]\,.
	\end{multline*}
	By assumption $\max\limits_{\mathbf{t}\in R^{(j)}_{2l,\Lambda_{r_n}}} V_{\textbf{t}}=\max\limits_{\mathbf{t}\in R^{(j)}_{2l,\Lambda_{r_n}}\cap (\cup_{i\ge 1}(({\cal H})_i)_+)} V_{\textbf{t}}$, thus $\max\limits_{\mathbf{t}\in R^{(j)}_{2l,\Lambda_{r_n}}} V_{\textbf{t}}1\big(\max\limits_{\mathbf{t}\in\mathcal{E}_j} V_{\mathbf{t}}\ne 0\big)\to 0$ a.s. as $l\to \infty$ under \eqref{BR}
and then \eqref{BR-2} follows by dominated convergence.

\section*{Acknowledgements}
Riccardo Passeggeri would like to acknowledge the support of the Fondation Sciences Mathématiques de Paris (FSMP). Olivier Wintenberger would like to acknowledge the support of the French Agence Nationale de la Recherche (ANR) under reference ANR20-CE40- 0025-01 (T-REX project). He also would like to thank Thomas Mikosch for useful discussions on the topic. 
\bibliographystyle{plain}

\end{quote}
\end{document}